\documentclass[11pt]{amsart}

\usepackage[text={420pt,650pt},centering]{geometry}

\usepackage{color}
\usepackage{esint,amssymb}
\usepackage{graphicx}
\usepackage{MnSymbol}
\usepackage{mathtools}
\usepackage{microtype}

\usepackage{bm}
\usepackage{scalerel} 
\usepackage{dsfont}
\usepackage{stmaryrd}
\usepackage[font={footnotesize}]{caption}
\usepackage{xfrac}
\usepackage{comment} 
\usepackage[colorlinks=true, pdfstartview=FitV, linkcolor=blue, citecolor=blue, urlcolor=blue,pagebackref=false]{hyperref}

\definecolor{darkgreen}{rgb}{0,0.5,0}
\definecolor{darkred}{rgb}{0.9,0.1,0.1}

\newcommand{\mn}[1]{{\color{blue}{#1}}}

\newtheorem{proposition}{Proposition}
\newtheorem{theorem}[proposition]{Theorem}
\newtheorem{lemma}[proposition]{Lemma}
\newtheorem{corollary}[proposition]{Corollary}

\newtheorem{assumption}[proposition]{Assumption}
\theoremstyle{definition}
\newtheorem{remark}[proposition]{Remark}

\newtheorem{definition}[proposition]{Definition}

\newtheorem{question}[proposition]{Question}

\newcommand{\cref}[1]{Corollary~\ref{c.#1}}

\numberwithin{equation}{section}
\numberwithin{proposition}{section}

\newcommand{\A}{\mathcal{A}}

\renewcommand{\le}{\leqslant}
\renewcommand{\ge}{\geqslant}
\renewcommand{\leq}{\leqslant}
\renewcommand{\geq}{\geqslant}

\newcommand{\N}{\mathbb{N}}
\newcommand{\R}{\mathbb{R}}

\newcommand{\Rd}{{\mathbb{R}^d}}

\newcommand{\ep}{\varepsilon}
\newcommand{\eps}{\varepsilon}

\renewcommand{\a}{\mathbf{a}}
\newcommand{\h}{\mathbf{h}}
\newcommand{\g}{\mathbf{g}}

\newcommand{\f}{\mathbf{f}}
\newcommand{\n}{\mathbf{n}}

\newcommand{\T}{\mathbb{T}}
\renewcommand{\subset}{\subseteq}

\renewcommand{\a}{\mathbf{a}}
\renewcommand{\b}{\mathbf{b}}

\renewcommand{\subset}{\subseteq}

\renewcommand{\fint}{\strokedint}

\newcommand{\Ll}{\left}
\newcommand{\Rr}{\right}

\DeclareMathOperator{\dist}{dist}

\DeclareMathOperator{\sign}{sign}
\DeclareMathOperator{\divg}{div}
\DeclareMathOperator{\supp}{supp}

\renewcommand{\bar}{\overline}

\renewcommand{\tilde}{\widetilde}
\newcommand{\td}{\widetilde}
\renewcommand{\div}{\divg}
\newcommand{\indc}{\mathds{1}}
\newcommand{\1}{\mathds{1}}

\newcommand{\al}{\alpha}

\newcommand{\ga}{\gamma}
\newcommand{\de}{\delta}
\newcommand{\si}{\sigma}

\newcommand{\hyp}{{\mathrm{hyp}}}
\newcommand{\kin}{{\mathrm{kin}}}

\newcommand{\la}{\langle}
\newcommand{\ra}{\rangle}

\newcommand{\omv}{\begin{pmatrix} 1 \\ v \end{pmatrix}}

\newcommand{\ltwog}{L^2_\gamma}

\newcommand{\Qvx}{Q_{v\cdot\nabla_x}^{\sfrac{1}{2}}}
\newcommand{\Qx}{Q_{\nabla_x}^{\sfrac{1}{3}}}
\newcommand{\norm}[1]{\left\| #1 \right\|}
\newcommand{\p}{\partial}

\newcommand{\Qvxt}{Q_{D_t}^{\sfrac{1}{2}}}
\newcommand{\Qvxteps}{Q_{D_t}^{\sfrac{1}{2},\varepsilon}}
\newcommand{\epsthird}{\varepsilon^{-\sfrac{1}{3}}}
\newcommand{\epshalf}{\varepsilon^{\sfrac{1}{2}}}

\begin{document}

\title[Variational methods for the kinetic Fokker-Planck equation]{Variational methods for the kinetic \\
Fokker-Planck equation}

\begin{abstract}
We develop a functional analytic approach to the study of the Kramers and kinetic Fokker-Planck equations which parallels the classical $H^1$ theory of uniformly elliptic equations. In particular, we identify a function space analogous to $H^1$ and develop a well-posedness theory for weak solutions in this space. In the case of a conservative force, we identify the weak solution as the minimizer of a uniformly convex functional. We prove new functional inequalities of Poincar\'e and H\"ormander type and combine them with basic energy estimates (analogous to the Caccioppoli inequality) in an iteration procedure to obtain the~$C^\infty$ regularity of weak solutions. We also use the Poincar\'e-type inequality to give an elementary proof of the exponential convergence to equilibrium for solutions of the kinetic Fokker-Planck equation which mirrors the classic dissipative estimate for the heat equation. Finally, we prove enhanced dissipation in a weakly collisional limit.
\end{abstract}

\author[D. Albritton]{D. Albritton} 
\address[D. Albritton]{Courant Institute of Mathematical Sciences, New York University, 251 Mercer St., New York, NY 10012}
\curraddr{School of Mathematics, Institute for Advanced Study, 1 Einstein Dr., Princeton, NJ 08540, USA}
\email{dallas.albritton@ias.edu}

\author[S. Armstrong]{S. Armstrong}
\address[S. Armstrong]{Courant Institute of Mathematical Sciences, New York University, 251 Mercer St., New York, NY 10012}
\email{scotta@cims.nyu.edu}

\author[J.-C. Mourrat]{J.-C. Mourrat}
\address[J.-C. Mourrat]{ENS Lyon, CNRS, 46 all\'ee d'Italie, 69007 Lyon, France; Courant Institute of Mathematical Sciences, New York University, 251 Mercer St., New York, NY 10012}
\email{jean-christophe.mourrat@ens-lyon.fr}

\author[M. Novack]{M. Novack}
\address[M. Novack]{Courant Institute of Mathematical Sciences, New York University, 251 Mercer St., New York, NY 10012}
\curraddr{School of Mathematics, Institute for Advanced Study, 1 Einstein Dr., Princeton, NJ 08540, USA}
\email{mdn@ias.edu}

\date{\today}

\maketitle
\setcounter{tocdepth}{1}
\tableofcontents

\section{Introduction}

\subsection{Motivation and informal summary of results}

In this paper, we develop a well-posedness and regularity theory for weak solutions of the hypoelliptic equation 
\begin{equation}
\label{e.thepde}
-\Delta_v f + v \cdot \nabla_v f  + v \cdot \nabla_x f + \b\cdot \nabla_v f =  f^* \quad \mbox{in} \ {\T^d} \times \Rd \, .
\end{equation}
The unknown function $f(x,v)$ is a function of the position variable $x\in \T^d$ and the velocity variable $v\in\Rd$. The PDE \eqref{e.thepde} is sometimes called the \emph{Kramers equation}. We also consider the time-dependent version of this equation, namely 
\begin{equation}
\label{e.thepde.witht}
\partial_t f -\Delta_v f + v \cdot \nabla_v f + v \cdot \nabla_x f + \b\cdot \nabla_v f =   f^* \quad \mbox{in}  \ 
(0,\infty) \times \T^d \times\Rd \, ,
\end{equation}
which is often called the \emph{kinetic Fokker-Planck equation}.

\smallskip

These equations were first studied by Kolmogorov~\cite{K} and were the main motivating examples for the general theory of H\"ormander~\cite{H} of hypoelliptic equations. They are of physical interest due to their relation with the \emph{Langevin diffusion process} formally defined by
\begin{equation}
\label{e.newton}
\ddot X =  \b(X) - \dot X + \dot B \, ,
\end{equation}
where $\dot X$, $\ddot X$ stand respectively for the first and second time derivatives of $X$, a stochastic process taking values in $\Rd$, and $\dot B$ denotes a white noise process. Equation~\eqref{e.newton} can be interpreted as Newton's law of motion for a particle subject to the force field $\b(X)$, friction and thermal noise. This process can be recast as a Markovian evolution for the pair $(X,V)$ evolving according to
\begin{equation*}  
\Ll\{
\begin{aligned}  
\dot X & =  - V, \\
\dot V & = - \b(X) - V - \dot B.
\end{aligned}
\Rr.
\end{equation*}
The infinitesimal generator of this Markov process is the differential operator appearing on the left side of \eqref{e.thepde}.

\smallskip

Kolmogorov~\cite{K} gave an explicit formula for the fundamental solution of~\eqref{e.thepde.witht} in the case $\b = 0$ and $U = \Rd$, which gives the existence of smooth solutions of~\eqref{e.thepde} and~\eqref{e.thepde.witht} and implies that the operators on the left sides of~\eqref{e.thepde} and~\eqref{e.thepde.witht} are \emph{hypoelliptic}---that is, if~$f$ is a distributional solution of either of these equations and~$f^*$ is smooth, then~$f$ is also smooth. This result is extended to more general equations in H\"ormander's celebrated paper~\cite{H}, where he gave an essentially complete classification of hypoelliptic operators. In the case of the particular equations~\eqref{e.thepde} and~\eqref{e.thepde.witht}, his arguments yield a more systematic proof of Kolmogorov's results and, in particular, interior regularity estimates. 

\smallskip

The study of hypoelliptic equations often falls back on the theory of pseudodifferential operators; see for example Kohn's proof \cite{Kohnsproof} of H\"ormander's classical result \cite{H}, which H\"ormander includes in his monograph \cite{hormandervol3}.
The purpose of this paper is rather to present a functional analytic and variational theory for~\eqref{e.thepde} and~\eqref{e.thepde.witht} which has strong analogies to the familiar theory of uniformly elliptic equations. In particular, in this paper we: 
\begin{itemize}
\item identify a function space $H^1_\hyp$ based on the natural energy estimates and develop a notion of weak solutions in this space;
\item prove functional inequalities for $H^1_\hyp$, for instance a Poincar\'e-type inequality, which implies uniform coercivity of our equations and holds not just on the spatial domain $\T^d$ but on any $C^1$ domain; 
\item develop a well-posedness theory of weak solutions based on the minimization of a uniformly convex functional;
\item develop a regularity theory for weak solutions, based on an iteration of energy estimates, which implies that weak solutions are smooth; 
\item prove dissipative estimates for solutions of~\eqref{e.thepde.witht}, using the coercivity of the variational structure, which imply an exponential decay to equilibrium.
\end{itemize}
Such a theory has until now remained undeveloped, despite the attention these equations have received in the last half century. The definition of the space~$H^1_\hyp$ is not new: it and variants of it have been studied previously in the works~\cite{BG,PV,Car}. However, the functional inequalities and other key properties which are required to work with this space are established here.  A robust notion of weak solutions and corresponding well-posedness theory---besides allowing one to prove classical 
results for~\eqref{e.thepde} and~\eqref{e.thepde.witht} in a different way---is important because it provides a natural framework for studying the stability of solutions (i.e., proving that a sequence of approximate solutions converges to a solution). In fact, it is just such an application---namely, developing a theory of homogenization for~\eqref{e.thepde.witht}---which motivated the present work. Furthermore, we expect that the theory developed here will provide a closer link between the hypoelliptic equations~\eqref{e.thepde} and~\eqref{e.thepde.witht} and the classical theory of uniformly elliptic and parabolic equations, allowing, for example, for a more systematic development of regularity estimates for solutions of the former by analogy to the latter. For instance, it would be interesting to investigate a possible connection between the functional-analytic framework proposed in this paper and the recent works~\cite{WZ1,WZ2,GIMV,Mouh} which develop De Giorgi-Nash-type H\"older estimates for generalizations of the kinetic Fokker-Planck equations with measurable coefficients.\footnote{We refer to works of Guerand and Imbert \cite{guerand2021logtransform} and  Anceschi and Rebucci \cite{anceschi2021note}, which appeared after the first version of the present paper.} 

\smallskip

In the first part of the paper, we address the well-posedness of \eqref{e.thepde} under a weak formulation based on the Sobolev-type space~$H^1_\hyp(\T^d)$, defined below in~\eqref{e.def.h1hypnorm}. In the case in which~$\b$ is a potential field, we provide two proofs of well-posedness.  The first relies on the abstract Lax-Milgram theorem, while the second identifies a \emph{uniformly convex} functional that has the sought-after weak solution as its unique minimizer. The identification of the correct convex functional is inspired by previous work of Br\'ezis and Ekeland \cite{BE1,BE2} on variational formulations of parabolic equations (see also the more recent works~\cite{GBook,ABM} and the references therein). The proof that our functional is {coercive} relies on a new Poincar\'e-type inequality for $H^1_\hyp$, see Theorem~\ref{t.hypoelliptic.poincare} below. The Poincar\'e inequality in fact holds in a much more general setting than the periodic setting in which we consider \eqref{e.thepde}.
Our convex-analytic arguments for well-posedness can be immediately adapted to cover non-linear equations such as those obtained by replacing $\Delta_v f$ in \eqref{e.thepde} with $\nabla_v \cdot(\a(x,v,\nabla_v f))$, for $p \mapsto \a(x,v,p)$ a Lipschitz and uniformly maximal monotone operator (uniformly over $x \in \T^d$ and $v \in \Rd$). 

\smallskip

Roughly speaking, the norm~$\left\| \cdot\right\|_{H^1_\hyp(U)}$ is a measure of the size of the vector fields $\nabla_v f$ and $v \cdot \nabla_x f$, but crucially, the former is measured in a strong~$L^2_xL^2_v$-type norm and the latter in a weaker~$L^2_xH^{-1}_v$-type norm (see~\eqref{e.def.h1hypnorm} below). The importance of measuring the vector fields~$\nabla_v f$ and~$v\cdot \nabla_x f$ using different norms also features prominently in other works including~\cite{Bo}, but only spaces of positive regularity are considered there. Measuring the term~$v\cdot \nabla_x f$ in a space of negative regularity in the~$v$-variable is related to the idea of \emph{velocity averaging}, the idea that one should expect better control of the spatial regularity of a solution of~\eqref{e.thepde} or~\eqref{e.thepde.witht} after averaging in the velocity variable. This concept is therefore wired into the definition of the~$H^{1}_\hyp$ norm, allowing us to perform velocity averaging in a systematic way. Once we have proved the existence of weak solutions to~\eqref{e.thepde} in $H^1_\hyp$, we are interested in showing that these solutions are in fact smooth. It is elementary to verify that the differential operators $\nabla_v$ and $v\cdot \nabla_x$ satisfy H\"ormander's bracket condition, and therefore, as exposed in~\cite{H}, a control of both~$\nabla_v f$ and~$v\cdot \nabla_x f$ in~$L^2_xL^2_v$ would yield control of the seminorm of the function~$f$ in a fractional Sobolev space of positive regularity, namely~$H^{\sfrac 12}_x L^2_v$. However, since the natural definition of the function space~$H^1_\hyp(U)$ provides us only with control of~$v\cdot \nabla_x f$ in a space of \emph{negative} regularity in~$v$, we are forced to revisit the arguments of \cite{H}. A key step there is an interpolation-type inequality which converts the $L^2_x H^{-1}_v$ control on $v\cdot\nabla_x f$ (i.e., ``velocity averaged'' regularity) and $L^2_x H^1_v$ regularity on $f$ into $L^2_xL^2_v$ regularity for a type of ``fractional derivative" $(v\cdot\nabla_x)^{\sfrac{1}{2}}f$.\footnote{{The analogous estimate for the heat equation is $f\in H^{\sfrac{1}{2}}_t L^2_x$.}} With this interpolation in hand, we then prove a functional inequality (see Theorem~\ref{t.hormander} below) which asserts that the~$H^1_\hyp(U)$ norm controls exactly one-third of a derivative in arbitrary~$x$-directions in the space $L^2_xL^2_v$ in a weaker (Besov) sense, and almost one-third of a derivative in a stronger (Sobolev) sense. The one-third exponent is identical to that in H\"ormander's paper and is sharp. \footnote{{When translating H\"ormander's work~\cite{H} into the present setting, the vector field is $X_0=\partial_t + v\cdot\nabla_x$, and for simplicity we consider the ``flat case" in which $X_1=\nabla_v$. The regularity along $X_0$ is of index $\sfrac{1}{2}$, while the regularity along $X_1$ is of index $1$.  Then H\"ormander's Theorem 4.3 gives regularity along the commutator $\nabla_x = [X_1,X_0]$ of index $\sfrac{1}{3}$, since $\frac{1}{\sfrac{1}{3}}=\frac{1}{1}+\frac{1}{\sfrac{1}{2}}$. In addition, the exponent $\sfrac{1}{3}$ arises naturally in the following way: consider $\p_t f + v \cdot \nabla_x f - \varepsilon \Delta_v f = 0$ on $\R_+ \times \R^d \times \R^d$. Dimensionally speaking, $[f] = M$, $[x] = L$, $[v] = L/T$, and $[\varepsilon] = L^2/T^3$. The above PDE has a two-parameter scaling symmetry which keeps $\varepsilon$ fixed, namely, $f \to \rho f(\lambda^{\sfrac{2}{3}} t, \lambda x, \lambda^{\sfrac{1}{3}} v)$, $\lambda, \rho > 0$. Here, $\varepsilon$ is considered ``dimensionless": $[\varepsilon] = 1$, that is, we identify $L^2 \sim T^3$. In this convention, the unique exponent $\alpha$ for which $\| (-\Delta)^{\alpha/2}_x f \|_{L^2_{t,x,v}}$ has the same dimensions as $\| \nabla_v f \|_{L^2_{t,x,v}}$ is $\alpha = \sfrac{1}{3}$. Furthermore, the ``flat case" is the formal limit of~\eqref{e.thepde.witht} upon ``zooming in."}}
\smallskip

Once we have proved that an arbitrary $H^1_\hyp$ function possesses at least a fractional derivative in the~$x$ variable, we are in a position to iterate the estimate by repeatedly differentiating the equation a fractional number of times to obtain higher regularity (and eventually smoothness, under appropriate assumptions on~$\b$ and $f^*$) of weak solutions. In order to perform this iteration, we again depart from the original arguments of~\cite{H} and subsequent treatments and rely on an appropriate version of the Caccioppoli inequality (i.e., the basic $L^2$ energy estimate) for the equation~\eqref{e.thepde}. This avoids any recourse to {sophisticated} pseudodifferential operators and once again mimics the classical functional analytic arguments in the uniformly elliptic setting. 

\smallskip

The developments described above and even the variational structure identified for the equation~\eqref{e.thepde} are not restricted to the time-independent setting. Indeed, we show that they can be adapted in a very straightforward way to the kinetic Fokker-Planck equation~\eqref{e.thepde.witht}, the main difference being that the first-order part in a ``sum-of-squares'' representation of the differential operator is now~$\partial_t + v\cdot \nabla_x$ instead of just~$\,v\cdot \nabla_x$. The adaptation thus consists in replacing the latter by the former throughout; the natural function space associated with equation~\eqref{e.thepde.witht}, denoted by~$H^1_{\kin}$, is defined in~\eqref{e.H1kin.def}--\eqref{e.H1kin.norm.def}. We also prove a Poincar\'e inequality for functions in~$H^1_\kin$ which implies the uniform coercivity of the variational structure with respect to the $H^1_\kin$ norm. This allows us to give a rather direct and natural proof of exponential long-time decay to equilibrium for solutions of~\eqref{e.thepde.witht} with constant-in-time right-hand sides.
This result (stated in Theorem~\ref{t.hypo.equilibrium} below) can be compared with the celebrated results of exponential convergence to equilibrium for kinetic Fokker-Planck equations on~$\Rd$ with confining potentials, see in particular \cite{DV,HerauN,HN,EH2,DV2,V,baudoin}, as well as \cite{Talay0,Talay} and references therein for a probabilistic approach. Compared to previous approaches, our proof of exponential convergence is once again closer to the classical dissipative argument for the heat equation based on differentiating the square of the spatial~$L^2$ norm of the solution. Informally, 
our method is based on the idea that hypocoercivity is simply coercivity with respect to the correct norm.

\subsection{Statements of the main results}

We begin by introducing the Sobolev-type function space $H^1_\hyp$ associated with the equation~\eqref{e.thepde}. We let $U\subset\R^d$ either be a bounded $C^1$ domain with boundary, or we consider the boundary-less settings of $\R^d$ itself or the torus $\T^d$ with periodic boundary conditions. {While we do not prove unique solvability in $H^1_\hyp$ of the Dirichlet problem in bounded $C^1$ domains, we nonetheless can prove the Poincar\'e inequality, so we study the two settings (with and without boundary) in tandem.} We denote by~$\gamma$ the standard Gaussian measure on~$\Rd$, defined by
\begin{equation}  
\label{e.def.gamma}
d\gamma(v) := \left( 2\pi \right)^{-\frac d2} \exp\Ll(-\frac 1 2 |v|^2\Rr) \, dv \, .
\end{equation}
For each $p\in [1,\infty)$, we denote by $L^p_\gamma := L^p(\Rd,d\gamma)$ the Lebesgue space with norm
\begin{equation*}
\left\| f \right\|_{L^p_\gamma} 
:= 
\left( \int_{\Rd} \left| f(v) \right|^p \,d\gamma(v)  \right)^{\sfrac 1p}\, ,
\end{equation*}
and by $H^1_\gamma$ the Banach space with norm
\begin{equation*} 
\left\| f \right\|_{H^1_\gamma} 
:= 
\left( \left\| f \right\|_{L^2_\gamma}^2 + \left\| \nabla f \right\|_{L^2_\gamma}^2 \right)^{\sfrac 12}.
\end{equation*}
The dual space of $H^1_\gamma$ is denoted by $H^{-1}_\gamma$. By abuse of notation, we typically denote the canonical pairing~$\left\langle \cdot,\cdot\right\rangle_{H^1_\gamma,H^{-1}_\gamma}$ between $f\in H^1_\gamma$ and $f^*\in H^{-1}_\gamma$ by
\begin{equation}
\label{e.def.dualitypairing}
\int_{\Rd} f f^* \,d\gamma  := \left\langle f, f^* \right\rangle_{H^1_\gamma,H^{-1}_\gamma}. 
\end{equation}

Concerning the vector field $\b$, we shall often make the following assumption. Throughout the rest of the paper, we shall remind the reader when this assumption is in effect, or when we take more general vector fields $\b$. 
\begin{assumption} 
\label{a.conserv}
There exists $W \in C^{0,1}(U;\R)$ such that $\b(x) = - \nabla W(x)$ for almost every $x \in U$.
\end{assumption}
Under the above assumption, we denote by $d\sigma$ the measure on $U$ defined by
\begin{equation} 
\label{e.def.sigma}
d\sigma(x) := \exp(-W(x)) \, dx
\end{equation}
and by $dm$ the measure on $U \times \Rd$ defined by
\begin{equation}  
\label{e.def.m}
dm(x,v) := d\sigma(x) \, d\gamma(v) = \exp \Ll( -W(x) - \frac 1 2 |v|^2 \Rr) \, dx \, dv \, .
\end{equation}
A consequence of this definition and integration by parts is the equality
\begin{equation}\label{e.ibp.B}
    \iint_{\T^d\times\R^d} \left( v\cdot\nabla_x f(x,v) + \b(x)\cdot\nabla_v f(x,v) \right) \,dm = 0
\end{equation}
for all smooth $\T^d$-periodic functions $f$.
\smallskip

Given~$p\in[1,\infty)$, $U\subset\R^d$ and an arbitrary Banach space $X$, we denote by $L^p(U;X)$ the Banach space consisting of measurable functions $f:U \to X$ with norm
\begin{equation*}  
\|f\|_{L^p(U;X)} := \Ll( \int_{U} \|f(x,\cdot)\|_{X}^p \, dx \Rr)^{\sfrac 1p} \, .
\end{equation*}
It will occasionally be convenient to consider the space $L^p_\sigma(U;X)$, which contains functions for which the norm
\begin{equation*}  
\|f\|_{L^p_\sigma(U;X)} := \Ll( \int_{U} \|f(x,\cdot)\|_{X}^p \, d\sigma \Rr)^{\sfrac 1p} \, 
\end{equation*}
is finite.  Notice that, on bounded domains, the above norms induced by $dx$ and $d\sigma$ are equivalent under Assumption~\ref{a.conserv}.
\smallskip

We define the space $H^1_\hyp(U)$ by
\begin{equation}  
\label{e.def.H1hyp}
H^1_\hyp(U) := \Ll\{ f \in L^2\left(U;H^1_\gamma\right) \ : \ v \cdot \nabla_x f \in L^2(U;H^{-1}_\gamma) \Rr\} 
\end{equation}
and equip it with the norm
\begin{equation}
\label{e.def.h1hypnorm}
\left\| f \right\|_{H^1_\hyp(U)} 
:=
\left( \left\| f \right\|_{L^2(U;H^1_\gamma)}^2
+ 
\left\| v \cdot \nabla_x f \right\|_{L^2(U;H^{-1}_\gamma)}^2 \right)^{\sfrac 12}.
\end{equation}
When $\b$ satisfies Assumption~\ref{a.conserv}, it is natural to define the $H^1_\hyp$ norm with $\| v \cdot \nabla_x f + \b \cdot \nabla_v f \|_{L^2_\sigma(U;H^{-1}_\gamma)}$ replacing $\| v \cdot \nabla_x f \|_{L^2(U;H^{-1}_\gamma)}$ in~\eqref{e.def.H1hyp}. The two norms are evidently equivalent on a bounded domain.

Given a bounded domain $U\subseteq\Rd$ and a vector field $\b\in L^\infty(U\times \Rd)^d$, we say that a function $f \in H^1_\hyp(U)$ is a \emph{weak solution of~\eqref{e.thepde} in~$U\times \Rd$} if
\begin{equation*}
\forall h\in L^2(U;H^1_\ga), \qquad 
\int_{U\times\Rd} \nabla_v h \cdot \nabla_v f 
\, dx \, d\gamma
=
\int_{U \times \Rd} 
h \left( f^* - v \cdot \nabla_x f - \b \cdot \nabla_v f \right) \, dx \, d\gamma \, .
\end{equation*}
As in \eqref{e.def.dualitypairing}, the precise interpretation of the right side is
\begin{equation}
\label{e.def.dualitypairing.2}
\int_{U} 
\left\langle h(x,\cdot) , \left( f^* - v \cdot \nabla_x f - \b \cdot \nabla_v f \right)(x,\cdot) \right\rangle_{H^1_\gamma,H^{-1}_\gamma} \,dx \, . 
\end{equation}

As mentioned previously, we assume throughout that the domain $U\subset\R^d$ is bounded and has a $C^1$ boundary, or that $U=\T^d$ with periodic boundary conditions or $U=\R^d$. In the case $U\neq \T^d,\R^d$, we denote by $\n_U$ the outward-pointing unit normal to $\partial U$ and define the \emph{hypoelliptic boundary} of $U$ by
\begin{equation}\notag
    \partial_\hyp U := \left\{ (x,v)\in\partial U \times \R^d \, : \,  v\cdot\n_U(x) < 0 \right\} \, .
\end{equation}
We denote by $H^1_{\hyp,0}(U)$ the closure in $H^1_{\hyp}(U)$ of the set of smooth functions with compact support in $\overline{U}\times\R^d$ which vanish on $\partial_\hyp U$.
\smallskip

We give a first demonstration that $H^1_\hyp(U)$ is indeed the natural function space on which to build a theory of weak solutions of~\eqref{e.thepde} by presenting a well-posedness result for the Kramers equation. 
\begin{theorem}
[Well-posedness of the Kramers equation]
\label{t.hypo.DP}
Let $\b$ satisfy Assumption~\ref{a.conserv}, and let $f^*\in  L^2(\T^d;H^{-1}_\gamma)$ be such that {$\iint_{\T^d\times\R^d} f^*(x,v) \, dm = 0$}. Then there exists a unique weak solution $f \in H^1_\hyp(\T^d)$ to the Kramers equation 
\begin{equation}
\label{e.DP}
\begin{aligned}
& -\Delta_v f + v \cdot \nabla_v f + v \cdot \nabla_x f + \b\cdot \nabla_v f = f^*
& \mbox{in} & \ \T^d \times \Rd\, 
\end{aligned}
\end{equation}
with $\iint_{\T^d\times\R^d} f(x,v) \,dm = 0$. Furthermore, there exists a constant $C(\b,d)<\infty$ such that~$f$ satisfies the estimate
\begin{equation}
\label{e.DP.est}
\left\| f \right\|_{H^1_\hyp(\T^d)} 
\leq 
C \left\| f^* \right\|_{L^2(\T^d;H^{-1}_\gamma)} \, .
\end{equation}
\end{theorem}

We next give an informal discussion regarding how one could naively guess that~$H^1_\hyp$ is the ``correct'' space for solving~\eqref{e.thepde}, and how our proof of Theorem~\ref{t.hypo.DP} will work. We take the simpler case of matrix inversion in finite dimensions as a starting point. Given two matrices~$A$ and~$B$ with~$B$ skew-symmetric and a vector $f^*$, consider the problem of finding $f$ such that 
\begin{equation}  
\label{e.matrix.pb}
(A^* A + B)f = f^*,
\end{equation}
where $A^*$ denotes the transpose of $A$. We propose to approach this problem
by looking for a minimizer of the functional
\begin{equation*}  
f \mapsto \inf \Ll\{ \frac 1 2 (Af - \g, Af-\g) \ : \ \g \ \text{ such that } \ A^*\g = f^* - Bf \Rr\} ,
\end{equation*}
where $(\cdot,\cdot)$ denotes the underlying scalar product. It is clear that the infimum is non-negative, and if $f$ is a solution to \eqref{e.matrix.pb}, then choosing $\g = Af$ shows that this infimum is actually zero (null). Moreover, since $B$ is skew-symmetric, whenever~$(f,\g)$ satisfy the constraint in the infimum above, we have
\begin{equation} 
\label{e.matrix.functional}
\frac 1 2 (Af - \g, Af-\g) = \frac 1 2 (Af, Af) + \frac 1 2 (\g,\g) - (f,f^*).
\end{equation}
The latter quantity is clearly a convex function of the pair $(f,\g)$. The point is that under very mild assumptions on $A$ and $B$, it will in fact be \emph{uniformly} convex on the set of pairs~$(f,\g)$ satisfying the (linear) constraint $A^*\g = f^* - Bf$. Informally, the functional in~\eqref{e.matrix.functional} is coercive with respect to the seminorm $(f,\g) \mapsto |Af| + |\g| + |A (A^* A)^{-1} B f|$.

\smallskip

With this analogy in mind, and assuming that~$\b$ vanishes for simplicity, we rewrite the problem of finding a solution to \eqref{e.thepde} (with $\b \equiv 0$) as that of finding a null minimizer of the functional
\begin{equation}  
\label{e.mapping.inf}
f \mapsto \inf \Ll\{ \int_{\T^d\times \Rd} \frac 1 2 |\nabla_v f - \g|^2 \, dx \, d\ga \ : \ \nabla_v^* \g = f^* - v \cdot \nabla_x f \Rr\} ,
\end{equation}
where $\nabla_v^* F := -\nabla_v \cdot F + v \cdot F$ is the formal adjoint of $\nabla_v$ in $L^2_\ga$. It is clear that the infimum above is non-negative, and if we are provided with a solution $f$ to \eqref{e.thepde} (with $\b \equiv 0$), then choosing $\g = \nabla_v f$ reveals that this infimum vanishes at $f$. This functional gives strong credence to the definition of the space $H^1_\hyp(U)$ given in \eqref{e.def.H1hyp}. Using convex-analytic arguments, we show that the mapping in \eqref{e.mapping.inf} is uniformly convex, and that its infimum is null. This implies the well-posedness of the problem~\eqref{e.thepde} with $\b \equiv 0$. The proof of coercivity relies on the following Poincar\'e-type inequality for $H^1_\hyp(U)$.

\smallskip

For every $f \in L^1(U;L^1_\ga)$, we denote $(f)_U := |U|^{-1} \int_{U\times \Rd} f(x,v) \, d\sigma(x) \, d\ga(v)$. For the purposes of the Poincar\'e inequality, we may set $U=\T^d$, or $U\subset\R^d$ a general $C^1$ domain. See Proposition~\ref{pro:poincarewithconfining} and~\cite{cao2019explicit} for an extension to the case $U=\R^d$ with a confining potential.

\begin{theorem}[Poincar\'e inequality for $H^1_\hyp$]
\label{t.hypoelliptic.poincare}
For $U=\T^d$ or $U\subset\R^d$ a general bounded $C^1$ domain, there exists a constant $C(U,d)<\infty$ such that for every $f \in H^1_{\hyp}(U)$, we have
\begin{equation} 
\label{e.poincare.mean}
\left\| f -(f)_U \right\|_{L^2(U;L^2_\gamma)} 
\leq
C \left( \left\| \nabla_v f \right\|_{L^2(U;L^2_\gamma)} 
+
\left\| v \cdot \nabla_x f \right\|_{L^2(U;H^{-1}_\ga)}
\right) \, .
\end{equation}
Moreover, if in addition $f\in H^1_{\hyp,0}(U)$, then we have
\begin{equation} 
\label{e.poincare.zerobndry}
\left\| f \right\|_{L^2(U;L^2_\gamma)} 
\leq
C \left( \left\| \nabla_v f \right\|_{L^2(U;L^2_\gamma)} 
+
\left\| v \cdot \nabla_x f \right\|_{L^2(U;H^{-1}_\ga)} \right) \, .
\end{equation}
\end{theorem}

The inequality~\eqref{e.poincare.mean} asserts that, up to an additive constant, the full $H^1_\hyp(U)$ norm of a function~$f$ is controlled by the seminorm 
\begin{equation*}
\llbracket f \rrbracket_{H^1_\hyp(U)} := \left\| \nabla_v f \right\|_{L^2(U;L^2_\gamma)} 
+
\left\| v \cdot \nabla_x f \right\|_{L^2(U;H^{-1}_\ga)}.\end{equation*}
In particular, any distribution~$f$ with $\llbracket f \rrbracket_{H^1_\hyp(U)} < \infty$ is actually a function, which moreover belongs to~$L^2_xL^2_\gamma$. The inequality \eqref{e.poincare.zerobndry} is a then simple extension which shows that for functions which vanish on the hypoelliptic boundary, the full $H^1_\hyp$ norm is controlled by the seminorm.

\smallskip

The proof of Theorem~\ref{t.hypoelliptic.poincare} thus necessarily uses the H\"ormander bracket condition, although in this case the way it is used is rather implicit. If we follow H\"ormander's ideas more explicitly, then we obtain more information, namely some positive (fractional) regularity in the~$x$ variable. This is encoded in the following functional inequality, which we call the \emph{H\"ormander inequality}. The definitions of the fractional Sobolev spaces $H^\alpha$ used in the statement are given in Section~\ref{ss.hormander}, see~\eqref{e.def.frac.seminorm}.  The Besov space $\Qx(U)$ is defined in \eqref{e.Qx.norm} in Section~\ref{ss.besov} and measures difference quotients in the spatial variable $x$ of fractional order $\sfrac{1}{3}$. 

\begin{theorem}
[{H\"ormander inequality for {$H^1_\hyp$}}]
\label{t.hormander}
Let $\alpha \in \left[ 0,\tfrac 13 \right)$, and let {$U=\T^d$ or $U = \R^d$}. 
There exists a constant~$C({\al}, d)<\infty$ such that, for every~$f\in H^1_\hyp(U)$, we have the estimate 
\begin{equation}
\label{e.hormander}
\left\| f \right\|_{H^{\alpha}(U;L^2_\gamma)}
\leq
C \left\| f \right\|_{H^1_\hyp(U)} \, .
\end{equation}
For $\alpha=\sfrac{1}{3}$, we have the estimate
\begin{equation}\label{e.hormander.limit}
\left\| f \right\|_{\Qx(U)}
\leq
C \left\| f \right\|_{H^1_\hyp(U)}\, .
\end{equation}
\end{theorem}

The inequality~\eqref{e.hormander} gives control over a norm with non-negative regularity in~$x$ and~$v$. The estimate should be considered as an interior estimate in~$x$; in other words, for $U$ a general domain and any~$f\in H^1_\hyp(U)$, we can apply the inequality~\eqref{e.hormander} after multiplying~$f$ by a smooth cutoff function which vanishes for~$x$ near $\partial U$.
\smallskip

Our next main result asserts that weak solutions of~\eqref{e.thepde} are actually smooth. This is accomplished by an argument which closely parallels the one for obtaining~$H^k$ regularity for solutions of uniformly elliptic equations. We first obtain a version of the Caccioppoli inequality, that is, a reverse Poincar\'e inequality, which states that the $H^1_\hyp$ seminorm of a solution of~\eqref{e.thepde} can be controlled by its~$L^2$ oscillation (see Lemma~\ref{l.caccioppoli} for the precise statement). Combined with Theorem~\ref{t.hormander}, this tells us that a fractional spatial derivative of a solution of~\eqref{e.thepde} can be controlled by the $L^2$ oscillation of the function itself. This estimate can then be iterated: we repeatedly differentiate the equation a fractional amount to obtain estimates of the higher derivatives of the solution in the~$x$ variable; we then obtain estimates for derivatives in the~$v$ variable relatively easily. 

\smallskip

Notice that the following statement implies that solutions of~\eqref{e.thepde} are $C^\infty$ in both variables $(x,v)$ provided that the vector field~$\b$ is assumed to be smooth. For convenience, in the statement below we use the convention~$C^{-1,1} = L^\infty$.

\begin{theorem}[Interior Sobolev regularity for~\eqref{e.thepde}]
\label{t.interior.regularity.both}
Let~$k \in\N$, $r\in (0,\infty)$ and $\b \in C^{k-1,1}(B_r \times \R^d;\Rd)$. There exists a constant~$C<\infty$ depending on
\begin{equation*}
\left(d,k,r, \left\| \b \right\|_{C^{k-1,1}(B_r \times \R^d;\Rd)} \right)
\end{equation*}
such that, for every~$f \in H^1_\hyp(B_r)$ and~$f^* \in L^2(B_r;H^{-1}_\gamma)$ satisfying
\begin{equation} 
\label{e.intreg.both}
-\Delta_v f + v\cdot \nabla_v f + v\cdot \nabla_x f + \b \cdot \nabla_v f = f^* \quad \mbox{in} \ B_r \times \Rd \, ,
\end{equation}
the following holds: If $\p^\alpha f^* \in L^2(B_r;H^{-1}_\gamma)$ for all multi-indices $\alpha \in \N^d \times \N^d$ with $|\alpha| \leq k$, then we have 
$\p^\alpha f\in H^1_{\rm hyp}\left(B_{{\sfrac{r}{2}}}\right)$ and the estimate
\begin{equation*}
\left\| \p^\alpha f \right\|_{H^1_{\rm hyp}\left(B_{{\sfrac{r}{2}}}\right)}
\leq 
C\left( 
\left\| f - \left(f\right)_{B_{r}} \right\|_{L^2 (B_{r};L^2_\gamma)}
+
\sum_{|\beta| \leq k}
\left\| \p^\beta \tilde{f}^* \right\|_{L^2(B_r;H^{-1}_\gamma)}
\right)
\end{equation*}
for all multi-indices $\alpha \in \N^d \times \N^d$ with $|\alpha| \leq k$.
\end{theorem}

The results stated above are for the time-independent Kramers equation~\eqref{e.thepde}. In Section~\ref{s.kin}, we develop an analogous theory for the time-dependent kinetic Fokker-Planck  equation~\eqref{e.thepde.witht} with an associated function space~$H^1_\kin$ (defined in~\eqref{e.H1kin.def}--\eqref{e.H1kin.norm.def}) in place of~$H^1_\hyp$. In particular, we obtain analogues of the results above for~\eqref{e.thepde.witht} which are stated in Section~\ref{s.kin}. 

\smallskip

The long-time behavior of solutions of~\eqref{e.thepde.witht} has been studied by many authors in the last two decades: see the works of Desvillettes and Villani~\cite{DV}, H\'erau and Nier~\cite{HerauN}, Helffer and Nier~\cite{HN}, Eckmann and Hairer~\cite{EH2}, Desvillettes and Villani~\cite{DV2} and Villani~\cite{V} as well as the references in~\cite{V}. Most of these papers consider the case in which~$\b(x) = -\nabla W(x)$ for a potential~$W$ which has sufficient growth at infinity, in which case $dm$ is an explicit invariant measure, and solutions of~\eqref{e.thepde.witht} can be expected to converge exponentially fast to the constant which is the integral of the initial data with respect to the invariant measure. This setting is in a certain sense easier than the Dirichlet problem, since one does not have to worry about the boundary. While our methods could also handle this setting, we formulate a result for the exponential convergence of a solution of the Cauchy-Dirichlet problem with constant-in-time right-hand side  to the solution of the time-independent problem.

\begin{theorem}
[Convergence to equilibrium]
\label{t.hypo.equilibrium}
Let $U\subset \Rd$ be a $C^1$ domain and $\b \in L^\infty(U;C^{0,1}(\Rd))^d$. There exists $\lambda\left(\|\b\|_{L^\infty(U\times \Rd)},U,d\right) > 0$ satisfying the following property. Let $f^* \in L^2(U;H^{-1}_\gamma)$. Suppose that $f_\infty\in H^1_{\hyp,0}(U)$ solves~\eqref{e.DP}, and that for every $T \in (0,\infty)$, $f \in H^1_{\kin}((0,T)\times U)$ solves
\begin{equation}
\label{e.equileq}
\left\{
\begin{aligned}
& 
\partial_t f -\Delta_v f + v \cdot \nabla_v f + v \cdot \nabla_x f + \b\cdot \nabla_v f = f^*
& \mbox{in} & \ (0,T)\times U \times\Rd \, , \\
& f = 0 & \mbox{on} & \ (0,T) \times \partial_\hyp U \, ,
\end{aligned} 
\right. 
\end{equation}
where the boundary condition is satisfied in the sense that $f \in H^1_{\kin,||}((0,T) \times U)$.\footnote{$H^1_{\kin,||}((0,T) \times U)$ is defined to be the closure of test functions $C^\infty([0,T];U)$ vanishing on the lateral part of the hypoelliptic boundary, see subsection~\ref{ss.decay}.} Then, for every $t \ge 0$, we have
\begin{equation}
\label{e.expnn.conv}
\left\| f(t,\cdot) - f_\infty \right\|_{L^2(U;L^2_\ga)} 
\leq 
2\exp(-\lambda t)\left\|f(0,\cdot)- f_\infty \right\|_{L^2(U;L^2_\ga)}.
\end{equation}
\end{theorem}
Notice that interior regularity estimates immediately upgrade the~$L^2$ convergence in~\eqref{e.expnn.conv} to convergence in spaces of higher regularity (at least in the interior) with the same exponential rate. 

\smallskip

Unlike previous arguments establishing the exponential decay to equilibrium of solutions of~\eqref{e.thepde.witht} which are based on differentiation of perhaps non-transparent
quantities involving the solution and several (possibly mixed) derivatives in both~$x$ and~$v$, the proof of Theorem~\ref{t.hypo.equilibrium} we give here is elementary and close to the classical dissipative estimate for uniformly parabolic equations. The essential idea is to differentiate the square of the~$L^2$ norm of the solution and then apply the Poincar\'e inequality. We cannot quite perform the computation exactly like this, and so we use a finite difference instead of the time derivative and apply a version of the Poincar\'e inequality adapted to the kinetic equation in a thin cylinder (see Proposition~\ref{p.poincare.kin}). Unlike previous approaches, our method therefore relates the positive constant~$\lambda$ in~\eqref{e.expnn.conv} to the optimal constant in a Poincar\'e-type inequality. {One caveat of Theorem~\ref{t.hypo.equilibrium} is that, while we have a hypoelliptic Poincar{\'e} inequality in the above setting, we do not yet have a well-posedness theory in $H^1_\kin$ except when $U = \T^d$.}

\smallskip

Finally, we prove an enhanced dissipation estimate for solutions to the kinetic Fokker-Planck equation on the torus $\T^d$ with no right-hand side and $\b\equiv 0$ in a weakly collisional limit $\varepsilon\rightarrow 0^+$. The PDE satisfied by $f$ when initial data $f_{\rm in}$ is given then becomes
\begin{equation}
\label{e.thepde.witht.andeps}
\left\{
\begin{aligned}
\partial_t f + v \cdot \nabla_x f &= \varepsilon \left( \Delta_v f - v \cdot \nabla_v f \right) \quad \mbox{in}  \ 
(0,\infty) \times \T^d \times\Rd \, \\
f|_{t=0}&=f_{\rm in}\, .
\end{aligned}
\right.
\end{equation}
The spatial averages $f_{\rm avg}(t,v) := \int_{\T^d} f (t,x,v) dx$ satisfy
\begin{equation}
    \label{eq:equationforleaverages}
    \p_t f_{\rm avg} = \varepsilon \left( \Delta_v f_{\rm avg} - v \cdot \nabla_v f_{\rm avg} \right)
\end{equation}
and decay only on the dissipative timescale $T_{\rm d} \sim \varepsilon^{-1}$, as can be seen by rescaling $t$ in~\eqref{eq:equationforleaverages}. In the setting of~\eqref{e.thepde.witht.andeps}, enhanced dissipation is the observation that $f - f_{\rm avg}$ decays on the faster timescale $T_{\rm e} \sim \varepsilon^{-\sfrac{1}{3}}$:

\begin{theorem}[Enhanced dissipation]\label{t.enhancement}
There exist constants $C(d) < \infty$ and $c(d) > 0$ such that for every $\ep \in (0,1]$, initial data $f_{\rm in} \in L^2(\T^d;L^2_\ga)$ satisfying
\begin{equation}\label{e.meanzero}
    \int_{\T^d} f_{\rm in} (x,v) dx = 0 \qquad \forall v \in \R^d\, ,
\end{equation}
and for $f$ the unique solution of~\eqref{e.thepde.witht.andeps} constructed in Proposition~\ref{pro:solvabilitykinetic}, we have
\begin{equation}\label{e.enhancement}
    \left\| f(t,\cdot,\cdot) \right\|_{L^2(\T^d; L^2_\gamma} \leq C \left\| f_{\rm in} \right\|_{L^2(\T^d; L^2_\gamma} \exp\left( - c \varepsilon^{-\sfrac{1}{3}} t \right) \, .
\end{equation}
\end{theorem}

When enhancement cannot be extracted directly from an explicit solution formula, it is often approached by hypocoercivity techniques, which were developed by Villani~\cite{V} in the context of kinetic theory; see also work of Guo \cite{guolandau}. These methods were adapted to the context of fluid dynamics in work of Beck and Wayne~\cite{beckwaynebar}, Gallagher, Gallay, and Nier~\cite{gallaghergallaynier}, and Bedrossian and Coti-Zelati~\cite{jacobmicheleshear}. In joint work of the first and last authors with Beekie~\cite{abn21}, we demonstrated enhancement for solutions of certain advection-diffusion equations (passive scalars in shear flows) by methods which adhered more closely to H\"ormander's original paper~\cite{H}. Theorem~\ref{t.enhancement}, which is inspired by~\cite{abn21}, follows from an appropriate time- and $\varepsilon$-dependent version of the H\"ormander inequality from Theorem~\ref{t.hormander}.

\smallskip 
In principle, one may also prove~\eqref{e.enhancement} with $\b$ satisfying Assumption~\ref{a.conserv}, see Remark~\ref{rmk:inprincipleonecan}. It would be interesting to understand this method in the context of the Boltzmann and Landau equations.

\subsection{On unique solvability of the Dirichlet problem}

There is a subtle point in the analysis of the Dirichlet problem for~\eqref{e.thepde} on general domains $U$ which is due to the fact that we should prescribe the boundary condition only on part of the boundary, namely $\partial_\hyp U:= \left\{ (x,v) \in \partial U \times \Rd \, :\, v\cdot \n_U(x) < 0\right\}$, where $\n_U$ denotes the outer normal to~$U$. There is a difficulty coming from the possibly wild behavior of the trace of an $H^1_\hyp$ function near the \emph{singular set}  $\left\{ (x,v) \in \partial U \times \Rd \, :\, v\cdot \n_U(x) = 0\right\}$, where particle trajectories graze the boundary.  The following question remains open:\footnote{It is not difficult to define a pointwise a.e. trace away from the singular set, see Lemma~4.3 in the original version~\cite{armstrong2019variational} of this paper on arXiv, but apparently this has limited usefulness.}
\begin{question}
\label{q.trace}
Does there exist $C(U,d) < \infty$ such that for every $f \in C^\infty_c(\bar U\times \Rd)$,
\begin{equation*}  
\int_{\partial U \times \Rd} f^2 \, |v\cdot \n_U| \, dx \, d\ga \le C \|f\|_{H^1_\hyp(U)}^2 \quad ?
\end{equation*}
\end{question}
In the case of one spatial dimension ($d=1$), this difficulty has been previously overcome and the well-posedness result was already proved in~\cite{BG}. A generalization to higher dimensions was announced in~\cite{Car}, but we think that the argument given there is incomplete because the difficulty concerning the boundary behavior was not satisfactorily treated. This is explained in more detail in Appendix~A of the original version~\cite{armstrong2019variational} of the present work. A different way to phrase the main difficulty is discussed in Remark~\ref{rmk:difficultywithboundary}. 

\smallskip

The original version~\cite{armstrong2019variational} of this paper contained an error in the treatment of the Dirichlet and Cauchy-Dirichlet problems for the Kramers and kinetic Fokker-Planck equations, respectively.\footnote{See two equations below (4.20) in the original version on arXiv (``Arguing as in for the last term in (4.19), ...").} We were unable to repair the proof, see Remark~\ref{rmk:difficultywithboundary} below. In this version, we only prove unique solvability on the torus. \emph{It remains an interesting open question whether unique solvability holds with boundary in the natural $H^1_{\rm hyp}$ class.}


\smallskip

In the intervening years, we succeeded in improving the results in other ways. Foremost, we sharpen the H{\"o}rmander-type inequality from $\alpha = \sfrac{1}{6}-$ to $\alpha = \sfrac{1}{3}-$ \emph{without} cutoffs in the velocity variable. The second and third authors view this as a significant strengthening of the paper, essentially due to the first and fourth authors. This allows us to prove enhanced relaxation to equilibrium, which was not contained in the first version of the paper. There have also been many works revisiting~\cite{H} and at least partially inspired by the first version, see~\cite{Bedrossian2021,bedrossian2020quantitative,ABM,guerand2021logtransform,anceschi2021note,brigati2021time,cao2019explicit,lu2020explicit}.


\subsection{Outline of the paper}
In the next section we present the function space~$H^1_\hyp(U)$ and its important properties, as well as the Besov spaces used in the H\"ormander inequality. In Section~\ref{s.functional} we prove the functional inequalities stated in Theorems~\ref{t.hypoelliptic.poincare} and~\ref{t.hormander} and establish the compactness of the embedding of~$H^1_\hyp(U)$ into~$L^2(U;L^2_\gamma)$.  In Section~\ref{s.wellpose} we give two proofs of Theorem~\ref{t.hypo.DP} on the well-posedness of the Dirichlet problem for the Kramers equation. The interior regularity of solutions, and in particular Theorem~\ref{t.interior.regularity.both}, is obtained in Section~\ref{s.regularity}. Finally, in Section~\ref{s.kin} we prove the analogous results for the kinetic Fokker-Planck equation~\eqref{e.thepde.witht} as well as the exponential decay to equilibrium (Theorem~\ref{t.hypo.equilibrium}) and the enhancement estimate (Theorem~\ref{t.enhancement}).

\section{Function space basics}

In this section, we establish some basic properties of the function space~$H^1_\hyp(U)$ defined in~\eqref{e.def.H1hyp}--\eqref{e.def.h1hypnorm} and introduce several Besov-type spaces which will be necessary for the proof of the H\"ormander inequality.

\subsection{Properties of \texorpdfstring{$H^1_\gamma$}{honegamma} and \texorpdfstring{$H^{-1}_\gamma$}{hminusonegamma}}
We start by setting up some notation that will be used throughout the paper.  We denote the formal adjoint of the operator $\nabla_v$ by $\nabla_v^*$; that is, for every $F \in (H^1_\ga)^d$, we denote
\begin{equation}  
\label{e.def.nablavstar}
\nabla_v^* F := -\nabla_v \cdot F + v \cdot F.
\end{equation}
This definition can be extended to any $F \in (L^2_\ga)^d$, in which case $\nabla_v^* F \in H^{-1}_\ga$ and we have, for every $f \in H^1_\ga$,
\begin{equation*}  
\int_\Rd f \, \nabla_v^* F \, d\gamma = \int_\Rd \nabla_v f \cdot F \, d\ga \, .
\end{equation*}
Recall that the left side above is shorthand notation for the duality pairing between $H^1_\ga$ and $H^{-1}_\ga$. 
We denote the average of a function $f \in L^1_\gamma$ by 
\begin{equation}
\label{e.def.rangle.gamma}
\left\langle f \right\rangle_\gamma := \int_{\Rd} f \,d\gamma \, .
\end{equation}
Since $1 \in H^1_\ga$, the definition of $\la f \ra_\ga$ can be extended to arbitrary $f \in H^{-1}_\ga$. The Gaussian Poincar\'e inequality states that, for every $f \in H^1_\ga$, 
\begin{equation*}  
\|f - \la f \ra_\ga\|_{L^2_\ga} \le \|\nabla_v f\|_{L^2_\ga} \, .
\end{equation*}
We can thus replace~$\left\| f \right\|_{L^2_\gamma}$ by~$\left|\left\langle f \right\rangle_\gamma\right|$ in the definition of~$H^1_\gamma$ and have an equivalent norm:
\begin{equation*} \label{}
\left| \left\langle f \right\rangle_\gamma \right|^2
+ \left\| \nabla f \right\|_{L^2_\gamma}^2
\le 
\left\| f \right\|_{H^1_\gamma}^2 
\le 
2\left| \left\langle f \right\rangle_\gamma \right|^2 
+ 3\left\| \nabla f \right\|_{L^2_\gamma}^2 \, .
\end{equation*}
This comparison of norms has the following counterpart for the dual space $H^{-1}_\ga$.
\begin{lemma}[Identification of $H^{-1}_\ga$]
\label{l.rep.H-1}
There exists a universal constant $C < \infty$ such that for every $f^* \in H^{-1}_\ga$,
\begin{equation}  
\label{e.rep.H-1}
C^{-1} \|f^*\|_{H^{-1}_\ga} 
\le \left|  \left\langle f^*\right\rangle_\gamma \right|
+ \inf\left\{ \left\| \h \right\|_{L^2_\gamma} \,:\, \nabla_v^* \h   = f^* -  \left\langle f^*\right\rangle_\gamma \right\} 
\le C \|f^*\|_{H^{-1}_\ga} \, .
\end{equation}
\end{lemma}
\begin{proof}
The bilinear form
\begin{equation*}  
(f,g) \mapsto \la f \ra_\ga \, \la g \ra_\ga + \int_\Rd \nabla_v f \cdot \nabla_v g \, d\ga
\end{equation*}
is a scalar product for the Hilbert space $H^1_\ga$. By the Riesz representation theorem, for every $f^* \in H^{-1}_\ga$, there exists $g \in H^1_\ga$ such that
\begin{equation*}  
\forall f \in H^1_\ga \qquad \int_\Rd f f^* \, d\ga = \la f \ra_\ga \, \la g \ra_\ga + \int_\Rd \nabla_v f \cdot \nabla_v g \, d\ga \, .
\end{equation*}
(Recall that the integral on the left side is convenient notation for the canonical pairing between $H^1_\ga$ and $H^{-1}_\ga$.) We clearly have $\la g \ra_\ga = \la f^* \ra_\ga$, and thus
\begin{equation*}  
\left| \la g \ra_\ga\right|^2 + \int_\Rd |\nabla_v g|^2 \, d\ga  \le \|g\|_{H^1_\ga} \, \|f^*\|_{H^{-1}_\ga} \, .
\end{equation*}
This implies that $\|\nabla_v g\|_{L^2_\ga} \le C \|f^*\|_{H^{-1}_\ga}$, and since $\nabla_v^*\nabla_v g = f^* - \la f^* \ra_\ga$, 
this proves the rightmost inequality in \eqref{e.rep.H-1}. Conversely, for any $\h \in L^2_\ga$, if 
\begin{equation*}  
f^* = \la f^* \ra_\ga + \nabla_v^* \h \, ,
\end{equation*}
then for every $f \in H^1_\ga$,
\begin{equation*}  
\Ll| \int_\Rd f f^* \, d\ga \Rr| 
\le 
\Ll|\la f \ra_\ga\Rr| \, \Ll|\la f^* \ra_\ga\Rr| + \|\nabla f\|_{L^2_\ga} \, \|\h\|_{L^2_\ga} \, ,
\end{equation*}
and thus the leftmost inequality in \eqref{e.rep.H-1} holds. \end{proof}

We often work with the dual pair of Banach spaces $L^2(U;H^1_\gamma)$ and~$L^2(U;H^{-1}_\gamma)$. With the identification given by Lemma~\ref{l.rep.H-1}, we have
\begin{align} 
\label{e.eq.L2H-1}
\left\| f^* \right\|_{L^2(U;H^{-1}_\gamma)} 
\simeq \left\|  \left\langle f^*\right\rangle_\gamma \right\|_{L^2(U)}
+
\inf \left\{ \left\| \g \right\|_{L^2(U;L^2_\gamma)} 
\,:\, 
\nabla_v^* \g = f^* - \left\langle f^*\right\rangle_\gamma
\right\},
\end{align}
in the sense that the norms on each side are equivalent.

For convenience, for every $f \in L^1(U;L^1_\ga)$, we use the shorthand notation
\begin{equation}
\label{e.mean.U}
(f)_U := |U|^{-1} \int_{U\times \Rd} f(x,v) \, d\sigma(x) \, d\ga(v).
\end{equation}
We will occasionally also use this notation in the case when $f$ depends only on the space variable $x$, in which case we simply have $(f)_U = |U|^{-1} \int_U f \, d\sigma(x)$.

\smallskip
In the proof of the H\"ormander inequality, it will be beneficial to understand which type of finite differences are controlled by $\left\| f \right\|_{H^1_\gamma}$. Recall that
\begin{equation}\notag
d\gamma(v) := \left( 2\pi \right)^{-\frac d2} \exp\Ll(-\frac 1 2 |v|^2\Rr) \, dv \, .
\end{equation}
The fundamental issue is that $\gamma(\cdot + h)$ is not comparable to $\gamma$, above and below, uniformly in $v$. For instance, while the translation of the measure $\ga$ by a fixed vector $y \in \Rd$ is absolutely continuous with respect to $\ga$, the associated Radon-Nikodym derivative is unbounded (unless $y = 0$).  This distinguishes Gaussians from $e^{-\la x \ra}$, for example, and changes the finite difference characterization of the space
\begin{equation}\notag
    \norm{\nabla_v u}_{L^2(U;L^2_\gamma)} \, ,
\end{equation}
since its finite difference characterization is not in the seminorm
\begin{equation}\notag
    \sup_{h>0} h^{-1} \norm{u(x,v+h) - u(x,v)}_{L^2(U;L^2_\gamma)} \, .
\end{equation}

Towards an appropriate characterization, we first note that a consequence of the logarithmic Sobolev inequality and the Gaussian Poincar\'e inequality is the estimate
\begin{equation}
    \label{eq:ineqwithv}
\norm{|v| u}_{L^2(U;L^2_\gamma)} \lesssim \norm{\nabla_v u}_{L^2(U;L^2_\gamma)} \,
\end{equation}
for functions $u$ satisfying $\langle u \rangle_\gamma =0$; the reader may consult \eqref{e.extrav2integrab} and the ensuing discussion for details.
%
%
The inequality~\eqref{eq:ineqwithv}, together with the product rule, gives that
\begin{equation}\notag
    \norm{\nabla_v (u \gamma^{1/2})}_{L^2(U; L^2(\R^d))} \lesssim \norm{\nabla_v u}_{L^2(U;L^2_\gamma(\R^d))} \, ,
\end{equation}
and since the left-hand side has a finite difference characterization, we have
\begin{equation}\label{e.weighted.difference}
    \sup_{h\in\R^d\setminus\{0\}} |h|^{-1} \norm{u(x,v+h) \gamma^{1/2}(v+h) - u(x,v) \gamma^{1/2}(v)}_{L^2(U;L^2(\R^d))} \lesssim \norm{\nabla_v u}_{L^2(U;L^2_\gamma)} \, .
\end{equation}
We refer to Lunardi~\cite{lunardi} for further discussion.

\subsection{Density of Smooth Functions in \texorpdfstring{$H^1_\hyp$}{H1hyyp}}

We show that the set of smooth functions is dense in~$H^1_\hyp$. 
\begin{proposition}
\label{p.density}
The set $C^\infty_c(\overline{U} \times \Rd)$ of smooth functions with compact support in $\overline{U} \times\Rd$ is dense in $H^1_{\hyp}(U)$.
\end{proposition}

\begin{proof}
We focus on the case when $U \subset \R^d$ is a bounded $C^1$ domain. When $U = \T^d$, the proof can be done more simply by cutting off in $v$ and mollifying.

We decompose the proof into three steps.

\smallskip

\emph{Step 1.} In this step, we show that it suffices to consider the case when $U$ satisfies a convenient quantitative form of the star-shape property. For every $z \in \partial U$, there exist a radius $r > 0$ and a $C^1$ function $\Psi \in C^1(\R^{d-1};\R)$ such that, up to a relabelling of the axes, we have
\begin{equation*}  
U \cap B(z,r) = \{x = (x_1,\ldots,x_d) \in B(z,r) \ : \ x_d > \Psi(x_1,\ldots,x_{d-1})\}.
\end{equation*}
Since $\Psi$ is a $C^1$ function, there exists $\delta > 0$ such that for every $x \in U \cap B(z,r)$, we have the cone containment property
\begin{equation}  
\label{e.cone.pty}
\Ll\{x + y \ : \ \frac{y_d}{|y|} \ge 1-\delta\Rr\} \cap B(z,r) \subset U.
\end{equation}
Setting
\begin{equation*}  
z' = z + \Ll(0,\ldots,0,\frac r 2\Rr) \in \Rd,
\end{equation*}
and reducing $\delta > 0$ if necessary, we claim that for every $x \in U\cap B(z,\delta^2)$ and $\ep \in (0,1]$, we have
\begin{equation}
\label{e.star.quant}
B\Ll(x - \ep(x-z'),\delta^2 \ep \Rr) \subset U.
\end{equation}
Assuming the contrary, let $y \in \Rd$ be such that
\begin{equation*}  
x + y \in B\Ll(x - \ep(x-z'),\delta^2 \ep \Rr) \setminus U.
\end{equation*}
Then
\begin{equation*}  
\Ll|y + \ep(x-z')\Rr| \le \de^2\ep,
\end{equation*}
and therefore
\begin{align*}  
\Ll|y  - \ep \Ll(0,\ldots,0,\frac r 2\Rr)\Rr| 
& \le \Ll|y + \ep(x-z) - \ep \Ll(0,\ldots,0,\frac r 2\Rr)\Rr| + \ep|x-z| 
\\
& 
\le \Ll|y + \ep(x-z') \Rr| + \ep|x-z| 
\\
&
\le 2 \de^2\ep.
\end{align*}
Taking $\delta > 0$ sufficiently small, we arrive at a contradiction with the cone property \eqref{e.cone.pty}. 
Now that \eqref{e.star.quant} is proved for every $x$ in a relative neighborhood of $z$, and up to a further reduction of the value of $\delta > 0$ if necessary, it is not difficult to show that one can find an open set $U'$ containing $z$ and $z'$ and such that \eqref{e.star.quant} holds for every $x \in U \cap U'$.

\smallskip

Summarizing, and using the fact that $U$ is a bounded set, we have shown that there exist families of bounded open sets $U_1,\ldots,U_M \subset \Rd$, of points $x_1,\ldots, x_M \in \Rd$ and a parameter $r > 0$ such that 
\begin{equation*}  
U = \bigcup_{k = 1}^M U_i
\end{equation*}
and
for every $k \in \{1,\ldots,M\}$, $x \in U_k$ and $\eps \in (0,1]$,
\begin{equation*}  
B\Ll(x - \ep(x-x_k),r\ep\Rr) \subset U_k.
\end{equation*}
By using a partition of unity, we can reduce our study to the case when this property is satisfied for the domain $U$ itself (in place of each of the $U_k$'s). By translation, we may assume that the reference point $x_k$ is at the origin, and by scaling, we may also assume that this property holds with $r = 1$. That is, from now on, we assume that for every $x \in U$ and $\ep \in (0,1]$, we have
\begin{equation}  
\label{e.easy.cone}
B \Ll( (1-\ep) x, \ep \Rr) \subset U.
\end{equation}

\smallskip

\emph{Step 2.} 
Let $f \in H^1_\hyp(U)$. We aim to show that $f$ belongs to the closure of the set $C^\infty_c(\bar U \times \Rd)$ in $H^1_\hyp(U)$. 
Without loss of generality, we may assume that $f$ is compactly supported in $\bar U \times \Rd$. Indeed, if $\chi \in C^\infty_c(\Rd;\R)$ is a smooth function with compact support and such that $\chi \equiv 1$ in a neighborhood of the origin, then the function $(x,v) \mapsto f(x,v) \chi(v/M)$ belongs to $H^1_\hyp(U)$ and converges to $f$ in $H^1_\hyp(U)$ as $M$ tends to infinity. 

\smallskip

Let $\zeta \in C^\infty_c(\Rd;\R)$ be a smooth function with compact support in $B(0,1)$ and such that $\int_\Rd \zeta = 1$. For each $\ep > 0$ and $x \in \Rd$, we write
\begin{equation}  
\label{e.def.zetaep}
\zeta_\ep(x) := \ep^{-d} \zeta(\ep^{-1} x),
\end{equation}
and we define, for each $\eps \in \Ll( 0,\frac 1 2 \Rr]$, $x \in U$ and $v \in \Rd$,
\begin{equation*}  
f_\ep(x,v) := \int_{\Rd} f((1-\ep) x + y,v) \zeta_\ep(y) \, dy.
\end{equation*}
Note that this definition makes sense by the assumption of \eqref{e.easy.cone}.
The goal of this step is to show that $f$ belongs to the closure in $H^1_\hyp(U)$ of the convex hull of the set $\Ll\{f_\ep \ : \ \ep \in \Ll( 0,\frac 1 2 \Rr]\Rr\}$. By Mazur's lemma (see \cite[page~6]{ET}), it suffices to show that $f_\ep$ converges weakly to $f$ in $H^1_\hyp(U)$. Since it is elementary to show that $f_\ep$ converges to $f$ in the sense of distributions, this boils down to checking that $f_\ep$ is bounded in $H^1_\hyp(U)$. By Jensen's inequality,
\begin{align*}  
\|\nabla_v f_\ep\|_{L^2(U;L^2_\ga)}^2 
& 
\le 
\int_{U\times \Rd} \int_\Rd |\nabla_v f|^2 \Ll( (1-\ep)x + y,v \Rr) \zeta_\ep(y) \, dy \, dx \, d\ga(v)
\\
& \le 
(1-\ep)^{-1} \|\nabla_v f\|_{L^2(U;L^2_\ga)}^2.
\end{align*}
In order to evaluate $\|v\cdot \nabla_x f_\ep\|_{L^2(U;H^{-1}_\ga)}$, we compute, for every $\varphi \in L^2(U;H^1_\ga)$,
\begin{align*}  
\lefteqn{
\int_{U\times\Rd} v \cdot \nabla_x f_\ep \, \varphi \, dx \, d\ga
} \qquad & 
\\ & 
= (1-\ep) \int_{U\times \Rd} \int_\Rd v \cdot \nabla_x f \Ll( (1-\ep)x + y,v \Rr) \zeta_\ep(y)  \varphi(x,v) \, dy\, dx\, d\ga(v)
\\ & 
= \int_{U\times \Rd} \int_\Rd v \cdot \nabla_x f \Ll( x + y,v \Rr) \zeta_\ep(y) \, \varphi\Ll(\frac{x}{1-\ep},v\Rr) \, dy\, dx\, d\ga(v)
\\ & 
= \int_{U\times \Rd} \int_\Rd v \cdot \nabla_x f \Ll( y,v \Rr) \zeta_\ep(y-x) \, \varphi\Ll(\frac{x}{1-\ep},v\Rr) \, dy\, dx\, d\ga(v).
\end{align*}
Since, by Jensen's inequality,
\begin{equation*}  
\int_{U\times \Rd} \Ll| \int_U \zeta_\ep(y-x) \varphi\Ll(\frac{x}{1-\ep},v\Rr) \, dx\Rr|^2 \, dy \, d\ga(v)
\le (1-\ep)^{-1} \|\varphi\|_{L^2(U;L^2_\ga)}^2
\end{equation*}
as well as
\begin{equation*}  
\int_{U\times \Rd} \Ll| \int_U \zeta_\ep(y-x) \nabla_v\varphi\Ll(\frac{x}{1-\ep},v\Rr) \, dx\Rr|^2 \, dy \, d\ga(v)
\le (1-\ep)^{-1} \|\nabla_v \varphi\|_{L^2(U;L^2_\ga)}^2,
\end{equation*}
we deduce that
\begin{equation*}  
\int_{U\times\Rd} v \cdot \nabla_x f_\ep \, \varphi \, dx \, d\ga \le (1-\ep)^{-\frac 1 2} \|v\cdot \nabla_x f\|_{L^2(U;H^{-1}_\ga)} \, \|\varphi\|_{L^2(U;H^1_\ga)},
\end{equation*}
and therefore 
\begin{equation*}  
\|v\cdot \nabla_x f_\ep\|_{L^2(U;H^{-1}_\ga)} \le (1-\ep)^{-\frac 1 2} \|v\cdot \nabla_x f\|_{L^2(U;H^{-1}_\ga)}.
\end{equation*}
This completes the proof that the set $\Ll\{f_\ep \ : \ \ep \in \Ll( 0,\frac 1 2 \Rr] \Rr\}$ is bounded in $H^1_\hyp(U)$, and thus that $f$ belongs to the closed convex hull of this set.

\smallskip

\emph{Step 3.}
It remains to be shown that for each fixed $\ep \in \Ll( 0,\frac 1 2 \Rr]$, the function $f_\ep$ belongs to the closure in $H^1_\hyp(U)$ of the set $C^\infty_c(\bar U \times \Rd)$. For every $\eta \in (0,1]$, we define
\begin{align*}  
f_{\ep,\eta}(x,v) & := \int_\Rd f_\ep(x,w) \zeta_\eta(v-w) \, dw 
\\
& = 
\int_\Rd \int_\Rd f(y,w) \zeta_\ep(y-(1-\ep)x) \zeta_\eta(v-w) \, dy \, dw.
\end{align*}
From the last expression, we see that $f_{\ep,\eta}$ belongs to $C^\infty_c(\bar U\times \Rd)$ (recall that $f$ itself has compact support in $\bar U \times \Rd$). Moreover, since $\nabla_v f_\ep \in L^2(U;L^2_\ga)$ and
\begin{equation*}  
\nabla_v f_{\ep,\eta}(x,v) = \int_\Rd \nabla_v f_\ep(x,v-w)  \zeta_\eta(w) \, dw,
\end{equation*}
it is classical to verify that $\nabla_v f_{\ep,\eta}$ converges to $\nabla_v f_\ep$ in $L^2(U;L^2_\ga)$ as $\eta$ tends to $0$. By the definition of $f_\ep$ and the fact that $f_\ep$ is compactly supported, we have that $v \cdot \nabla_x f_\ep \in L^2(U;L^2_\ga)$. The same reasoning as above thus gives that $v \cdot \nabla_x f_{\ep,\eta}$ converges to $v \cdot \nabla_x f_\ep$ in $L^2(U;L^2_\ga)$, and thus a fortiori in $L^2(U;H^{-1}_\ga)$, as $\eta$ tends to $0$. This shows that 
\begin{equation*}  
\lim_{\eta \to 0} \|f_{\ep,\eta} - f_\ep\|_{H^1_\hyp(U)} = 0
\end{equation*}
and thus completes the proof of the proposition. \end{proof}

\subsection{Besov Spaces}\label{ss.besov}

We shall use the following Besov-type spaces in the proof of the H\"{o}rmander inequality. The first of these spaces measures fractional regularity along the vector field $v\cdot\nabla_x$, while the second measures fractional regularity along $\nabla_x$. As the H\"ormander inequality is an interior estimate, we only consider these spaces in the cases that $U=\R^d$ or $U=\T^d$. To lighten the notation, we may frequently write $\left\| \cdot \right\|_{\Qvx}$ rather than $\left\| \cdot \right\|_{\Qvx(U)}$, as the choice of $U=\R^d,\T^d$ plays no role in the argument. The $Q$ stands for ``quotient."
\begin{definition}\label{d.Qvx}
For measurable $f:U\times\mathbb{R}^d\rightarrow\mathbb{R}$, we define
\begin{equation}\label{e.Qvx.norm}
    \left\| f \right\|_{\Qvx(U)}^2 := \sup_{0< \eta < \infty} \frac{1}{\eta^2} \iint_{U\times\mathbb{R}^d} \left( f(x+\eta^2 v, v) - f(x,v) \right)^2 \,d\gamma(v)\,dx\, .
\end{equation}
\end{definition}

\smallskip

\begin{definition}\label{d.Qx}
For measurable $f:U\times\mathbb{R}^d\rightarrow\mathbb{R}$, we define
\begin{equation}\label{e.Qx.norm}
    \left\| f \right\|_{\Qx(U)}^2 := \sup_{\substack{0< \eta <\infty \\ x'\in \mathbb{S}^{d-1}}} \frac{1}{\eta^{2}} \iint_{U\times\mathbb{R}^d} \left( f(x+\eta^3 x', v) - f(x,v) \right)^2\,d\gamma(v) \,dx\, .
\end{equation}
\end{definition}

\smallskip

\section{Functional inequalities for \texorpdfstring{$H^1_\hyp$}{H1hyp}}
\label{s.functional}

In this section we present the proofs of  Theorems~\ref{t.hypoelliptic.poincare} and~\ref{t.hormander}. 

\subsection{The Poincar\'e inequality for \texorpdfstring{$H^1_\hyp$}{H1hyp}}

We begin with the proof of Theorem~\ref{t.hypoelliptic.poincare}, the Poincar\'e-type inequality for the space $H^1_\hyp(U)$. 
The proof requires the following fact regarding the equivalence (up to additive constants) of the norms $\| h \|_{L^2(U)}$ and $\| \nabla h \|_{H^{-1}(U)}$.

\begin{lemma}
\label{l.est.L2.H-1}
Let $U$ be a Lipschitz domain or $U=\T^d$. Then there exists $C(U,d) < \infty$ such that for every $h \in L^2(U)$,
\begin{equation*}  
\Ll\|h - (h)_U \Rr\|_{L^2(U)} \le C \|\nabla h\|_{H^{-1}(U)} \, .
\end{equation*}
\end{lemma}
\begin{proof}
We begin by considering the case $U$ is a Lipschitz domain. Without loss of generality, we assume that $(h)_U = 0$. We consider the problem 
\begin{equation} 
\label{e.boundary.pb}
\left\{
\begin{aligned}
& \nabla\cdot \f  = h & \mbox{in} & \ U, \\
& \f = 0 & \mbox{on} & \ \partial U \, . 
\end{aligned}
\right.
\end{equation}
Bogovskii's operator~\cite{bogovskii} (see also Galdi's book~\cite[Section III.3]{galdi}) guarantees the existence of
a solution~$\f$ with components in $H^1_0(U)$ satisfying the estimate
\begin{equation} 
\label{e.est.bdypb}
\left\| \f \right\|_{H^1(U)} \leq C \left\| h \right\|_{L^2(U)} \, . 
\end{equation}
Then we have
\begin{align*}  
\|h\|_{L^2(U)}^2
= \int_U h \, \nabla \cdot \f 
= - \int_U \nabla h \cdot \f 
\leq \|\nabla h\|_{H^{-1}(U)} \, \|\f\|_{H^1(U)} \, .
\end{align*}
The conclusion then follows by \eqref{e.est.bdypb}. {In the case $U=\T^d$, the estimate follows from classical Littlewood-Paley estimates, and we omit the details.} \end{proof}

\begin{proof}
[Proof of Theorem~\ref{t.hypoelliptic.poincare}]
Let $f \in H^1_\hyp(U)$. In view of Proposition~\ref{p.density}, we can without loss of generality assume that $f$ is a smooth function. We decompose the proof into five steps.

\smallskip

\emph{Step 1.} We show that 
\begin{equation} 
\label{e.yosccontrol}
\left\| f   - \langle f \rangle_\gamma   \right\|_{L^2(U;L^2_\gamma)} \leq \left\| \nabla_v f\right\|_{L^2(U;L^2_\gamma)}.  
\end{equation}
By the Gaussian Poincar\'e inequality, we have for every $x\in U$ that
\begin{equation*} \label{}
\left\| f(x,\cdot)  - \langle f \rangle_\gamma(x)  \right\|_{L^2_\gamma} \leq \left\| \nabla_vf(x,\cdot) \right\|_{L^2_\gamma}.  
\end{equation*}
This yields~\eqref{e.yosccontrol} after integration over~$x\in U$. 

\smallskip

\emph{Step 2.} We show that
\begin{equation} 
\label{e.bracketuHminus1}
\left\| \nabla \langle f \rangle_\gamma \right\|_{H^{-1}(U)}
\leq
C\left( 
\left\| \nabla_v f \right\|_{L^2(U;L^2_\gamma)} +\left\| v \cdot \nabla_x f \right\|_{L^2(U;H^{-1}_\gamma)} 
\right).
\end{equation}
We select~$\xi_1, \ldots, \xi_d \in C^\infty_c(\Rd)$ satisfying
\begin{equation} 
\label{e.xi_i.2}
\int_{\Rd} v \xi_i(v)\,d\gamma(v) = e_i,
\end{equation}
and for each test function $\phi\in H^1_0(U)$ and $i \in \{1,\ldots,d\}$, we compute
\begin{align*}
\int_U \partial_{x_i} \phi(x) \langle f \rangle_\gamma(x)\,dx
& 
=
\int_{U\times \Rd} 
v \cdot \nabla_x \phi(x)  \langle f \rangle_\gamma(x)\xi_i(v) \,dx \,d\gamma(v)
\\ & 
= 
\int_{U\times \Rd} 
v \cdot \nabla_x \phi(x)  f(x,v)\xi_i(v) \,dx \,d\gamma(v)
\\ & \qquad 
+
\int_{U\times \Rd} 
v \cdot \nabla_x \phi(x)  \left( f(x,v) -\langle f \rangle_\gamma(x) \right) \xi_i(v) \,dx \,d\gamma(v).
\end{align*}
To control the first term on the right side, we perform an integration by parts to obtain 
\begin{align*}
\left| \int_{U\times \Rd} 
v \cdot \nabla_x \phi(x)  f(x,v)\xi_i(v) \,dx \,d\gamma(v) \right| 
&
=
\left| \int_{U\times\Rd}
\phi(x) \xi_i(v) \, v\cdot \nabla_x f(x,v) \,dx \, d\gamma(v) \right| 
\\ & 
\leq 
C
\left\| \phi \xi_i \right\|_{L^2(U;H^1_\gamma)}
\left\| v\cdot \nabla_x f \right\|_{L^2(U;H^{-1}_\gamma)}
\\ & 
\leq
C
\left\| \phi\right\|_{L^2(U)}
\left\|  \xi_i  \right\|_{H^1_\gamma}
\left\| v\cdot \nabla_x f \right\|_{L^2(U;H^{-1}_\gamma)}
\\ & 
\leq
C \left\| \phi \right\|_{L^2(U)} 
\left\| v\cdot \nabla_x f \right\|_{L^2(U;H^{-1}_\gamma)}.
\end{align*}
To control the second term, we use~\eqref{e.yosccontrol} and the fact that $\xi_i$ has compact support:
\begin{align*} \label{}
\lefteqn{
\left| \int_{U\times \Rd} 
v \cdot \nabla_x \phi(x)  \left( f(x,v) -\langle f \rangle_\gamma(x) \right) \xi_i(v) \,dx \,d\gamma(v) 
\right| 
} \qquad & 
\\ &
\leq
C\int_{U\times \Rd} |v| |\xi_i(v)| \left| \nabla_x\phi(x) \right| \left| f(x,v) - \langle f \rangle_\gamma(x)\right| \,dx\,d\gamma(v) 
\\ & 
\leq C \left\| \phi \right\|_{H^1(U)} \left\| \nabla_v f \right\|_{L^2(U;L^2_\gamma)}.
\end{align*}
Combining the above displays and taking the supremum over~$\phi\in H^1_0(U)$ with $\| \phi \|_{H^1(U)}\leq 1$ yields~\eqref{e.bracketuHminus1}.

\smallskip

\emph{Step 3.} We deduce from~Lemma~\ref{l.est.L2.H-1},~\eqref{e.yosccontrol} and~\eqref{e.bracketuHminus1} that 
\begin{align*} \label{}
\left\| f - \left(f\right)_U \right\|_{L^2(U;L^2_\gamma)}
&
\leq
\left\| f - \left\langle f \right\rangle_\gamma \right\|_{L^2(U;L^2_\gamma)}
+ \left\| \left\langle f \right\rangle_\gamma - \left( f \right)_U \right\|_{L^2(U)}
\\ & 
\leq
\left\| f - \left\langle f \right\rangle_\gamma \right\|_{L^2(U;L^2_\gamma)}
+ C\left\| \nabla \left\langle f \right\rangle_\gamma  \right\|_{H^{-1}(U)}
\\ & 
\leq 
C\left( 
\left\| \nabla_v f \right\|_{L^2(U;L^2_\gamma)} +\left\| v \cdot \nabla_x f \right\|_{L^2(U;H^{-1}_\gamma)} 
\right).
\end{align*}
This completes the proof of~\eqref{e.poincare.mean}. 

\smallskip

\emph{Step 4.} The remaining steps are specific to the case with boundary. To complete the proof of~\eqref{e.poincare.zerobndry}, we must show that, under the additional assumption that~$U\neq \T^d$ and $f \in H^1_{\hyp,0}(U)$, we have 
\begin{equation} 
\label{e.meancontrol}
\left| \left(f\right)_U \right| 
\leq
C\left( 
\left\| \nabla_v f \right\|_{L^2(U;L^2_\gamma)} 
+\left\| v \cdot \nabla_x f \right\|_{L^2(U;H^{-1}_\gamma)}
\right).
\end{equation}
Let $f_1$ be a test function belonging to $C^\infty_c\left(\overline{U} \times \Rd\right)$, to be constructed below, which satisfies the following:
\begin{equation} 
\label{e.wvanishhypbndry}
f_1 = 0 \quad\mbox{on} \ (\partial U \times \Rd) \setminus \partial_{\hyp} (U),
\end{equation}
\begin{equation} 
\label{e.testw.meanone}
\fint_U \int_{\Rd}  v\cdot \nabla_x f_1 \,d\gamma \,dx  = 1
\end{equation}
and, for some constant $C(U,d)<\infty$,
\begin{equation} 
\label{e.testw.L2bounds}
\left\|    v\cdot \nabla_x f_1 \right\|_{L^2(U;L^2_\gamma)} \leq C.
\end{equation}
The test function~$f_1$ is constructed in Step~5 below. We first use it to obtain~\eqref{e.meancontrol}. 
We proceed by using~\eqref{e.testw.meanone} to split the mean of~$f$ as
\begin{equation*}
\left(f\right)_U 
= 
\fint_U \int_{\Rd} f \,  v\cdot \nabla_x f_1  \,d\gamma \,dx
- 
\fint_U \int_{\Rd} \left( f - \left( f \right)_U \right)  v\cdot \nabla_x f_1  \,d\gamma \,dx
\end{equation*}
and estimate the two terms on the right side separately. For the first term, we have
\begin{align*}
\lefteqn{
\left|
\fint_U \int_{\Rd} f \, v\cdot \nabla_x f_1 \,d\gamma \,dx
\right|
} \quad & 
\\ & 
= 
\left| -\fint_U \int_{\Rd}  f_1 \, v\cdot\nabla_x f \,d\gamma \,dx
+ \frac{1}{|U|} \int_{\partial U} \int_{\Rd} (v\cdot \n_U) ff_1 \,d\gamma \,dx \right|
\\ & 
= \left| \fint_U \int_{\Rd}  f_1 v\cdot\nabla_x f   \,d\gamma\,dx \right|,
\end{align*}
where we used that $(v\cdot \n_U) ff_1$ vanishes on $\partial U \times \Rd$ to remove the boundary integral. (Recall that by the definition of $H^1_{\hyp,0}(U)$, we can assume without loss of generality that the function $f$ is smooth, so the justification of the integration by parts above is classical.) We thus obtain that
\begin{align*} \label{}
\left| \fint_U \int_{\Rd}  f_1 v\cdot\nabla_x f   \,d\gamma\,dx \right|
\leq
\frac1{|U|}
 \left\| f_1 \right\|_{L^2(U;H^1_\gamma)} 
\left\| v\cdot \nabla_xf \right\|_{L^2(U;H^{-1}_\gamma)}. 
\end{align*}
This completes the estimate for the first term. For the second term, we use \eqref{e.testw.L2bounds} to get
\begin{align*}
\left|
\fint_U \int_{\Rd} \left( f - \left( f \right)_U \right) \,  v\cdot \nabla_x f_1 \,d\gamma \,dx
\right|
& 
\leq 
\left\| f - \left( f \right)_U \right\|_{L^2(U;L^2_\gamma)} 
\left\|  v\cdot \nabla_x f_1 \right\|_{L^2(U;L^2_\gamma)} 
\\
&
\leq C \left\| f - \left( f \right)_U \right\|_{L^2(U;L^2_\gamma)} ,
\end{align*}
which is estimated using the result of Step~3. Putting these together yields~\eqref{e.meancontrol}. 

\smallskip

\emph{Step 5.} 
We construct the 
test function~$f_1 \in C^\infty_c\left(\overline{U} \times \Rd\right)$
satisfying~\eqref{e.wvanishhypbndry},~\eqref{e.testw.meanone} and~\eqref{e.testw.L2bounds}. Fix $x_0 \in \partial U$ where $\n_U(x_0)$ is well defined. Since the unit normal $\n_U$ is continuous at $x_0$, there exist $v_0 \in \Rd$ and $r > 0$ such that for every $x,v \in \Rd$ satisfying $(x,v) \in (B_r(x_0) \cap \partial U) \times B_r(v_0)$, we have $v \cdot \n_U(x) > 0$. In other words, every $(x,v) \in (B_r(x_0) \cap \partial U) \times B_r(v_0)$ is such that $(x,v) \in \partial_\hyp U$. Observe that, for every $f_1 \in C^\infty_c(\Rd\times \Rd)$, we have
\begin{align*} \label{}
\fint_U \int_{\Rd} v\cdot \nabla_x f_1  \,d\gamma \,dx 
=  \frac{1}{|U|} \int_{\partial U} \int_{\Rd} (v\cdot \n_U) f_1  \,d\gamma \,dx.
\end{align*}
We select a function $f_1 \in C^\infty_c(\Rd\times \Rd)$ with compact support in $B_r(x_0) \times B_r(v_0)$ and such that $f_1 \ge 0$ and $f_1(x_0,v_0) = 1$. 
In this case, the integral on the right side above is nonnegative, since $f_1$ vanishes whenever $v\cdot \n_U \le 0$. In fact, since $f_1$ is positive on a set of positive measure on $\partial U \times \Rd$ (in the sense of the product of the $(d-1)$-dimensional Hausdorff and Lebesgue measures), the integral above is positive. Up to multiplying $f_1$ by a positive scalar if necessary, we can thus ensure that \eqref{e.testw.meanone} holds. It is clear that this construction also ensures that \eqref{e.wvanishhypbndry} and \eqref{e.testw.L2bounds} hold. \end{proof}

\begin{remark}  
\label{r.weaker.boundary.assumption}
As the argument above reveals, for the inequality \eqref{e.poincare.zerobndry} to hold, the assumption of $f \in H^1_{\hyp,0}(U)$ can be weakened: it suffices that $f$ vanishes on a relatively open piece of the boundary $\partial U \times \Rd$. The constant $C$ in \eqref{e.poincare.zerobndry} then depends additionally on the identity of this piece of the boundary where $f$ is assumed to vanish.
\end{remark}

\subsubsection{Poincar{\'e} inequality with confining potential}

It is also interesting to understand Theorem~\ref{t.hypoelliptic.poincare} in the global setting with confining potential.\footnote{A proof is also contained in~\cite{cao2019explicit} following the methods in the original version of this paper, which only discussed bounded domains.}

\emph{Only in this subsection}, we redefine $H^1_{\rm hyp}(\R^d)$ according to the norm
\begin{equation}
\label{eq:redefinedh1hyp}
    \| f \|_{H^1_\hyp(\R^d)} = \left\| f \right\|_{L^2_\sigma(\R^d;H^1_\gamma)} +
\| v \cdot \nabla_x f + \b \cdot \nabla_v f \|_{L^2_\sigma(\R^d;H^{-1}_\gamma)} ,
\end{equation}
and when $\b$ satisfies Assumption~\ref{a.conserv} with $U = \Rd$, and {$f \in L^1_\si(\Rd;L^1_\ga)$,  we use the notation
\begin{equation*}  
(f)_{\R^d} := \int f \, d m .
\end{equation*}
}

\begin{proposition}[Poincar{\'e} with confining potential]
    \label{pro:poincarewithconfining}
Suppose that $\b$ satisfies Assumption~\ref{a.conserv} with $U = \R^d$, {the potential $W$ satisfies $W \in C^{1,1}(\R^d)$}, and that there exists a constant $C_W < \infty$ such that the following weighted Poincar{\'e} inequality holds for all $h \in H^1_\sigma(\R^d)$ with $(h)_{\Rd} = 0$:
\begin{equation}
    \label{eq:Poincareassumption}
   \int_U |\nabla_x W|^2 |h|^2 \, d\sigma \leq C_W \int_U |\nabla_x h|^2 \, d\sigma.
\end{equation}
Then there exists a constant $C(W,d) < \infty$ such that for all $f \in H^1_{\rm hyp}(\R^d)$, defined according to~\eqref{eq:redefinedh1hyp}, with $(f)_{\R^d} = 0$, 
\begin{equation} \notag
\left\| f \right\|_{L^2_\sigma(\R^d;L^2_\gamma)} 
\leq
C \left( \left\| \nabla_v f \right\|_{L^2_\sigma(\R^d;L^2_\gamma)} 
+
\left\| v \cdot \nabla_x f  + \b \cdot \nabla_v f \right\|_{L^2_\sigma(\R^d;H^{-1}_\ga)} \right) \, .
\end{equation}
\end{proposition}

First, we require an analogue of Lemma~\ref{l.est.L2.H-1}.

\begin{lemma}[Auxiliary lemma]
    \label{lem:auxiliaryh1lemma}
Under the assumptions of Proposition~\ref{pro:poincarewithconfining}, there exists $C(W,d) < \infty$ such that for every $h \in L^2_\sigma$,
\begin{equation*}  
\Ll\|h - (h)_{\R^d} \Rr\|_{L^2_\sigma} \le C \|\nabla_x h\|_{H^{-1}_\sigma}.
\end{equation*}
\end{lemma}
\begin{proof}
Without loss of generality, we assume that $(h)_{\R^d} = 0$. Consider the operators
\begin{equation}\notag
    \tilde{A} = \nabla_x, \quad \tilde{A}^* = - \div_x - \b \cdot.
\end{equation}
We consider the problem 
\begin{equation} 
\label{e.boundary.pb1}
\tilde{A}^* \g  = h \quad \text{ in } \R^d,
\end{equation}
where we seek $\g \in H^1_\sigma$. The problem can be solved by defining $\g = \tilde{A} f$ and solving
\begin{equation}
    \label{eq:auxiliaryPDEforf}
    \tilde{A}^* \tilde{A} f = h \quad \text{ in } \R^d
\end{equation}
with $(f)_{\R^d} = 0$. By the Lax-Milgram lemma, there exists a solution $f \in H^1_\sigma$ with $(f)_{\R^d} = 0$ and $\| f \|_{H^1_\sigma} \leq C \| h \|_{H^{-1}_\sigma} \leq C \| h \|_{L^2}$. To demonstrate that $\g \in H^1_\sigma$, we commute a derivative $\p_i$ through~\eqref{eq:auxiliaryPDEforf}:
\begin{equation}
    \label{eq:auxiliaryPDEforpif}
   \tilde{A}^* \tilde{A} \p_i f = -\Delta_x \p_i f - \b \cdot \nabla_x \p_i f  = \p_i h +  \p_i \b \cdot \nabla_x f =: F,
\end{equation}
where $F$ is a forcing term in $H^{-1}_\sigma$.
Clearly, $\| \p_i h \|_{H^{-1}_\sigma} \leq C \| h \|_{L^2_\sigma}$.\footnote{This follows from integration by parts against a test function $g \in H^1_\sigma$ and the Poincar{\'e} inequality in~\eqref{eq:Poincareassumption}, which controls the term $\int \p_i W g h \, d\sigma$ appearing when $\p_i$ hits the weight.} For the commutator term, we have
\begin{equation}\notag
    \| \p_i \b \cdot \nabla_x f \|_{L^2_\sigma} \leq \| \p_i \b \|_{L^\infty} \| \nabla_x f \|_{L^2_\sigma} \leq C \| h \|_{H^{-1}_\sigma},
\end{equation}
where $C$ depends on the $C^{1,1}$ regularity of $W$. By the Lax-Milgram lemma (or energy estimates) applied to~\eqref{eq:auxiliaryPDEforpif} for each $i$, we have that
\begin{equation}
    \label{e.est.bdypb23}
    \| \nabla_x \g \|_{L^2_\sigma} \leq C \| \nabla_x^2 f \|_{L^2_\sigma} \leq C \| F \|_{H^{-1}_\sigma} \leq C \| h \|_{L^2_\sigma}.
\end{equation}
While $\g$ may not have zero average, it was already controlled in $L^2_\sigma$.
Finally, we have
\begin{align*}  
\|h\|_{L^2_\sigma}^2
= \int_{\R^d} h \, \nabla \cdot \g  \, d\sigma
= - \int_{\R^d} \nabla h \cdot \g  \, d\sigma
\leq \|\nabla h\|_{H^{-1}_\sigma} \, \|\g\|_{H^1_\sigma}.
\end{align*}
The conclusion then follows by \eqref{e.est.bdypb23}. \end{proof}

\begin{proof}[Proof of Proposition~\ref{pro:poincarewithconfining}]

Let $f \in H^1_\hyp(\R^d)$, see~\eqref{eq:redefinedh1hyp}. By applying an approximation procedure with smooth cut-off in $x$ and $v$ and mollifying, we can without loss of generality assume that $f$ is a compactly supported, smooth function. Again, we decompose the proof into three steps. Step~1 is identical, so we skip to

\smallskip

\emph{Step 2.} We show that
\begin{equation} 
\label{e.bracketuHminus12}
\left\| \nabla \langle f \rangle_\gamma \right\|_{H^{-1}_\sigma(\R^d)}
\leq
C\left( 
\left\| \nabla_v f \right\|_{L^2_\sigma(\R^d;L^2_\gamma)} +\left\| (v \cdot \nabla_x + \b \cdot \nabla_v ) f \right\|_{L^2_\sigma(\R^d;H^{-1}_\gamma)} 
\right).
\end{equation}
We select~$\xi_1, \ldots, \xi_d \in C^\infty_c(\Rd)$ satisfying
\begin{equation}\notag
\label{e.xi_i.23}
\int_{\Rd} v \xi_i(v)\,d\gamma(v) = e_i \, ,
\end{equation}
and for each test function $\phi\in H^1_\sigma(\R^d)$ and $i \in \{1,\ldots,d\}$, we compute
\begin{equation}
    \label{eq:firstterminpoincareproof}
    \int \phi \p_{x_i} \langle f \rangle_\gamma \, dm = \int \phi \xi_i(v) v \cdot \nabla_x \langle f \rangle_\gamma \, dm  = - \int \phi \xi_i v \cdot \nabla_x (f - \langle f \rangle_\gamma) \, dm  + \int \phi \xi_i v \cdot \nabla_x f \, dm \, .
\end{equation}
We expand the second term on the right-hand side as
\begin{equation}
\label{eq:secondterminpoincareproof}
    \int \phi \xi_i v \cdot \nabla_x f \, dm  = \int \phi \xi_i (v \cdot \nabla_x + \b \cdot \nabla_v) f \, dm  - \int \phi \xi_i \b \cdot \nabla_v (f - \langle f \rangle_\gamma) \, dm \, ,
\end{equation}
where we use that $\b \cdot \nabla_v  \langle f \rangle_\gamma = 0$.
Combining~\eqref{eq:firstterminpoincareproof} and~\eqref{eq:secondterminpoincareproof}, we have
\begin{equation}\notag
    \int \phi \p_{x_i} \langle f \rangle_\gamma \, dm  = \int \phi \xi_i (v \cdot \nabla_x + \b \cdot \nabla_v) f\, dm  - \int \phi \xi_i (v \cdot \nabla_x + \b \cdot \nabla_v )(f - \langle f \rangle_\gamma) \, dm  = {\rm I} + {\rm II}.
\end{equation}
For ${\rm I}$, we have
\begin{equation}\notag
    \left| \int \phi \xi_i (v \cdot \nabla_x + \b \cdot \nabla_v) f\, dm  \right| \leq C \| \phi \xi_i \|_{L^2_\sigma(\R^d;H^1_\gamma)} \| (v \cdot \nabla_x + \b \cdot \nabla_v) f \|_{L^2_\sigma(\R^d;H^{-1}_\gamma)} \, .
\end{equation}
For ${\rm II}$, we integrate by parts across the measure $dm$:
\begin{equation}\notag
    - \int \phi \xi_i (v \cdot \nabla_x + \b \cdot \nabla_v )(f - \langle f \rangle_\gamma) \, dm  = \int v \cdot \nabla_x \phi \xi_i (f - \langle f \rangle_\gamma) \, dm  + \int \phi \b \cdot \nabla_v \xi_i (f - \langle f \rangle_\gamma)\, dm   = {\rm II}_a + {\rm II}_b \, .
\end{equation}
For ${\rm II}_a$, we use
\begin{equation}\notag
    \left| \int v \cdot \nabla_x \phi \xi_i (f - \langle f \rangle_\gamma) \, dm  \right| \leq C \| v \xi_i \|_{L^\infty_\gamma} \| \phi \|_{L^2_\sigma} \| f - \langle f \rangle_\gamma \|_{L^2_\sigma(\R^d;L^2_\gamma)} \, .
\end{equation}
For ${\rm II}_b$, we use
\begin{equation}\notag
    \left| \int \phi \b \cdot \nabla_v \xi_i (f - \langle f \rangle_\gamma) \, dm  \right| \leq \| |\nabla W| |\phi| \|_{L^2_\sigma}  \| \nabla_v \xi_i \|_{L^\infty_\gamma} \| f - \langle f \rangle_\gamma \|_{L^2_\sigma(\R^d;L^2_\gamma)} \, .
\end{equation}
We use the assumed Poincar{\'e} inequality~\eqref{eq:Poincareassumption} to control $\| |\nabla W| |\phi| \|_{L^2_\sigma}$ by $\| \phi \|_{H^1_\sigma}$. Then using \eqref{e.yosccontrol} concludes the proof of \eqref{e.bracketuHminus12}.

\smallskip

\emph{Step 3.} We deduce from Lemma~\ref{lem:auxiliaryh1lemma},~\eqref{e.yosccontrol} and~\eqref{e.bracketuHminus12} that 
\begin{align*} \label{}
\left\| f - \left(f\right)_U \right\|_{L^2_\sigma(\R^d;L^2_\gamma)}
&
\leq
\left\| f - \left\langle f \right\rangle_\gamma \right\|_{L^2_\sigma(\R^d;L^2_\gamma)}
+ \left\| \left\langle f \right\rangle_\gamma - \left( f \right)_{\R^d} \right\|_{L^2_\sigma}
\\ & 
\leq
\left\| f - \left\langle f \right\rangle_\gamma \right\|_{L^2_\sigma(\R^d;L^2_\gamma)}
+ C\left\| \nabla \left\langle f \right\rangle_\gamma  \right\|_{H^{-1}_\sigma}
\\ & 
\leq 
C\left( 
\left\| \nabla_v f \right\|_{L^2_\sigma(\R^d;L^2_\gamma)} +\left\| (v \cdot \nabla_x + \b \cdot \nabla_v) f \right\|_{L^2_\sigma(\R^d;H^{-1}_\gamma)} 
\right).
\end{align*}
This completes the proof. \end{proof}

\subsection{Interpolation and H\"ormander inequalities for \texorpdfstring{$H^1_\hyp$}{H1hyp}}
\label{ss.hormander}

In this subsection, we use the H\"ormander bracket condition to obtain a functional inequality which provides some interior spatial regularity for general $H^1_\hyp$ functions. Both the statement and proof of the inequality follow closely the ideas of H\"ormander~\cite{H}.  Other variants of H\"ormander's inequality have been previously obtained, see in particular~\cite{Bo} and \cite{abn21}. We remind the reader that our initial estimates are phrased in terms of the Besov-type norms defined in subsection~\ref{ss.besov} and are thus valid for $U=\T^d, \R^d$.
\smallskip

\begin{proposition}[Interpolation]\label{lem:interpolation}
For every $\delta>0$, there exists $C(d,\delta) < \infty$ such that for $U= \T^d, \R^d$ and any {smooth function} $u:U\times\R^d\rightarrow\R$, we have
\begin{equation}\label{eq:interpolation}
    \left\| u \right\|_{\Qvx(U)} \leq C \left( \left\| u \right\|_{L^2(U;H^1_\gamma)} + \left\| v\cdot\nabla_x u \right\|_{L^2(U;H^{-1}_\gamma)} \right) + \delta \left\| u \right\|_{\Qx(U)} \, .
\end{equation}
\end{proposition}
\begin{proof}
\textit{Step 1}. Let $\phi \in C^\infty_0((-1,1)^d)$ be a smooth, positive, radial function with unit $L^1$ norm.  For $t\in(0,\infty)$, we define $\phi_{t} u (x,v)$ by 
\begin{equation}\label{eq:mollifier}
\phi_{t} u (x,v) = \int_{\mathbb{R}^d} u (x+t^3 x',v) \phi(x') \, dx' \, , \notag  
\end{equation}
where in the case $U=\T^d$ we have periodicially extended $u$ to a function defined on all of $\R^d$. Using Jensen's inequality, we calculate that 
\begin{align}
   \left\| \phi_{t}u(x,v) - u(x,v) \right\|_{L^2(U;L^2_\gamma)}^2  &= \iint_{\R^d\times U} \left( \int_{\R^d} \phi(x')\left( u(x+t^3x',v) - u(x,v) \right) \,dx' \right)^2 \,dx\,d\gamma(v) \notag\\
   &\leq \iiint_{\R^d\times U\times \R^d} \phi(x')\left( u(x+t^3x',v) - u(x,v) \right)^2 \,dx'\,dx\,d\gamma(v) \notag\\
   & = \iiint_{\R^d\times U\times \R^d} \phi(x') t^2 \frac{1}{t^2} \left( u(x+t^3x',v) - u(x,v) \right)^2 \,dx'\,dx\,d\gamma(v) \notag\\
   & \leq \int_{\R^d} \phi(x') t^2 \left\| u \right\|_{\Qx}^2 \, dx'\, , \notag
\end{align}
and thus we see that
\begin{equation}\label{eq:mollifying:estimate}
    \left\| \phi_{t}u(x,v) - u(x,v) \right\|_{L^2 (U; L^2_\gamma)}^2 \leq t^{2} \left\| u \right\|^2_{\Qx(U)}\, .
\end{equation}
\smallskip

\textit{Step 2}.  Let 
\begin{equation}\notag
f(t) = \left\| u(x+t^2 v, v) - u(x,v) \right\|_{L^2(U; L^2_\gamma)}^2\, .
\end{equation}
For $t\in(0,\infty)$, it will suffice to show that 
\begin{equation}\label{eq:ft}
f(t) \leq t^2 \left( C \left( \left\| u \right\|_{L^2(U;H^1_\gamma)} + \left\| v \cdot\nabla_x u \right\|_{L^2(U;H^{-1}_\gamma)} \right) + \delta \left\| u \right\|_{\Qx(U)} \right)^2 \, .  
\end{equation}
Moreover, for $t \geq 1$, we have the obvious estimate $f(t) \leq 4 \| u \|_{L^2(U;L^2_\gamma)}^2$, so we consider only $t \in (0,1)$.
We may write that
\begin{equation}
    \label{eq:writingallmuhtermsout}
\begin{aligned}
        f(t) &\leq \left\| \phi_{\delta t}u(x+t^2 v,v) - u(x+t^2 v,v) \right\|_{L^2(U;L^2_\gamma)}^2 \\
    & \qquad + \left\| \phi_{\delta t}u(x+t^2 v,v) - \phi_{\delta t} u(x,v) \right\|_{L^2(U;L^2_\gamma)}^2 + \left\| \phi_{\delta t}u(x,v) - u(x,v) \right\|_{L^2(U;L^2_\gamma)}^2 \, . 
\end{aligned}
\end{equation}
By Step 1, the first and third terms of~\eqref{eq:writingallmuhtermsout} are bounded by
$$ \delta^2 t^{2} \left\| u \right\|_{\Qx}^2 \, .$$
\smallskip

\textit{Step 3}.  It remains to estimate the second term in~\eqref{eq:writingallmuhtermsout}. For $t\in(0,1)$ and $0\leq\tau\leq t^2$, consider
\begin{equation}\label{eq:Ftau}
    F(\tau) = \left\| \phi_{\delta t} u(x+\tau v, v) - \phi_{\delta t} u(x,v) \right\|^2_{L^2(U; L^2_\gamma)} \, ,
\end{equation}
where $F(t^2)$ is precisely the second term in~\eqref{eq:writingallmuhtermsout}. Since $F(0) = 0$, it will suffice to show that there exists $C(d,\delta) < \infty$ such that 
$$ F'(\tau) \leq C^2 \left(\left\| u \right\|^2_{L^2(U;H^1_\gamma)} + \left\| v \cdot\nabla_x 
u \right\|^2_{L^2(U;H^{-1}_\gamma)}\right) + \delta^2 \left\| u \right\|_{\Qx}^2 \, . $$
We have that
\begin{align}
    F'(\tau) &= \iint_{\mathbb{R}^d\times U} \left( \phi_{\delta t}u(x+\tau v,v) - \phi_{\delta t} u(x,v) \right) v\cdot \nabla_x \left(\phi_{\delta t} u\right) (x+\tau v,v) \,dx\,d\gamma(v)\notag\\
    &= \iint_{\mathbb{R}^d\times U} \left(\phi_{\delta t}u(x,v) - \phi_{\delta t}u(x-\tau v,v) \right) v\cdot \nabla_x \left(\phi_{\delta t} u\right)(x,v)\, dx\,d\gamma(v)\, . \notag
\end{align}
Since $[v\cdot\nabla_x,\phi_{\delta t}]u = \left[ \nabla_v, \phi_{\delta t} \right]u=0$ and we have a bound on $\left\| v\cdot\nabla_x u \right\|_{L^2(U; H^{-1}_\gamma)}$, we will achieve the desired estimate for $F'(\tau)$ if we can bound
\begin{equation}
    \label{eq:thatthingiwanttobound}
    \left( \phi_{\delta t}u(x,v)-\phi_{\delta t}u(x-\tau v,v) \right)
\end{equation}
in $L^2(U;H^1_\gamma)$. The only non-trivial estimate comes when the $\nabla_v$ lands on the $x$ coordinate of the second term in~\eqref{eq:thatthingiwanttobound}, which we may write out as
\begin{align}
    \int_{\R^d}& -\tau \nabla_x u \left( x+(\delta t)^3 x'-\tau v,v \right) \phi(x')\,dx'\notag\\
    &= - \int_{\R^d} \frac{\tau}{(\delta t)^3 } \nabla_{x'}u(x+(\delta t)^3 x'-\tau v,v) \phi(x') \,dx'\notag\\
    &=\int_{\R^d}  \frac{\tau}{(\delta t)^3 } u(x+(\delta t)^3 x'-\tau v,v) \nabla_{x'} \phi(x') \,dx' \notag\\
    &= \int_{\R^d}  \frac{\tau}{(\delta t)^3} \left(u(x+(\delta t)^3 x'-\tau v,v)-u(x-\tau v,v)\right) \nabla_{x'} \phi(x') \,dx' \, . \notag 
\end{align}
But by Step 1, this is bounded in $L^2(U;\ltwog)$ by a constant multiple of
$$ \frac{\tau}{\delta t^3 } |t| \left\| u \right\|_{\Qx} \leq \frac{1}{\delta^3} \left\| u \right\|_{\Qx}\, , $$
where we have used the assumption that $\tau\leq t^2$. Note that in order to absorb the $\sfrac{1}{\delta^3}$ in the denominator, we may appeal to the Cauchy-Schwarz and Young inequalities in front of $\left\| v\cdot \nabla_x u \right\|_{L^2_x(U;H^{-1}_\gamma)}$, which leads to the estimate \eqref{eq:interpolation} after 
modifying $\delta$ to absorb any implicit constants. \end{proof}

With Proposition~\ref{lem:interpolation} in hand, we can now prove a H\"ormander inequality which provides regularity in the $x$ variable, measured in the $\Qx$ space.  The $H^\alpha$ estimate in Theorem~\ref{t.hormander} for $\alpha<\sfrac{1}{3}$ will be an immediate corollary, and essentially amounts to converting $B^{\sfrac{1}{3}}_{2,\infty}$-type regularity to $B^\alpha_{2,2}$-type regularity.  Following~\cite{H}, the proof of Theorem~\ref{t.hormander} is based on the splitting of a first-order finite difference in the~$x$ variable into finite differences which are either in the~$v$ variable, or in the~$x$ variable in the direction of~$v$. Explicitly, we have
\begin{align}
\label{e.firstordersplitting}
f(x+t^3y,v) - f(x,v) 
& 
= 
f(x+t^3y,v) - f(x+t^3y, v-ty) 
\\ & \quad \notag
+ f(x+t^3y, v-ty) - f(x+t^3y + t^2(v-ty), v-ty)
\\ & \quad \notag
+ f(x+t^2v, v-ty) - f(x+tv, v) 
\\ & \quad \notag
+ f(x+t^2v,v) - f(x,v) \, . 
\end{align}
Notice that the right side consists of four finite differences, two for each of the derivatives $\nabla_v$ and $v\cdot \nabla_x$ which we can expect to control by the $L_2(U;H^1_\gamma)$ and $\Qvx$ norms, respectively. The fact that the increment on the left is of size~$t^3$ and those on the right side are of size~$t$ and $t^2$ suggests that we may expect to have one-third derivative in the statement of Theorem~\ref{t.hormander}, which we are able to obtain in a Besov sense with the $\Qx$ norm.  The exponent $\sfrac{1}{3}$ is optimal, although it may be possible to improve the endpoint regularity from $B^{\sfrac{1}{3}}_{2,\infty}$-type to $B^{\sfrac{1}{3}}_{2,2}$ using more advanced microlocal techniques.

\smallskip

The relation~\eqref{e.firstordersplitting} is a special case of H\"ormander's bracket condition introduced in~\cite{H}, which for the particular equation we consider here is quite simple to check. Indeed, let $X_1,\ldots,X_d$, $V_1,\ldots,V_d$ denote the canonical vector fields and~$X_0$ be the vector field $(x,v) \mapsto (v,0)$. Then the H\"ormander bracket condition is implied by the identity
\begin{equation}
\label{e.hormander.bracket}
[V_i,X_0] = X_i.
\end{equation}
This is a local version of the identity \eqref{e.firstordersplitting}. More precisely, for every vector field $Z$, if we denote by $t \mapsto \exp(tZ)$ the flow induced by the vector field $Z$ on $\R^d \times \Rd$, then 
\begin{multline}  
\label{e.exp.horm}
\exp(-t V_i) \exp(-t X_0) \exp(t V_i) \exp (t X_0) (x,v) 
\\
= (x,v) + t^2 [V_i,X_0] (x,v) + o(t^2) \qquad (t \to 0) \, .
\end{multline}
For the vector fields of interest, $Z \in \{X_0, X_1,\ldots,X_d,V_1,\ldots,V_d\}$, the flows take the very simple form
\begin{equation*}  
\exp(tZ)(x,v) = (x,v) + tZ(x,v),
\end{equation*}
the relation \eqref{e.exp.horm} becomes an identity (that is, the term $o(t^2)$ is actually zero), and {loosely,} this identity can be rephrased in the form of~\eqref{e.firstordersplitting}. {The only difference is that, to exploit that our functions have only $\sfrac{1}{2}$ derivatives in the $v \cdot \nabla_x$ direction, it is advantageous to flow in the direction $v \cdot \nabla_x$ with speed $t$ rather than unit speed.}

\smallskip

\begin{proposition}[Besov-Type H\"ormander Inequality]\label{p.hormander}
There exists a dimensional constant $C(d) < \infty$ such that for $U=\T^d,\R^d$ and any smooth function $u:U\times\R^d\rightarrow\R$, we have the estimate
\begin{equation}\label{eq:hormander:new}
    \left\| u \right\|_{\Qx(U)} \leq C \left( \left\| u \right\|_{\Qvx(U)} + \left\| u \right\|_{L^2_x(U; H^1_\gamma)} \right) \, .
\end{equation}
\end{proposition}
\begin{proof}[{Proof of Proposition~\ref{p.hormander}}]
Let $f(x,v)=u(x,v)\gamma^{\sfrac{1}{2}}(v)$, and choose $\eta\in(0,\infty)$ and $x'\in\mathbb{S}^{d-1}$. Then we may write that
\smallskip
\begin{align}
    \left\| u(x+\eta^3x',v) - u(x,v) \right\|_{L^2(U;L^2_\gamma)} = \left\| f(x+\eta^3x',v) - f(x,v) \right\|_{L^2(U;L^2)} \, , \notag
\end{align}
and
\begin{align}
    f(x+\eta^3 x',v) - f(x,v) &= f(x+\eta^3 x',v) - f(x+\eta^3 x',v-\eta x') \notag\\& \quad
    + f(x+\eta^3 x',v-\eta x') - f(x+\eta^3 x' + \eta^2(v-\eta x'),v-\eta x')\notag\\ & \quad
    +f(x+\eta^2 v,v-\eta x') - f(x+\eta^2 v,v)\notag\\ & \quad
    + f(x+\eta^2 v,v) - f(x,v) \, . \label{e.hormander.string}
\end{align}
Dividing by $\eta$, integrating in $L^2(U;L^2(\R^d))$, and appealing to \eqref{e.weighted.difference} bounds the first term:
\begin{align}
    \frac{1}{\eta^2} \iint_{\R^d\times U} \left(f(x+\eta^3x', v)-f(x+\eta^3x',v-\eta x') \right)^2 \,dx\,dv \leq C \left\| \nabla_v u \right\|_{L^2(U;L^2_\gamma)}^2 \, , \notag
\end{align}
with a similar bound holding for the third term. Dividing again by $\eta$ and integrating in $L^2(U;L^2(\R^d))$ yields the bound
\begin{align}
    \frac{1}{\eta^2} \iint_{\R^d\times U} \left(f(x+\eta^3x', v-\eta x')-f(x+\eta^3x'+\eta^2(v-\eta x'),v-\eta x') \right)^2 \,dx\,dv \leq \left\| u \right\|_{\Qvx(U)}^2 \, , \notag
\end{align}
with a similar bound holding for the fourth term.  Appealing to \eqref{eq:interpolation} with a suitably small choice of $\delta$ concludes the proof. \end{proof}

To obtain the statements in Theorem~\ref{t.hormander} for $\alpha<\sfrac{1}{3}$, we must work in $H^\alpha_x$ rather than $(B^\alpha_{2,\infty})_x$ spaces of fractional differentiability, and so we introduce the Banach space-valued fractional Sobolev spaces, defined as follows: for every domain~$U\subseteq\Rd$,~$\al\in (0,1)$, Banach space~$X$ with norm $\| \!\cdot\! \|_X$ and $u\in L^2(U;X)$, we define the seminorm 
\begin{equation}
\label{e.def.frac.seminorm}
\left\llbracket u \right\rrbracket_{H^{\al}(U;X)}
:=
\left( 
\int_U\int_U
\frac{\left\| u(x)-u(y) \right\|_{X}^2}{|x-y|^{d+2\al}}  \,dx\,dy
\right)^{\sfrac12}
\end{equation}
and the norm
\begin{equation*}
\left\| u \right\|_{H^{\al}(U;X)} 
:= 
\left( \left\| u \right\|_{L^2(U;X)}^2 
+
\left\llbracket u \right\rrbracket_{H^{\al}(U;X)}^2 \right)^{\sfrac12}.
\end{equation*}
We then define the fractional Sobolev space
\begin{equation}
\label{e.def.Halpha}
H^\al(U;X) := \left\{ u\in L^2(U;X) \,:\,  \left\| u \right\|_{H^{\al}(U;X)} < \infty \right\}.
\end{equation}
The space $H^\al(U;X)$ is a Banach space under the norm~$\| \!\cdot\!\|_{H^\al(U;X)}$. We understand that $H^0(U;X) = L^2(U;X)$. We also set
\begin{equation*}  
\|u\|_{H^{1+\alpha}(U;X)} := \left(  \|u\|_{L^2(U;X)}^2 + \|\nabla u\|_{H^\alpha(U;X)}^2 \right)^{\sfrac 12},
\end{equation*}
and define the Banach space $H^{1+\al}(U;X)$ as in \eqref{e.def.Halpha}. We may now use Proposition~\ref{p.hormander} to prove the non-endpoint estimates from Theorem~\ref{t.hormander}.
\begin{proof}[Proof of Theorem~\ref{t.hormander}]
We have that for $\alpha<\sfrac{1}{3}$,
\begin{align}
    \left\llbracket u \right\rrbracket_{H^\alpha(U;L^2_\gamma)}^2  &= \iint_{U\times U} \frac{\left\| u(x,\cdot) - u(y,\cdot) \right\|_{L^2_\gamma}^2}{|x-y|^{d+2\alpha}} \,dx\,dy \notag\\
    &= \iint_{U\times U} \frac{\left\| u(x'+y,\cdot) - u(y,\cdot) \right\|_{L^2_\gamma}^2}{|x'|^{d+2\alpha}} \,dx'\,dy \notag\\
    &= \iint_{\{|x'|<1\}\times U} \frac{\left\| u(x'+y,\cdot) - u(y,\cdot) \right\|_{L^2_\gamma}^2}{|x'|^{d+2\alpha}} \,dy\,dx' \notag\\
    &\qquad + \iint_{\{|x'|\geq 1\}\times U} \frac{\left\| u(x'+y,\cdot) - u(y,\cdot) \right\|_{L^2_\gamma}^2}{|x'|^{d+2\alpha}} \,dy\,dx' \notag\\
    &\leq \int_{\{|x'|<1\}} \frac{|x'|^{\sfrac{2}{3}}\left\| u \right\|_{\Qx}^2}{|x'|^{d+2\alpha}}\, dx' \quad +\quad C(\alpha) \left\| u \right\|^2_{L^2(U;L^2_\gamma)} \notag\\
    &\leq C(\alpha) \left( \left\|  u \right\|_{L^2(U;H^1_\gamma)}^2 + \left\| v\cdot\nabla_x u \right\|_{L^2(U;H^{-1}_\gamma)}^2 \right) \leq C(\alpha) \| u \|_{H^1_\hyp(U)}^2 \, , \notag
\end{align}
concluding the proof. Notice that we have restricted the domain to $U=\R^d,\T^d$ in the above estimates. \end{proof}

For the purposes of interpolation, we also need to consider fractional Sobolev spaces in the velocity variable. As discussed in the arguments leading to \eqref{e.weighted.difference}, the relevant spaces are weighted by the measure $\ga$, which is strongly inhomogeneous.  
Because of this difficulty, we use the following definition. For each $f \in L^2_\ga$ and $t > 0$, we set
\begin{equation*}  
K(t,f) := \inf \Ll\{ \|f_0\|_{L^2_\ga} + t \|f_1\|_{H^1_\ga} \ : \ f = f_0 + f_1, \ f_0 \in L^2_\ga, \ f_1 \in H^1_\ga \Rr\} ,
\end{equation*}
and, for every $\al \in (0,1)$, we define
\begin{equation}
\label{e.def.frac.sob.v}
\|f\|_{H^\al_\ga} := \Ll( \int_0^\infty \Ll( t^{-\al} K(f,t) \Rr) ^2 \, \frac {dt}{t} \Rr)^\frac 1 2.
\end{equation}
We also define $H^{-\al}_\ga$ to be the space dual to $H^{\al}_\ga$. 

\smallskip

We may utilize interpolation to obtain embeddings into other similar spaces of positive regularity in both variables. In particular, appealing to Theorem~\ref{t.hormander} and the interpolation inequality
\begin{equation*} \label{}
\left\| f \right\|_{H^{\theta \beta}(U;H^{1-\theta}_\gamma)} 
\leq 
\left\| f \right\|_{H^{\beta}(U;L^{2}_\gamma)}^{\theta}
\left\| f \right\|_{L^2(U;H^{1}_\gamma)}^{1-\theta} \qquad \theta\in[0,1]\, \qquad U=\T^d,\R^d \, ,
\end{equation*}
immediately implies the following estimate. 

\begin{corollary}
[{H\"ormander inequality for $H^1_\hyp$}]
\label{c.sobolev.16}
Let $\alpha \in \left[ 0,\tfrac 13 \right)$ and $U=\T^d,\R^d$. There exists a constant~$C(\alpha,d)<\infty$ such that for every~$\theta \in [0,1]$ and every~$f\in H^1_\hyp(U)$, we have the estimate 
\begin{equation*}
\left\| f \right\|_{H^{\theta \alpha}(U;H^{1-\theta}_\gamma)}
\leq
C \left\| f \right\|_{H^1_\hyp(U)} \, .
\end{equation*}
\end{corollary}

Observe that, by introducing a cutoff function in the spatial variable, we also obtain analogous embeddings for bounded domains $U\subset \R^d$, such as 
\begin{equation*}
H^1_\hyp(U) \hookrightarrow H^{\alpha}(U_\delta;L^2_\gamma) \, ,
\end{equation*}
valid for every $\alpha <\frac13$ and $\delta>0$, where $U_\delta:= \left\{ x\in U\,:\,\dist(x,\partial U) > \delta \right\}$. 

\subsection{Compact embedding of \texorpdfstring{$H^1_\hyp$}{H1hyp} into \texorpdfstring{$L^2(U;L^2_\ga)$}{L2}}
Using the results of the previous subsection, we show that the embedding $H^1_{\hyp}(U) \hookrightarrow L^2(U;L^2_\gamma)$ is compact. In this section, we assume that $U\subseteq\Rd$ is a bounded $C^1$ domain or $\T^d$. 
\smallskip

\begin{proposition}[{Compact embedding of $H^1_{\hyp}(U)$ into $L^2(U;L^2_\gamma)$}] 
The inclusion map $H^1_{\hyp}(U) \hookrightarrow L^2(U;L^2_\gamma)$ is compact. 
\label{p.compactembed}
\end{proposition}

The proof is straightforward on $\T^d$. First, approximate by functions in $C^\infty_0(\T^d \times \R^d)$. Next, we use the embedding $H^1_{\rm hyp}(\T^d) \subset H^\alpha(\T^d \times B_{v_0})$ for all $v_0 \in [1,+\infty)$. Finally, we apply the standard Rellich compactness theorem. Hence, we focus only on bounded $C^1$ domains $U \subset \R^d$ below.

\smallskip

Before we give the proof of Proposition~\ref{p.compactembed}, we need to review some basic facts concerning the logarithmic Sobolev inequality and a generalized H\"older inequality for Orlicz norms. The logarithmic Sobolev inequality states that, for some~$C<\infty$, 
\begin{equation}
\label{e.LSI}
\int_{\Rd} f^2(v) \log\left( 1+ f^2(v) \right) \,d\gamma(v) 
\leq C \int_{\Rd} \left| \nabla f \right|^2\,d\gamma(v),
\quad \forall f\in H^1_\gamma\ \mbox{with} \ \left\| f \right\|_{L^2_\gamma} =1.   
\end{equation}
Let $F:\R \to [0,\infty)$ denote the (strictly) convex function
\begin{equation*}
F(t) := |t| \log \left( 1 + |t| \right). 
\end{equation*}
Let $F^*$ denote its dual convex conjugate function, defined by
\begin{equation*}
F^*(s):= \sup_{t\in\R} \left( st - F(t) \right). 
\end{equation*}
Then~$(F,F^*)$ is a \emph{Young pair} (see~\cite{RR}), that is, both~$F$ and~$F^*$ are nonnegative, even, convex, satisfy~$F(0)=F^*(0)=0$ as well as
\begin{equation*}
\lim_{|t| \to\infty} |t|^{-1}F(t) = \lim_{|s|\to \infty} |s|^{-1} F^*(s) = \infty.
\end{equation*}
Moreover, both~$F$ and~$F^*$ are strictly increasing on~$[0,\infty)$ and in particular vanish only at~$t=0$. Given any measure space $(X,\omega)$, the \emph{Orcliz spaces}~$L_{F}(X,\omega)$ and $L_{F^*}(X,\omega)$, which are defined by the norms
\begin{equation*}
\left\{
\begin{aligned}
& \left\| g \right\|_{L_{F}(X,\omega)} 
:=
\inf \left\{ t>0\,:\, \int_X F\left( t^{-1}g \right)\,d\omega \leq F(1) \right\}, 
\quad \mbox{and} \\
& \left\| g \right\|_{L_{F^*}(X,\omega)} 
:=
\inf \left\{ t>0\,:\, \int_{X} F^*\left( t^{-1}g \right)\,d\omega \leq F^*(1) \right\}
\end{aligned}
\right.
\end{equation*}
are dual Banach spaces and the following generalized version of the H\"older inequality is valid (see~\cite[Proposition 3.3.1]{RR}):
\begin{equation*}
\int_{X} \left| gg^* \right| \,d\omega
\leq 
\left\| g \right\|_{L_{F}(X,\omega)} \left\| g^* \right\|_{L_{F^*}(X,\omega)}, 
\quad \forall g\in L_{F}(X,\omega),\, g^* \in L_{F^*}(X,\omega).
\end{equation*}
The logarithmic Sobolev inequality~\eqref{e.LSI} may be written in terms of the Orcliz norm as
\begin{equation*}
\left\| f^2 \right\|_{L_{F}(\Rd,\gamma)}
\leq 
C\left( \left| \langle f \rangle_\gamma \right|^2 + \left\| \nabla f \right\|_{L^2_\gamma}^2\right), 
\quad \forall f\in H^1_\gamma .
\end{equation*}
The previous two displays imply that 
\begin{align}
\label{e.HolderandLSI}
\left( \int_{U\times \Rd}
g
\left| f \right|^2
\,dx\,d\gamma(v) \right)^{\frac12}
& 
\leq 
C \left\| g \right\|_{L_{F^*}(U\times \Rd, dx  d\gamma)}^{\frac12} 
\left\| f \right\|_{L^2(U;H^1_\gamma)}.
\end{align}
We do not identify $F^*$ with an explicit formula, although we notice that the inequality 
\begin{equation*}
s(t+1) \leq \exp(s) + t \log (1+t), \quad \forall s,t\in(0,\infty) 
\end{equation*}
implies that 
\begin{equation*}
F^*(s) \leq \exp(s) - s.
\end{equation*}
This allows us in particular to obtain from~\eqref{e.HolderandLSI} that 
\begin{equation}
\label{e.extrav2integrab}
\left( \int_{U\times \Rd}
|v|^2 
\left| f \right|^2
\,dx\,d\gamma(v) \right)^{\frac12}
\leq 
C \left\| f \right\|_{L^2(U;H^1_\gamma)}.
\end{equation}
We also point out that~\eqref{e.extrav2integrab} also implies the existence of $C(d,U)<\infty$ such that, for every $f\in L^2(U;L^2_\gamma)$, 
\begin{equation}
\label{e.Hm1.isgood}
\left\| \nabla_v f \right\|_{L^2(U;H^{-1}_\gamma)} 
\leq
C \left\| f \right\|_{L^2(U;L^2_\gamma)}. 
\end{equation}

\smallskip

We now turn to the proof of Proposition~\ref{p.compactembed}.

\begin{proof}[{Proof of Proposition~\ref{p.compactembed}}]
For each $\theta > 0$, we denote 
\begin{equation}
\label{e.def.Utheta}
U_\theta:=\left\{ x \,:\, \dist(x,\partial U) < \theta \right\}.
\end{equation}
Since~$U$ is a~$C^1$ domain, we can extend the outer normal~$\n_U$ to a globally $C^0$ function on~$\overline{U}$. We can moreover assume that, for some $\theta_0(U) > 0$, this extension $\n_U$ coincides with the gradient of the mapping $x \mapsto - \dist(x,\partial U)$ in $U_{\theta_0}$.

\smallskip

By Proposition~\ref{p.density}, we may work under the qualitative assumption that all of our~$H^1_\hyp(U)$ functions belong to~$C^\infty_c(\overline{U}\times\Rd)$. Select $\ep>0$ and a sequence $\{ f_n \}_{n\in\N} \subseteq H^1_{\hyp}(U)$ satisfying 
\begin{equation*} \label{}
\sup_{n\in\N} \left\| f_n \right\|_{H^1_\hyp(U)} \leq 1.
\end{equation*}
We will argue that there exists a subsequence $\{ f_{n_k} \}$ such that 
\begin{equation} 
\label{e.embedwts}
\limsup_{k\to \infty} \sup_{i,j\geq k} \left\| f_{n_i} - f_{n_j} \right\|_{L^2(U;L^2_\gamma)} \leq \ep \, . 
\end{equation}
The proposition may then be obtained by a diagonalization argument.

\smallskip

\emph{Step 1.} 
We claim that there exists~${v_0}\in [1,\infty)$ such that, for every $f\in H^1_{\hyp}(U)$,  
\begin{equation*} \label{}
\left( \int_{U} \int_{\Rd \setminus B_{v_0}}  \left|f(x,v) \right|^2 \, dx\, d\gamma(v) \right)^{\frac12}
\leq 
\frac{\ep}{3} \left\| f \right\|_{H^1_\hyp(U)} \, .
\end{equation*}
Indeed, applying~\eqref{e.HolderandLSI}, we find that 
\begin{equation*} \label{}
\left( \int_{U} \int_{\Rd \setminus B_{v_0}}  \left|f(x,v) \right|^2 \, dx\, d\gamma(v) \right)^{\frac12}
\leq 
C  \left\| \indc_{U \times(\Rd\setminus {v_0})} \right\|_{L_{F^*}(U\times \Rd, dx  d\gamma)}^{\frac12}  \left\| f \right\|_{H^1_\hyp(U)}.
\end{equation*}
Taking ${v_0}$ sufficiently large, depending on~$\ep$, ensures that
\begin{equation*}
C  \left\| \indc_{U \times(\Rd\setminus {v_0})} \right\|_{L_{F^*}(U\times \Rd, dx  d\gamma)}^{\frac12} \leq \frac\ep3.
\end{equation*}

\smallskip

\emph{Step 2.} We next claim that there exists $\delta \in \left(0,\tfrac12\right]$ such that, for every $f\in H^1_{\hyp}(U)$,  
\begin{equation*} \label{}
\left( \int_{U} \int_{\Rd }  \left|f(x,v) \right|^2 \indc_{\left\{ \left| \n_U\cdot v \right| < \delta \right\}} \, dx\, d\gamma(v) \right)^{\frac12}
\leq 
\frac{\ep}{3} \left\| f \right\|_{H^1_\hyp(U)}.
\end{equation*}
The argument here is similar to the estimate in Step~1, above. We simply apply~\eqref{e.HolderandLSI} after choosing $\delta$ small enough that 
\begin{equation*}
C  \left\| \indc_{\left\{ \left| \n_U\cdot v \right| < \delta \right\}} \right\|_{L_{F^*}(U\times \Rd, dx  d\gamma)}^{\frac12} 
\leq 
\frac\ep3. 
\end{equation*}

\smallskip

\emph{Step 3.} We next show that, for every~$\delta>0$,  there exists~$\theta > 0$ such that, for every function $f\in H^1_{\hyp}(U)$,  
\begin{equation}
\label{e.channelflush}
\left( \int_{U} \int_{\Rd}  \left|f(x,v) \right|^2 \indc_{\left\{ \left| \n_U\cdot v \right| \geq \delta \right\}} \indc_{\left\{ \dist(x,\partial U) < \theta \right\}} \, dx\, d\gamma(v) \right)^{\frac12}
\leq 
\frac{\ep}{3} \left\| f \right\|_{H^1_\hyp(U)}.
\end{equation}
For $\theta \in \Ll( 0,\frac{\theta_0}{2}\Rr]$ to be taken sufficiently small in terms of $\delta > 0$ in the course of the argument, we let $\varphi\in C^{1,1}(\overline{U})$ be defined by
\begin{equation*}
\varphi(x) := - \eta \left( \dist(x,\partial U) \right), 
\end{equation*}
where $\eta\in C^\infty_c([0,\infty))$ satisfies 
\begin{equation*}
0\leq \eta \leq 2\theta, \quad
0 \leq \eta' \leq 1, \quad 
\eta(x) = x \ \mbox{on} \ \left[0,\theta \right], \quad
\eta' = 0 ,  \  \mbox{on} \ [2\theta,\infty).
\end{equation*}
We have $-2\theta \leq \varphi \leq 0$. Moreover, by the definition of $\theta_0$ below \eqref{e.def.Utheta},  its gradient~$\nabla \varphi$ is proportional to~$\n_U$ in $U$, it vanishes outside of $U_{2\theta}$, and $\nabla \varphi = \n_U$ in $U_\theta$. 
We next select another test function $\chi\in C^{\infty}_c([0,\infty))$ satisfying
\begin{equation*}
0\leq \chi \leq 1, \quad
\chi \equiv 0 \ \mbox{on} \ \left[0,\tfrac12\delta\right], \quad
\chi\equiv 1 \  \mbox{on} \ [\delta,\infty), \quad 
|\chi'|  \leq  \delta^{-1},
\end{equation*}
and define 
\begin{equation*}
\psi_{\pm}(x,v) := \chi \left( \left( v \cdot \n_U(x) \right)_\pm \right),
\end{equation*}
where for $r\in\R$, we use the notation $r_- := \max(0,-r)$ and $r_+ := \max(0,r)$. Observe that 
\begin{equation*}
\left| \nabla_v \psi_{\pm}(x,v) \right| 
=
\left| \chi' \left( \left( v \cdot \n_U(x) \right)_\pm \right) \right| \left| \n_U(x) \right|
 \leq
 C\delta^{-1}.
\end{equation*}
Therefore
\begin{align*}
\left\| \varphi f \psi_{\pm} \right\|_{L^2(U;H^1_\gamma)} 
&
\leq
C\left( \left\| \varphi f \psi_{\pm} \right\|_{L^2(U;L^2_\gamma)} 
+
\left\| \varphi \nabla_v \left( f \psi_{\pm} \right) \right\|_{L^2(U;L^2_\gamma)} \right)
\\ & 
\leq 
C \theta \left(
\left\| f \right\|_{L^2(U;L^2_\gamma)} 
+\left\| \nabla_v f  \right\|_{L^2(U;L^2_\gamma)} 
+\left\| f \nabla_v \psi_{\pm} \right\|_{L^2(U;L^2_\gamma)} 
\right)
\\ & 
\leq C\theta\delta^{-1} \left\| f \right\|_{L^2(U;H^1_\gamma)},
\end{align*}
and hence
\begin{equation*}
\left| \int_{U \times\Rd} \varphi f \psi_{\pm} v\cdot \nabla_x f \,dx \,d\gamma(v) \right| \leq 
C \theta\delta^{-1} \left\| f \right\|_{H^1_\hyp(U)}^2. 
\end{equation*}
On the other hand, 
\begin{align*}
\lefteqn{
\int_{U \times\Rd} \varphi f \psi_{\pm} v\cdot \nabla_x f \,dx \,d\gamma(v)
} \qquad & 
\\ &
=
- \frac12 
\int_{U \times\Rd} f^2 v\cdot \nabla_x \left( \varphi \psi_{\pm} \right) \,dx \,d\gamma(v)
\\ & 
= 
- \frac12 
\int_{U \times\Rd} \varphi f^2 v\cdot \nabla_x \psi_{\pm}  \,dx \,d\gamma(v)
-
\frac12 
\int_{U \times\Rd} \psi_{\pm} f^2 v\cdot \nabla\varphi \,dx \,d\gamma(v).
\end{align*}
Since $\left| v\cdot \nabla_x \psi_{\pm}(x,v) \right| 
\leq C\delta^{-1} |v|^2$, we have, by~\eqref{e.extrav2integrab}, 
\begin{align*}
\left| \int_{U \times\Rd} \varphi f^2 v\cdot \nabla_x \psi_{\pm}  \,dx \,d\gamma(v) \right| 
\leq 
C\theta\delta^{-1}\int_{U \times\Rd} |v|^2 f^2 \,dx \,d\gamma(v) 
\leq 
C\theta\delta^{-1} \left\| f \right\|_{H^1_\hyp(U)}^2. 
\end{align*}
We deduce that 
\begin{equation*}
\left| \int_{U \times\Rd} \psi_{\pm} f^2 v\cdot \nabla\varphi \,dx \,d\gamma(v) \right| 
\leq 
C\theta\delta^{-1} \left\| f \right\|_{H^1_\hyp(U)}^2.
\end{equation*}
Finally, we observe from the properties of $\varphi$ and $\psi_{\pm}$ that  
\begin{align*}
\lefteqn{
\int_{U} \int_{\Rd}  \left|f(x,v) \right|^2 \indc_{\left\{ \left| \n_U\cdot v \right| \geq \delta \right\}} \indc_{\left\{ \dist(x,\partial U) < \theta \right\}} \, dx\, d\gamma(v)
}
\qquad & 
\\ & 
\leq 
\delta^{-1} \left(  \left| \int_{U \times\Rd} \psi_{+} f^2 v\cdot \nabla\varphi \,dx \,d\gamma(v) \right| 
+ 
\left| \int_{U \times\Rd} \psi_{-} f^2 v\cdot \nabla\varphi \,dx \,d\gamma(v) \right| \right)
\\ & 
\leq C \theta \delta^{-2} \left\| f \right\|_{H^1_\hyp(U)}^2. 
\end{align*}
Taking $\theta = c \ep^2 \delta^2$ for a sufficiently small constant $c>0$ yields the claimed inequality~\eqref{e.channelflush}. 

\smallskip

\emph{Step 4.} By the results of the previous three steps, to obtain~\eqref{e.embedwts} it suffices to exhibit a subsequence $\{ f_{n_k} \}$ satisfying
\begin{equation*} \label{}
\limsup_{k\to \infty} \sup_{i,j\geq k} 
\int_{U_\theta \times B_{v_0}} \left|  f_{n_i} - f_{n_j} \right|^2 \,dx\,d\gamma(v)  
=0. 
\end{equation*}
This is an immediate consequence of  Corollary~\ref{c.sobolev.16} and the compactness of the embedding $H^{\sfrac1{10}}\left(U_\theta;H^{\sfrac13}_\gamma\right) \hookrightarrow L^2(U_\theta;L^2_\gamma(B_{v_0}))$ (see for instance \cite[Theorem~2.32]{AF}). \end{proof}

\section{The Kramers equation}
\label{s.wellpose}

In this section, we present two proofs of the existence of weak solutions in $H^1_\hyp(\T^d)$ to the Kramers equation
\begin{equation}\label{e.kramers.vproof}
 -\Delta_v f + v\cdot \nabla_v f + v\cdot\nabla_x f + \b\cdot\nabla_v f = g^* \, , 
\end{equation}
where $g^*\in L^2(\T^d; H^{-1}_\gamma)$ {satisfies $\iint_{\T^d\times\R^d} g^* \, dm =0$} (recall that the weighted mean of $g^*$ is well defined by duality since the function $1$ belongs to $L^2(\T^d;H^1_\gamma)$). The first proof uses the abstract Lions-Lax-Milgram theorem and a modification of~\eqref{e.kramers.vproof} with a penalization term $\nu f$. The hypoelliptic energy estimates are used in sending the parameter $\nu$ to zero. This approach is partly inspired by~\cite{Car}. The second proof uses a dual variational approach which characterizes the weak solutions of \eqref{e.kramers.vproof} as the minimizers of a natural energy under an appropriate constraint, in analogy with the discussion following the statement of Theorem~\ref{t.hypo.DP}. In both cases, the Poincar\'e inequality from Theorem~\ref{t.hypoelliptic.poincare} provides the necessary coercivity.

Throughout this section, the force field $\b(x) = -\nabla W(x)$ is as in Assumption~\ref{a.conserv}. In particular, $\b$ depends only on~$x$ and is conservative. Let $dm$ be as defined in \eqref{e.def.m}.

\subsection{The Lions-Lax-Milgram approach}
We recall the abstract version of Lions' representation theorem from
\cite[Theorem 3.1, p. 109]{Showalter}.

\begin{lemma}[Lions' representation theorem]
\label{lem:lionslaxmilgram}
Let $H$ be a Hilbert space and $\Phi$ a pre-Hilbert space. Let $E : H \times \Phi \to \R$ be a bilinear form satisfying the continuity criterion
\begin{equation}
    \label{eq:bilinearformcont}
    E(\cdot,\phi) \in H^* \text{ for all } \phi \in \Phi \, .
\end{equation}
Then the following two properties are equivalent:
\begin{itemize}
    \item (Coercivity) We have
\begin{equation}
    \label{eq:bilinearformcoer}
    \inf_{\| \phi \|_{\Phi} = 1} \sup_{\| h \|_{H} \leq 1} |E(h,\phi)| \geq c > 0 \, .
\end{equation}
\item (Solvability)
For each $L \in \Phi^*$, there exists $f \in H$ such that
\begin{equation}
\label{eq:weakformoftheproblem}
    E(f,\phi) = L(\phi), \quad \text{ for all } \phi \in \Phi \, .
\end{equation}
\end{itemize}
\end{lemma}

Notice that uniqueness and stability estimates are not guaranteed by Lemma~\ref{lem:lionslaxmilgram} itself; they are concluded {a posteriori}.

\begin{proof}[Proof of Theorem~\ref{t.hypo.DP}]
We split the argument into steps; in the first step, we solve a penalized problem, and in the second, we send the penalization parameter $\nu$ to zero. 
\smallskip

\emph{Step 1}. Consider the penalized problem
\begin{equation}
    \label{eq:penalizedPDE}
    (v \cdot \nabla_x + \b \cdot \nabla_v)f + \nu f = g^* + \Delta f - v \cdot \nabla_v f \, 
\end{equation}
posed on the torus $\T^d$ where $\nu \in (0,1]$. We define the following objects:
\begin{enumerate}
\item the \emph{test function space}
\begin{equation}\notag
    \Phi = C^\infty_0(\T^d \times \R^d)
\end{equation}
with inner product
\begin{equation}
    \label{eq:sameinnerproduct}
    (\phi,\psi) = \iint_{\T^d\times\R^d} \nabla_v \phi \cdot \nabla_v \psi \, dm + \iint_{\T^d\times\R^d} \phi \psi  \, dm \, ,
\end{equation}
    \item the \emph{solution space}
\begin{equation}\notag
    H = \{ h \in L^2_\sigma(\T^d;H^1_\gamma) : (h)_{\T^d} = 0 \}
\end{equation}
with inner product~\eqref{eq:sameinnerproduct},
\item the penalized \emph{bilinear form}
\begin{equation}\notag
    E(h,\phi) = \iint_{\T^d\times\R^d} \nabla_v h \cdot \nabla_v \phi \, dm + \nu \iint_{\T^d\times\R^d} h\phi \, dm - \iint_{\T^d\times\R^d} h (v \cdot \nabla_x + \b \cdot \nabla_v) \phi \, dm \, ,
\end{equation}
\item and the \emph{linear functional}
\begin{equation}\notag
    L = g^* \in L^2_\sigma(\T^d;H^{-1}_\gamma) \text{ with } (g^*)_{\T^d} = 0 \, .
\end{equation}
\end{enumerate}
It is not difficult to verify that $E$ is continuous~\eqref{eq:bilinearformcont} and coercive~\eqref{eq:bilinearformcoer}. Indeed, the key features are that (i) the anti-symmetric operator $v \cdot \nabla_x + \b \cdot \nabla_v$ hits the test function $\phi$, and (ii) the penalization term $\nu \iint_{\T^d\times\R^d} |\phi|^2 \, dm$ controls the `lower part' ($L^2(\T^d;L^2_\gamma)$) of the norm after testing with $\phi$. Hence, Lemma~\ref{lem:lionslaxmilgram} guarantees the existence of a solution $f \in H$ to~\eqref{eq:weakformoftheproblem}, which is the distributional formulation of the penalized equation~\eqref{eq:penalizedPDE}.

From the equation itself, we recover that $(v \cdot \nabla_x + \b \cdot \nabla_v) f \in L^2_\sigma(\T^d;H^{-1}_\gamma)$, and therefore, $f \in H^1_{\rm hyp}(\T^d)$ qualitatively. By the density of smooth functions in $H^1_\hyp(\T^d)$, this is enough regularity\footnote{To justify this, one may use the density of test functions demonstrated in Proposition~\ref{p.density}.} to  multiply~\eqref{eq:penalizedPDE} by $f$ and integrate by parts to demonstrate the basic energy estimate:
\begin{equation}
    \label{eq:basicpenalizedenergyestimate}
     \iint_{\T^d\times\R^d} |\nabla_v f|^2 \, dm + \nu \iint_{\T^d\times\R^d} |f|^2 \, dm \leq C \nu^{-1} \| g^* \|_{L^2_\sigma(\T^d;H^{-1}_\gamma)}^2 \, , 
\end{equation}
which guarantees that the solution is unique.\footnote{The estimate~\eqref{eq:basicpenalizedenergyestimate} can be made more convenient, without the factor $\nu^{-1}$, if $\la g^* \ra_\gamma \equiv 0$.} From the equation itself, we have
\begin{align}
    \| (v \cdot \nabla_x + \b \cdot \nabla_v) f \|_{L^2_\sigma(\T^d;H^{-1}_\gamma)} &\leq C  \| A^* A f \|_{L^2_\sigma(\T^d;H^{-1}_\gamma)} + \left\| g^* \right\|_{L^2(\T^d;H^{-1}_\gamma)} + C \nu \| f \|_{L^2_\sigma(\T^d;H^{-1}_\gamma)} \notag\\
    &\overset{\eqref{eq:basicpenalizedenergyestimate}}{\leq} C  \| g^* \|_{L^2_\sigma(\T^d;H^{-1}_\gamma)} \, , \label{eq:basicpenalizedenergyestimate2}
\end{align}
where the constant $C$ changes from line to line. Then~\eqref{eq:basicpenalizedenergyestimate},~\eqref{eq:basicpenalizedenergyestimate2}, and the hypoelliptic Poincar{\'e} inequality for mean-zero functions imply that
\begin{equation}\notag
    \| f \|_{H^{1}_{\rm hyp}(\T^d)} \leq C \| g^* \|_{L^2_\sigma(\T^d;H^{-1}_\gamma)} \, .
\end{equation}
\smallskip

\emph{Step 2}. Next, we consider $\nu \to 0^+$. Let $f^\nu$ denote the unique solution of the penalized problem~\eqref{eq:penalizedPDE}. Subtracting two solutions $f^{\nu_1}$ and $f^{\nu_2}$, we have that the difference $\tilde{f}^{\nu_1,\nu_2}$ solves the equation
\begin{equation}
    \label{eq:penalizedPDEdifference}
    (v \cdot \nabla_x + \b \cdot \nabla_v)\tilde{f}^{\nu_1,\nu_2} + (\nu_1 f^{\nu_1} - \nu_2 f^{\nu_2}) = (\Delta - v \cdot \nabla_v) \tilde{f}^{\nu_1,\nu_2}.
\end{equation}
We may regard $\nu_1 f^{\nu_1} - \nu_2 f^{\nu_2}$ as a forcing term which is $O(\nu_1 + \nu_2)$ in $L^2_\sigma(\T^d;H^{-1}_\gamma)$. By the hypoelliptic energy estimates for~\eqref{eq:penalizedPDEdifference}, we have
\begin{equation}\notag
    \| \tilde{f}^{\nu_1,\nu_2} \|_{H^{1}_{\rm hyp}(\T^d)} = O(\nu_1 + \nu_2) \, .
\end{equation}
Choosing $\nu = 2^{-k}$, the sequence $(f_k)$ of solutions 
to~\eqref{eq:penalizedPDE} with penalization $\nu = 2^{-k}$ is Cauchy in $H^1_{\rm hyp}(\T^d)$ and therefore converges to a solution $f$ in $H^1_{\rm hyp}(\T^d)$ with $(f)_{\T^d} = 0$. By passing to the distributional limit in each term in~\eqref{eq:penalizedPDE}, we find that $f$ solves~\eqref{e.kramers.vproof} in the sense of distributions. The proof is complete. \end{proof}

\begin{remark}[Role of the penalization]
    \label{rmk:roleofpenalization}
    The above proof requires a \emph{coercive} bilinear form $E$ which, in particular, controls the $L^2$ norm. The {a priori} estimates for solutions of~\eqref{e.kramers.vproof} do indeed control the $L^2$ part of the norm through the hypoelliptic Poincar{\'e} inequality, but the control of $\| (v \cdot \nabla_x + \b \cdot \nabla_v ) f \|_{L^2(\T^d;H^{-1}_\gamma)}$ is encoded by the PDE itself rather than the bilinear form $E$, which only encodes the energy estimate. This is why we include the penalization $\nu f$. In some sense, control of  $\| (v \cdot \nabla_x + \b \cdot \nabla_v ) f \|_{L^2(\T^d;H^{-1}_\gamma)}$ is concluded {a posteriori}.
    
    In the time-dependent case, one can skip the penalization by instead considering the equation satisfied by $e^t f$; see Proposition~\ref{pro:solvabilitykinetic}.
\end{remark}

\begin{remark}[Difficulty with boundary]
    \label{rmk:difficultywithboundary}
    Consider~\eqref{e.kramers.vproof} in a bounded $C^1$ domain $U$ with force $f^*$ and zero Dirichlet condition on $\p_\hyp U$. What goes wrong with the proof? One can demonstrate that there exists a solution $f^\nu \in H^1_{\rm hyp}(U)$ of the penalized equations which satisfies $f^\nu|_{\p_\hyp U} = 0$ \emph{away from the singular set}. However, we do not know how to justify that $f^\nu \in H^1_{\hyp,0}(U)$. That is, we cannot characterize $H^1_{\hyp,0}(U)$ as consisting of $H^1_{\hyp}(U)$ functions which vanish on $\p_\hyp U$ away from the singular set. Consequently, we cannot justify the integration by parts that would generate the energy estimates that would imply uniqueness of $f^\nu$ and allow us to send $\nu \to 0^+$.
\end{remark}

\subsection{The dual variational approach}

Define
\begin{equation}\label{e.def.B}
    Bf := v \cdot \nabla_x f + \b\cdot\nabla_v f \, .
\end{equation}

\newcommand{\ff}{\mathbf{f}}

\newcommand{\JJ}{\mathcal{J}}
\newcommand{\jj}{\mathbf{j}}

Consider the functional 
\begin{equation}\label{e.convex.mapping}
 \JJ[f,\jj] = \iint_{\T^d\times\R^d} \frac{1}{2}\left| \nabla_v f - \jj \right|^2 \,{d\sigma(x)}\,d\gamma(v) \,   
\end{equation}
evaluated at pairs $(f,\jj)\in H^1_\hyp(\T^d) \times \left(L^2(\T^d;L^2_\gamma)\right)^d$ satisfying
\begin{equation}\label{e.kv.constraint}
\nabla_v^*\jj= g^* - Bf = g^* - (v\cdot\nabla_x f + \b\cdot\nabla_v f) \, , \qquad (f)_{\T^d} = 0  \, . 
\end{equation}
In the remainder of this section, we \emph{always} consider $f\in H^1_\hyp(\T^d)$ satisfying the second condition. We seek a null minimizer of $\JJ$ restricted to such pairs, which, if it exists, will satisfy the implication
$$ \nabla_v f =\jj\implies \nabla_v^*\jj= \nabla_v^* \nabla_v f = g^* - Bf \, ,  $$
which is precisely \eqref{e.kramers.vproof}. 

\begin{proposition}[Solvability of the Kramers equation]
    \label{pro:solvabilitykramers}
Under Assumption~\ref{a.conserv} and the assumption that
$$  \iint_{\T^d\times\R^d} g^* \,d\gamma(v)\,d\sigma(x) = 0 \, ,  $$
there exists a unique solution $f$ to \eqref{e.kramers.vproof} such that $(f)_{\T^d}=0$, and $f$ is given as the null minimizer of the functional $\JJ[f,\jj]$ over pairs $(f,\jj)$ satisfying the constraint \eqref{e.kv.constraint}.
\end{proposition}

Before proving Proposition~\ref{pro:solvabilitykramers}, we argue that one may assume that $\langle g^* \rangle_\gamma=0$ as a function of $x$. For this, we require

\begin{lemma}\label{l.identifyZ}
Let $h\in L^2(\T^d)$ be given with $(h)_{\T^d}:=\int_{\T^d}h(x) \,d\sigma(x) = 0$.  Then there exists $g\in H^1_\hyp(\T^d)$ with $(g)_{\T^d}=0$ such that
\begin{equation}\label{e.averaging.correct}
    \langle v\cdot\nabla_x g + \b (x) \cdot \nabla_v g \rangle_\gamma(x) = h(x) \, , \qquad \left\| g \right\|_{H^1_\hyp(\T^d)} \leq C \left\| h \right\|_{L^2(\T^d)} \,  .
\end{equation}
\end{lemma}

Suppose that we can solve~\eqref{e.kramers.vproof} under the simplification $\langle g^* \rangle_\gamma = 0$. By Lemma~\ref{l.identifyZ} with $h=\langle g^* \rangle_\gamma$, we can find $g \in H^1_\hyp(\T^d)$ such that $\langle v\cdot\nabla_x g + \b \cdot\nabla_v g \rangle_\gamma=h$. Then, since $\langle-\Delta_v g + v\cdot \nabla_v g\rangle_\gamma=0$, we can solve
$$   -\Delta_v f + v\cdot \nabla_v f + v\cdot\nabla_x f + \b \cdot\nabla_v f = g^* - \left( -\Delta_v g + v\cdot \nabla_v g + v\cdot\nabla_x g + \b \cdot\nabla_v g \right) \, , $$
so $f+g$ solves \eqref{e.kramers.vproof}. We now show that such a $g$ exists, and in the argument below we always work under the assumption that $\langle g^* \rangle_\gamma=0$. We shall occasionally use the notation $g^* \in L^2(\T^d;\dot{H}^{-1}_\gamma)$ to signify that $\langle g^* \rangle_\gamma = 0$.

\begin{proof}[Proof of Lemma~\ref{l.identifyZ}]
Let $\ff \in H^1(\T^d;\R^d)$ be a solution to the problem\footnote{For example, one could argue as in the proof of Lemma~\ref{lem:auxiliaryh1lemma} to produce $\ff$ via the Lax-Milgram theorem satisfying the bound $\left\| \ff \right\|_{H^1(\T^d)} \leq C \left\| h \right\|_{L^2(\T^d)}$.}
$$  \nabla_x \cdot \ff(x) + \b (x) \cdot \ff(x) = h(x) \, .  $$
Let $\xi(s):\mathbb{R}\rightarrow \mathbb{\R}$ be a compactly supported, smooth, odd function of a single variable such that $\int_{\R} \xi(s) s \, ds \neq 0$.  Define $\xi_i: \R^d\rightarrow\R$ by
$$  \xi_i(v) = \xi(v_i) \prod_{i'\neq i} \xi'(v_{i'}) \, ,  $$
so that $\xi_i$ is odd in $v_i$ and even in all other $v_{i'}$ for $i'\neq i$. Under an appropriate normalization, we find that
$$ \int_{\R^d} \partial_{v_j} \xi_i(v) \,d\gamma(v)  = \int_{\R^d} v_j \xi_i(v) \,d\gamma(v) = \delta_{ij} \, ,   $$
since $v_j \xi_i(v) d\gamma(v)$ is odd in $v_i$ unless $i=j$, in which case it is even in all components of $v$. Define
$$  g(x,v) = \ff_i(x) \xi_i(v) \, , $$
where we have used the summation convention over repeated indices. By the smoothness of the $\xi_i$s and the $H^1(\T^d)$ regularity of $\ff$, it is clear that $g\in H^1_\hyp(\T^d)$ with norm controlled by the sum of the respective $H^1$ norms of $\ff$ and $\xi$. Furthermore, $(g)_{\T^d}=0$ since for $1\leq i \leq d$, $\xi_i$ is odd in $v_i$. Now we may compute that
\begin{align}
    \langle B {g} \rangle_\gamma(x)  &= \int_{\R^d} \left( v_j \partial_{x_j} g(x,v) + \b _j(x) \partial_{v_j} g(x,v) \right) \, d\gamma(v) \notag\\
    &= \int_{\R^d} \left( v_j \partial_{x_j} \ff_i(x) \xi_i(v) + \b _j \ff_i(x) \partial_{v_j} \xi_i(v) \right) \, d\gamma(v) \notag\\
    &= \partial_i \ff_i(x) \, + \, \b _i(x) \ff_i(x) \notag\\
    &= h(x) \, . \notag 
\end{align}
The proof is complete. \end{proof}

\begin{proof}[Proof of Proposition~\ref{pro:solvabilitykramers}]
We split the argument into five steps.
\smallskip

\emph{Step 1.} In this step, we show that the functional $\JJ$ is not uniformly equal to $+\infty$ and is uniformly convex on pairs $(f,\jj)$ satisfying the constraint~\eqref{e.kv.constraint}. Let us denote the set of pairs satisfying the constraint by
\begin{equation}\notag
    \mathcal{A}(g^*):= \left\{ (f,\jj) \in H^1_\hyp(\T^d) \times (L^2(\T^d;L^2_\gamma))^d \, : \, \nabla_v^*\jj= g^* - Bf  \, , \,  (f)_{\T^d} = 0  \right\} \, .
\end{equation}
First, since $g^* \in L^2(\T^d;\dot H^{-1}_\gamma)$, there exists $\jj \in L^2(\T^d;L^2_\gamma)$ such that $g^* = A^* \jj$. The pair $(0,\jj)$ belongs to $\A(g^*)$, and $\JJ(0,\jj) < +\infty$.

\smallskip

We now demonstrate uniform convexity. 
Since for every $(f',\jj') \in \A(g^*)$ and $(f,\jj) \in \A(0)$,
\begin{equation}\label{e.parallelogram}
\frac 1 2 \JJ[f' + f,\jj' +\jj] + \frac 1 2\JJ[f'-f,\jj' -\jj] - \JJ[f',\jj'] = \JJ[f,\jj] \, ,
\end{equation}
it suffices to show that there exists $C(d) < \infty$ such that for every $(f,\jj) \in \A(0)$,
\begin{equation}  
\label{e.convex}
 \JJ[f,\jj] \ge C^{-1} \left(\|f\|_{H^1_{\hyp}(\T^d)}^2 + \|\jj\|_{L^2(\T^d;L^2_\gamma)}^2 \right) \, .
\end{equation}
Expanding the square and using that $\nabla_v^*\jj= - Bf$, we find
\begin{equation*}  
\JJ[f,\jj] = \iint_{\T^d \times \R^d} \left(\frac 1 2 |\nabla_v f|^2 + \frac 1 2 |\jj|^2 + f Bf\right) \, dm \, .
\end{equation*}
Moreover, by \eqref{e.ibp.B}, the term $\iint_{\T^d\times\R^d} f Bf \, dm$ vanishes. Finally, from $-Bf=\nabla_v^*\jj$, we have $\langle Bf \rangle_\gamma = 0$, and thus
\begin{align*}  
\|v \cdot \nabla_x f\|_{L^2(\T^d;H^{-1}_\gamma)} 
& 
\le \|Bf\|_{L^2(\T^d;H^{-1}_\gamma)} + \|\b (x)\cdot \nabla_v f\|_{L^2(\T^d;H^{-1}_\gamma)} \\
& 
\le C \|\jj\|_{L^2(\T^d;L^2_\gamma)} + C \|\nabla_v f\|_{L^2(\T^d;L^2_\gamma)} \, .
\end{align*}
Combining the last displays and Theorem~\ref{t.hypoelliptic.poincare} yields \eqref{e.convex}, and thus also the uniform convexity of the functional in \eqref{e.convex.mapping}. 
\newline
\smallskip

\emph{Step 2.} In this step, we rephrase the problem in terms of a perturbed convex minimization problem. Denote by $(f_1,\jj_1)$ the unique minimizing pair of the functional $\JJ$ over $\A(g^*)$. 
We obviously have
\begin{equation*}  
\JJ[f_1,\jj_1] \ge 0 \, .
\end{equation*}

We now show that there is a one-to-one correspondence between solutions $f$ of the Kramers equation and null minimizers $(f,\jj)$ of $\JJ$ satisfying the constraint \eqref{e.kv.constraint}: for every $f \in H^1_{\hyp}(\T^d)$ with $(f)_{\T^d}=0$, we have
\begin{align*}  
f \mbox{ solves \eqref{e.kramers.vproof}} \iff \JJ[f,\jj_1] = 0 \, .
\end{align*}
Indeed, the implication $\implies$ is clear, since if $f$ solves \eqref{e.kramers.vproof}, then
\begin{equation*}  
(f,\nabla_v f) \in \A(g^*) \ \text{ and } \ \JJ[f,\nabla_v f] = 0 \, .
\end{equation*}
Conversely, if $\JJ[f_1,\jj_1] = 0$, then by convexity we have that $f = f_1$ (assuming the mean-zero constraint from \eqref{e.kv.constraint}), and
\begin{equation*}  
\nabla_v f_1 =\jj_1 \qquad \mbox{a.e. in } \T^d\times \R^d \, .
\end{equation*}
Then since $\nabla_v^*\jj_1 = g^* - B f_1$, we recover that $f = f_1$ is indeed a solution of~\eqref{e.kramers.vproof}. In particular, the fact that there is at most one solution to~\eqref{e.kramers.vproof} is clear.

\smallskip

To complete the proof, it thus remains to show that given the unique minimizing pair $(f_1,\jj_1)$, we have that
\begin{equation}  
\label{e.null.min}
\JJ[f_1,\jj_1] \le 0 \, .
\end{equation}
We phrase this as a perturbed convex minimization problem for the functional $G$, which is defined for every $f^* \in L^2(\T^d; {H}^{-1}_\gamma)$ with $(f^*)_{\T^d}=0$ by 
$$  G(f^*) := \inf_{ \substack{f \in H^1_\hyp(\T^d)\\ (f)_{\T^d}=0}} \left( \iint_{\T^d\times\R^d} f f^* \, dm \, + \inf_{\substack{\jj \in L^2(\T^d) \\ (f,\jj) \in \mathcal{A}(f^* + g^*)}}\JJ[f,\jj] \right) \, . $$
To complete the proof, we must show that $G(0) \leq 0$. We decompose the argument into the next three steps. 
\newline
\smallskip

\emph{Step 3.} In this step, we show that $G$ is convex and reduce the problem to showing that the convex dual of $G$ is nonnegative. For every pair~$(f,\jj)$ satisfying $(f,\jj) \in \mathcal{A}(f^* + g^*)$, we have
\begin{equation}
\label{e.constraint.fj}
\nabla_v^*\jj= f^* + g^* - Bf \, , \qquad (f)_{\T^d} = 0 \, , 
\end{equation}
and so utilizing \eqref{e.ibp.B} we find that
\begin{align}  
\iint_{\T^d\times\R^d} ff^* \,dm + \JJ[f,\jj] & = \iint_{\T^d\times\R^d} ff^* \,dm + \iint_{\T^d\times \R^d} \frac 1 2 |\nabla_v f -\jj|^2 \, dm \notag\\
&= \iint_{\T^d\times\R^d} ff^* \,dm + \iint_{\T^d\times\R^d} \frac{1}{2} |\nabla_v f|^2 + \frac{1}{2}|\jj|^2 - f \nabla_v^* \jj\, dm\notag \\
&= \iint_{\T^d\times\R^d} ff^* \,dm + \iint_{\T^d\times\R^d} \frac{1}{2} |\nabla_v f|^2 + \frac{1}{2}|\jj|^2 - f (f^*+g^* - Bf) \, dm \notag \\
&= \iint_{\T^d\times\R^d} \frac{1}{2} |\nabla_v f|^2 + \frac{1}{2}|\jj|^2 - g^* f \, dm . \notag
\end{align}
Taking the infimum over all $(f,\jj)$ satisfying the affine constraint $(f,\jj) \in \mathcal{A}(f^* + g^*)$, we obtain the quantity $G(f^*)$. {We thus infer that $G$ is convex in the variable $f^*$.}
By Lemma~\ref{l.identifyZ}, given $f^* \in L^2(\T^d;H^{-1}_\gamma)$ with vanishing mean, we may find $f_0 \in H^1_\hyp(\T^d)$ such that $\langle B f_0 \rangle_\gamma = \langle f^* + g^* \rangle_\gamma = \langle f^* \rangle_\gamma$.  Then since $\langle f^* + g^* - Bf_0 \rangle_\gamma=0$, we may find $\jj\in (L^2(\T^d;L^2_\gamma))^d$ such that $\nabla_v^* \jj = f^* + g^* - Bf_0$, and we see that the function $G$ is also locally bounded above. These two properties imply that $G$ is lower semi-continuous, see \cite[Lemma~I.2.1 and Corollary~I.2.2]{ET}. We denote by $G^*$ the convex dual
of $G$, defined for every $h \in L^2(\T^d;H^1_\gamma)$ with $(h)_{\T^d}=0$ by
\begin{equation*}  
G^*(h) := \sup_{\substack{f^* \in L^2(\T^d;{H}^{-1}_\gamma) \\ (f^*)_{\T^d}=0}} \left( -G(f^*) + \iint_{\T^d\times \R^d} hf^* \, dm \right) \, ,
\end{equation*}
and by $G^{**}$ the bidual of $G$. 
Since $G$ is lower semi-continuous, we have that $G^{**} = G$ (see \cite[Proposition~I.4.1]{ET}), and in particular,
\begin{equation*}  
G(0) = G^{**}(0) = \sup_{\substack{h \in L^2(\T^d;H^1_\gamma)\\ {(h)_{\T^d} = 0} }} \left( -G^*(h) \right) \, .
\end{equation*}
In order to prove that $G(0) \le 0$, it therefore suffices to show that
\begin{equation}  
\label{e.get.to.nullmin}
\forall h \in L^2(\T^d;H^1_\gamma) \textnormal{ with } (h)_{\T^d} = 0  \, , \qquad G^*(h) \ge 0 \, .
\end{equation}
\smallskip

\emph{Step 4.} In this step we show that 
\begin{equation}  
\label{e.finite.G*}
G^*(h) < +\infty \quad \implies \quad h \in H^1_{ \hyp}(\T^d) \, .
\end{equation}
We rewrite $G^*(h)$ in the form
\begin{equation}  
\label{e.second.G*}
G^*(h) = \sup \left\{ \iint_{\T^d\times \R^d} \left( - \frac 1 2 |\nabla_v f -\jj|^2 - f f^* + h f^* \right) \, dm \right\} ,
\end{equation}
where the supremum is over every $f \in H^1_{\hyp}(\T^d)$, $\jj \in L^2(\T^d;L^2_\gamma)^d$ and $f^* \in L^2(\T^d;{H}^{-1}_\gamma)$ satisfying the constraint \eqref{e.constraint.fj}. 
Given $f$ with $ (f)_{\T^d} = 0$,
we choose to restrict the supremum above to $f^* := Bf$ and $\jj=\jj_0$ the solution of $\nabla_v^*\jj_0 = g^* $. 
Recall that such a $\jj_0 \in L^2(\T^d;L^2_\gamma)^d$ exists since $\langle g^* \rangle_\gamma = 0$. With such choices of $f^*$ and $\jj$, the constraint \eqref{e.constraint.fj} is satisfied, and we obtain that
\begin{multline*}  
G^*(h)
\ge \sup \left\{ \iint_{\T^d \times \R^d} \left( -\frac 1 2 |\nabla_v f -\jj_0|^2 - fBf + h Bf \right) \, dm \, : \, f \in H^1_{\hyp}(\T^d)\, , \,  (f)_{\T^d}= 0 \right\} \, .
\end{multline*}
Recalling that $\iint f Bf \, dm = 0$, and using that $C^\infty_0(\T^d \times \R^d)$ is dense in  $H^1_\hyp(\T^d)$, we deduce that
\begin{multline*}  
G^*(h)
\ge \sup \left\{ \iint_{\T^d \times \R^d} \left( -\frac 1 2 |\nabla_v f -\jj_0|^2 + h Bf \right) \, dm \, : \, f \in C^\infty_c(\T^d\times \R^d) \, , \, (f)_{\T^d}=0 \right\} \, .
\end{multline*}
Then 
the assumption of $G^*(h) < \infty$ implies that
\begin{equation*}  
\sup \left\{ \iint_{\T^d\times \R^d} h Bf\, dm \ : \ f \in C^\infty_c(\T^d\times \R^d) \, , \,  (f)_{\T^d}=0 \, , \,  \|f\|_{L^2(\T^d;H^1_\gamma)} \le 1 \right\} < \infty \, .
\end{equation*}
This then shows that the distribution $Bh$ belongs to the dual of $L^2(\T^d;H^1_\gamma)$, which is $L^2(\T^d;H^{-1}_\gamma)$. Since 
\begin{equation*}  
v \cdot \nabla_x h = Bh - \b \cdot \nabla_v h \, ,
\end{equation*}
the proof of \eqref{e.finite.G*} is complete.
\newline
\smallskip

\emph{Step 5.} 
In place of \eqref{e.get.to.nullmin}, we have left to show that
\begin{equation}
\label{e.get.to.nullmin2}
\forall h \in H^1_{\hyp}(\T^d) \textnormal{ with } (h)_{\T^d}=0 \, , \qquad G^*(h) \ge 0 \, .
\end{equation}
Since $Bf \in L^2(\T^d;H^{-1}_\gamma)$, we may replace $f^*$ by $f^* + Bf$ in the variational formula \eqref{e.second.G*} for $G^*$ to get that
\begin{equation}  
\label{e.G*2}
G^*(h) = \sup \left\{ \iint_{\T^d \times \R^d} \left( -\frac 1 2 |\nabla_v f - \jj|^2 +(h-f)( f^* + Bf) \right) \, dm \right\} \, ,
\end{equation}
where the supremum is now over every $f \in H^1_{\hyp}(\T^d)$, $\jj \in L^2(\T^d;L^2_\gamma)^d$ and $f^* \in L^2(\T^d;H^{-1}_\gamma)$ satisfying the constraint
\begin{equation}
\label{e.constraint.fj2}
\nabla_v^* \jj = f^* + g^* \, , \qquad (f)_{\T^d}=0  \, .
\end{equation}
Setting $f = h$ in \eqref{e.G*2}, 
we find that
\begin{equation*}  
G^*(h) \ge \sup \left\{ \iint_{\T^d \times \R^d} -\frac 1 2 |\nabla_v h - \jj|^2   \, dm \right\} \, ,
\end{equation*}
with the supremum ranging over all $f^* \in L^2(\T^d;H^{-1}_\gamma)$ and $\jj \in L^2(\T^d;L^2_\gamma)^d$ satisfying the constraint \eqref{e.constraint.fj2}. We now simply select $\jj = \nabla_v h \in L^2(\T^d;L^2_\gamma)^d$ and
$$f^* = \nabla_v^* \jj - g^* \in L^2(\T^d;H^{-1}_\gamma)\,, $$
at which point we conclude that $G^*(h) \ge 0$. \end{proof}

\section{Interior regularity of solutions}
\label{s.regularity}

In this subsection, we use energy methods to obtain interior regularity estimates for solutions of the equation
\begin{equation} 
\label{e.pde.f(x)}
-\Delta_v f + v\cdot \nabla_v f + v\cdot \nabla_x f + \b \cdot \nabla_v f + cf = f^* \, . 
\end{equation}
In analogy to the classical theory for uniformly elliptic equations (such as the Laplace or Poisson equations), we obtain an appropriate version of the Caccioppoli inequality, apply it iteratively to obtain~$H^1_\hyp$ estimates on all spatial derivatives of the solution, and then apply the H\"ormander and Sobolev inequalities to obtain pointwise estimates. In particular, we obtain higher regularity estimates---strong enough to imply that our weak solutions are $C^\infty$---without resorting to sophisticated theory for pseudodifferential operators. 

\smallskip

We begin with a version of the Caccioppoli inequality for the equation~\eqref{e.pde.f(x)}.

\begin{lemma}
[Caccioppoli inequality]
\label{l.caccioppoli}
Suppose that $r>0$, $\b \in L^\infty(B_{r};L^\infty(\R^d;\Rd))$, $c \in L^\infty(B_r;L^\infty(\R^d))$, and the pair $(f,f^*) \in L^2(B_r;H^1_\gamma) \times L^2(B_r;H^{-1}_\gamma)$ satisfies the equation
\begin{equation} 
\label{e.cacc.pde}
-\Delta_v f + v\cdot \nabla_v f + v\cdot \nabla_x f + \b \cdot \nabla_v f + cf= f^*  \quad \mbox{in} \ B_r \times \Rd.
\end{equation}
Then $f \in H^1_\hyp(B_{r})$, and there exists $C\left(d,r,\| \b \|_{L^\infty(B_r;L^\infty(\R^d))}, \left\| c \right|_{L^\infty(B_r;L^\infty(\R^d))} \right)<\infty$ such that 
\begin{equation} 
\label{e.caccioppoli}
\left\| \nabla_v f\right\|_{L^2(B_{{\sfrac{r}{2}}};L^2_\gamma)} 
+
\left\| v\cdot \nabla_x f \right\|_{L^2(B_{{\sfrac{r}{2}}};H^{-1}_\gamma)}
\\
\leq 
C \left\| f \right\|_{L^2 (B_{r};L^2_\gamma)} 
+ C \left\| f^* \right\|_{L^2(B_r;H^{-1}_\gamma)}.
\end{equation}
\end{lemma}


\begin{proof}
The PDE~\eqref{e.cacc.pde} guarantees that $f \in L^2(B_r;H^1_\gamma)$ belongs qualitatively to $H^1_\hyp(B_r)$.
\smallskip

\emph{Step 1.} We show that there exists $C(d)<\infty$ such that
\begin{align} 
\label{e.cacc.nablayu}
\left\| \nabla_v f\right\|_{L^2(B_{{\sfrac{r}{2}}};L^2_\gamma)} 
&\leq 
C\left( \frac1r + \left\| \b \right\|_{L^\infty(B_{r}\times\Rd)} + \left\| c \right\|_{L^\infty(B_{r}\times\Rd)}^{\sfrac 12} \right)  \left\| f \right\|_{L^2 (B_{r};L^2_\gamma)}  \\
&\qquad \qquad + C (1+r ) \left\| f^* \right\|_{L^2(B_{r};H^{-1}_\gamma)} \, . \notag
\end{align}
Select a smooth cutoff function $\phi\in C^\infty_c(B_r)$ 
which is compactly supported in $B_{r}$ and satisfies 
$0 \leq \phi \leq 1$ in $B_r$, $\phi\equiv 1$ on~$B_{{\sfrac{r}{2}}}$ 
and $\left\| \nabla \phi \right\|_{L^\infty(B_r)} \leq 8r^{-1}$. 
Testing the equation~\eqref{e.cacc.pde} with $(x,v) \mapsto \phi^2(x) f(x,v)$ yields
\begin{align} 
\label{e.initialtest}
\int_{B_{r} \times\Rd} \phi^2 \left| \nabla_v f \right|^2 \,dx\,d\gamma
&= \int_{B_{r}\times\Rd}\phi^2 f \,  f^* \, dx\,d\gamma  - \int_{B_{r}\times\Rd} \phi^2 f v\cdot \nabla_x f\,dx\,d\gamma \\
&\qquad -
\int_{B_{r} \times \Rd} \phi^2 f \b \cdot \nabla_v f\,dx\,d\gamma - \int_{B_{r} \times \R^d} \phi^2 cf^2 \, dx \, d\gamma \, .
\end{align}
We estimate each of the terms on the right-hand side of~\eqref{e.initialtest} separately. 
\smallskip

For the first term on the right side of~\eqref{e.initialtest}, we use 
\begin{align*} \label{}
\left| \int_{B_{r}\times\Rd}\phi^2 f \,  f^* \, dx\,d\gamma \right|
&  
\leq 
\left\| \phi^2 f \right\|_{L^2(B_{r};H^1_\gamma)} 
\left\| f^* \right\|_{L^2(B_{r};H^{-1}_\gamma)} 
\\ & 
\leq 
\left(
\left\| \phi^2 \nabla_v f \right\|_{L^2(B_{r};L^2_\gamma)} + \left\| f \right\|_{L^2(B_{r};L^2_\gamma)}
\right) \left\| f^* \right\|_{L^2(B_{r};H^{-1}_\gamma)} 
\end{align*}
and then apply Young's inequality to obtain
\begin{align} 
\label{e.cacc.firstterm}
& \left| \int_{B_{r}\times\Rd}\phi^2 f \,  f^* \, dx\,d\gamma \right| 
\\ & \qquad \notag
\leq 
\frac16 \int_{B_{r} \times\Rd} \phi^2 \left| \nabla_v f \right|^2 \,dx\,d\gamma 
+ \frac C{r^2} \int_{B_{r}\times\Rd} f^2 \,dx\,d\gamma 
+ C (1+r^2) \left\| f^* \right\|_{L^2(B_{r};H^{-1}_\gamma)}^2 \, .
\end{align}
For the second term on the right side of~\eqref{e.initialtest}, we integrate by parts to find
\begin{align*} \label{}
- \int_{B_{r}\times\Rd} \phi^2 f v\cdot \nabla_x f\,dx\,d\gamma
&
=
- \int_{B_{r}\times\Rd} \phi^2 v\cdot \nabla_x \left(\frac12 f^2 \right) \,dx\,d\gamma
\\ &
= \int_{B_{r}\times\Rd} \phi \nabla_x \phi \cdot v f^2  \,dx\,d\gamma
\\ & 
= \int_{B_{r}\times\Rd} \phi(x) \nabla_x \phi(x) \cdot  v \exp \Ll( -\frac {|v|^2}{2} \Rr)    f^2(x,v)  \,dx\,dv
\\ & 
= - \int_{B_{r}\times\Rd} 2f \phi \nabla_x \phi \cdot \nabla_vf  \,dx\,d\gamma \, .
\end{align*}
Thus, by Young's inequality,
\begin{align} 
\label{e.cacc.secondterm}
\left| \int_{B_{r}\times\Rd} \phi^2 f v\cdot \nabla_x f\,dx\,d\gamma \right| 
& 
\leq
\frac16 \int_{B_{r} \times\Rd} \phi^2 \left| \nabla_v f \right|^2 \,dx\,d\gamma
+ C \int_{B_{r} \times\Rd} f^2 \left| \nabla_x \phi \right|^2 \,dx\,d\gamma
\\ &  \notag
\leq 
\frac16 \int_{B_{r} \times\Rd} \phi^2 \left| \nabla_v f \right|^2 \,dx\,d\gamma
+ \frac{C}{r^2} \int_{B_{r} \times\Rd} f^2 \, dx \, d\ga \, .
\end{align}
For the third term on the right side of~\eqref{e.initialtest}, we use Young's inequality to obtain
\begin{align} 
\label{e.cacc.thirdterm}
\left| \int_{B_{r} \times \Rd} \phi^2 f \b\cdot \nabla_v f \right| 
&
\leq 
\frac16 \int_{B_{r} \times\Rd} \phi^2 \left| \nabla_v f \right|^2 \,dx\,d\gamma
+ C \int_{B_{r} \times\Rd} \phi^2 f^2 \left| \b \right|^2 \,dx\,d\gamma
\\ &  \notag
\leq \frac16 \int_{B_{r} \times\Rd} \phi^2 \left| \nabla_v f \right|^2 \,dx\,d\gamma
+ C \left\| \b \right\|_{L^\infty(B_{r}\times\Rd)}^2 \int_{B_{r} \times\Rd} f^2 \,dx\,d\gamma \, .
\end{align}
To conclude, we combine~\eqref{e.initialtest}-\eqref{e.cacc.thirdterm} and the obvious estimate on the final term to obtain
\begin{align*} \label{}
\int_{B_{r} \times\Rd} \phi^2 \left| \nabla_v f \right|^2 \,dx\,d\gamma
&
\leq
\frac23 \int_{B_{r} \times\Rd} \phi^2 \left| \nabla_v f \right|^2 \,dx\,d\gamma
+ \frac{C}{r^2} \int_{B_{r} \times\Rd} f^2\, dx \, d\ga
\\ & \qquad 
+ C (1+r^2) \left\| f^* \right\|_{L^2(B_{r};H^{-1}_\gamma)}^2\\
&\qquad + C \left( \left\| \b \right\|_{L^\infty(B_{r}\times\Rd)}^2 + \left\| c \right\|_{L^\infty(B_{r}\times\Rd)} \right) \int_{B_{r}  \times\Rd} f^2 \,dx\,d\gamma.
\end{align*}
The first term on the right may now be reabsorbed on the left. Using that $\phi=1$ on $B_{{\sfrac{r}{2}}}$, we thus obtain~\eqref{e.cacc.nablayu}. The analysis in Step~1 is enough to conclude that $f \in H^1_{\rm hyp}(B_{{\sfrac{r}{2}}})$ and the gradient bound in~\eqref{e.caccioppoli}.

\smallskip

\emph{Step 2.} We show that there exists $C(d)<\infty$ such that
\begin{align} 
\label{e.cacc.ynablaxu}
\left\| v\cdot \nabla_x f \right\|_{L^2(B_{{\sfrac{r}{2}}};H^{-1}_\gamma)}&
\leq 
C 
\left( 1 + \left\| \b \right\|_{L^\infty(B_{{\sfrac{r}{2}}}\times\Rd)} \right) 
\left\| \nabla_v f \right\|_{L^2(B_{{\sfrac{r}{2}}};L^2_\gamma)} \\
&\qquad + C \left\| c \right\|_{L^\infty(B_{{\sfrac{r}{2}}}\times\Rd)} \left\| f \right\|_{L^2(B_{{\sfrac{r}{2}}};L^2_\gamma)}
+ C\left\| f^* \right\|_{L^2(B_{{\sfrac{r}{2}}};H^{-1}_\gamma)}
. \notag
\end{align}
This estimate may be combined with~\eqref{e.cacc.nablayu} to obtain the bound for the second term in~\eqref{e.caccioppoli}, which completes the proof of the lemma. 

\smallskip

To obtain~\eqref{e.cacc.ynablaxu}, we test the equation~\eqref{e.cacc.pde} with $w\in L^2(B_{{\sfrac{r}{2}}};H^1_\gamma)$ to find that
\begin{equation*}
\int_{B_r\times \Rd}  w \,  (v\cdot \nabla_xf)  \, dx \, d\gamma
= - \int_{B_r\times \Rd} \nabla_v f \cdot \left(  \nabla_v w + w \b\right) + \int_{B_r\times \Rd} w f^* \, dx \, d\ga - \int_{B_r \times \R^d} cwf \, dx \, d\gamma \, .
\end{equation*}
We deduce that
\begin{align*} \label{}
& \left|\int_{B_r\times \Rd}  w \,  (v\cdot \nabla_xf)  \, dx \, d\gamma \right|
\\
&\qquad 
\leq
\left\| \nabla_v f \right\|_{L^2(B_{{\sfrac{r}{2}}};L^2_\gamma)}
\left( 
\left\| \nabla_v w \right\|_{L^2(B_{{\sfrac{r}{2}}};L^2_\gamma)} 
+ \left\| \b \right\|_{L^\infty(B_{{\sfrac{r}{2}}}\times\Rd)} 
\left\| w \right\|_{L^2(B_{{\sfrac{r}{2}}};L^2_\gamma)} \right)
\\ 
& \qquad \qquad
+ \left\| w \right\|_{L^2(B_{{\sfrac{r}{2}}};H^1_\gamma)} \left\| f^* \right\|_{L^2(B_{{\sfrac{r}{2}}};H^{-1}_\gamma)} + \| c \|_{L^\infty(B_r \times \R^d)} \| f \|_{L^2(B_r;L^2_\gamma)} \| w \|_{L^2(B_r;L^2_\gamma)}.
\end{align*}
Taking the supremum over $w\in L^2(B_{{\sfrac{r}{2}}};H^1_\gamma)$ with $\left\| w \right\|_{L^2(B_{{\sfrac{r}{2}}};H^1_\gamma)} \leq 1$ yields~\eqref{e.cacc.ynablaxu}.  

\smallskip

The combination of~\eqref{e.cacc.nablayu} and~\eqref{e.cacc.ynablaxu} yields~\eqref{e.caccioppoli}. \end{proof}

In the next lemma, under appropriate regularity conditions on the coefficients, we differentiate the equation~\eqref{e.pde.f(x)} with respect to~$x_i$ to obtain an equation for~$\partial_{x_i}f$, and then apply the previous lemma to obtain an interior~$H^1_{\hyp}$ estimate for~$\partial_{x_i}f$. We need to essentially differentiate the equation a fractional number of times (cf.~\cite{M1,M2}). 

\begin{lemma}[Differentiating in $x$]
\label{l.differentiating}
Fix $r\in (0,\infty)$ and coefficients $\b \in C^{0,1}(B_{r} \times\Rd;\Rd)$, $c\in C^{0,1}(B_r\times\Rd;\R)$. Suppose  that $f^* \in H^1(B_r;H^{-1}_\gamma)$ and $f \in H^1_\hyp(B_r)$ satisfy
\begin{equation} 
\label{e.readyfordiff}
-\Delta_v f + v\cdot \nabla_v f 
+ v\cdot \nabla_x f 
+ \b \cdot \nabla_v f 
+ c f
= f^* \quad \mbox{in} \ B_r \times \Rd
\end{equation}
Then, for each $i\in \{1,\ldots,d\}$, the function $h:= \partial_{x_i} f$ belongs to $H^1_{\hyp}(B_{r'})$ for all $r' \in (0,r)$ and satisfies
\begin{equation} 
\label{e.differentiated}
-\Delta_v h + v\cdot \nabla_v h + v\cdot \nabla_x h + \b  \cdot \nabla_v h + ch = \partial_{x_i} f^* - \partial_{x_i} \mathbf{b} \cdot \nabla_vf - \partial_{x_i} c  \, f\quad \mbox{in} \ B_{r'} \times \Rd.
\end{equation}
Moreover, there exists $C\left(d,r,\left\| \b \right\|_{C^{0,1}(B_r \times \R^d)},\left\| c \right\|_{C^{0,1}(B_r \times \R^d)} \right)<\infty$ such that 
\begin{align}
\label{e.diff.cacc}
\left\| \partial_{x_i} f \right\|_{H^1_\hyp(B_{{\sfrac{r}{2}}})}
&
\leq
C \left\| f  \right\|_{L^2 (B_{r};L^2_\gamma)} 
+ C \left\| f^* \right\|_{H^1(B_r;H^{-1}_\gamma)} \, .
\end{align}
\end{lemma}
\begin{proof}
The argument is by induction on the fractional exponent of differentiability of~$f$ in the spatial variable~$x$. Essentially, we want to differentiate the equation a fractional amount (almost $\sfrac 13$ times), apply the Caccioppoli inequality to the fractional derivative, and then iterate until we have one full spatial derivative.

\smallskip

\emph{Step 1}.  We first prove that, for every $(f,f^*)\in H^1_\hyp(B_r) \times H^1(B_r,H^{-1}_\gamma)$ satisfying~\eqref{e.readyfordiff}, there exists~$C\left(d,r,\left\| \b \right\|_{C^{0,1}(B_r \times \R^d)}, \left\| c \right\|_{C^{0,1}(B_r \times \R^d)} \right)<\infty$ such that $f$ belongs to~$H^1(B_{{\sfrac{r}{2}}};H^1_\ga)$ and satisfies the estimate
\begin{equation}
\label{e.diffonetime}
\left\| \nabla_x f  \right\|_{L^2(B_{{\sfrac{r}{2}}};H^1_\ga)}
\leq
C \left\| f  \right\|_{L^2 (B_{r};L^2_\gamma)} 
+ C \left\| f^* \right\|_{H^1(B_r;H^{-1}_\gamma)} \, .
\end{equation}

Suppose that~$\alpha_0 \in [0,1)$ is such that the following statement is valid: For every $\alpha \in [0,\alpha_0]$, $r>0$, and pair $(f,f^*)\in H^1_\hyp(B_r) \times H^{\alpha}(B_r,H^{-1}_\gamma)$ satisfying~\eqref{e.readyfordiff}, we have $f\in H^\alpha(B_{{\sfrac{r}{2}}};H^1_\gamma)$ and, for~$C\left(d,r,\left\| \b \right\|_{C^{0,1}(B_r \times \R^d)}, \left\| c \right\|_{C^{0,1}(B_r \times \R^d)},\alpha \right)<\infty$, the estimate
\begin{equation}
\label{e.diffalphatimes}
\left\| f \right\|_{H^\alpha(B_{{\sfrac{r}{2}}};H^1_\gamma)}
\leq 
C \left\| f \right\|_{L^2 (B_{r};L^2_\gamma)} 
+ C \left\| f^* \right\|_{H^{\alpha}(B_r;H^{-1}_\gamma)}.
\end{equation}
We argue that the statement is also valid for $\min\left(\alpha_0+\sfrac{1}{3}-\delta,1\right)$ in place of~$\alpha_0$ for all $\delta \in (0,\sfrac{1}{3})$. Note that this statement is clearly valid for $\alpha_0 = 0$ by the Caccioppoli inequality (Lemma~\ref{l.caccioppoli}).

\smallskip

Fix~$\alpha \in [0,\alpha_0]$ and a pair
\begin{equation*}
(f,f^*) \in H^1_\hyp(B_r) \times H^{\alpha}(B_r,H^{-1}_\gamma)
\end{equation*}
satisfying~\eqref{e.readyfordiff}, an index~$i\in \{1,\ldots,d\}$, and a cutoff function~$\phi\in C^\infty_c(B_{{\sfrac{r}{2}}})$ with $0\leq \phi\leq 1$ and $\phi\equiv 1$ on $B_{{\sfrac{r}{4}}}$. Define the functions
\begin{align*}
\tilde{f} &: = \phi^2 f \, , \notag\\
\tilde{f}^* &:= \phi^2 f^* + 2f\phi \, v\cdot \nabla_x \phi \, .
\end{align*}
Observe that $\tilde{f}\in H^1_\hyp(\Rd)$ and $\tilde{f}^* \in H^\alpha(\Rd ; H^{-1}_\gamma)$ are compactly supported in $B_r$ and satisfy
\begin{align*}
    \| \tilde{f} \|_{H^\alpha(\R^d;L^2_\gamma)} &\leq C \| f \|_{H^\alpha(B_r;L^2_\gamma)} \, , \notag \\ 
    \| \tilde{f}^* \|_{H^\alpha(\R^d;H^{-1}_\gamma)} &\leq C \left(\| f^* \|_{H^\alpha(B_r;H^{-1}_\gamma)} + \| f \|_{H^\alpha(B_r;L^2_\gamma)}\right) \, ,
\end{align*}
and the PDE~\eqref{e.pde.f(x)} in $\R^d \times \R^d$.
\smallskip

Next, we mollify. This step ensures that the function qualitatively belongs to good enough spaces to justify the computations (the analogous step in Nirenberg's method is finite differences). Define
\begin{equation}\notag
    \bar{f} = \tilde{f} \ast_x \psi^\varepsilon,
\end{equation}
\begin{equation}\notag
    \bar{f}^* = \tilde{f}^* \ast_x \psi^\varepsilon - [\psi^\varepsilon \ast_x , \b \cdot] \nabla_v \tilde{f} - [\psi^\varepsilon \ast_x, c]  \tilde{f} \, ,
\end{equation}
where $\psi^\varepsilon$ is an appropriate mollification at scale $\varepsilon$. Then $(\bar{f},\bar{f}^*)$ satisfies the PDE~\eqref{e.pde.f(x)} in $\R^d \times \R^d$. We have that
\begin{equation}
    \label{eq:barfest1}
    \left\| (1-\Delta_x)^{\sfrac{\alpha}{2}} \bar{f} \right\|_{L^2(\R^d;L^2_\gamma)} \leq C \left\| \tilde{f} \right\|_{H^{\alpha}(\R^d;L^2_\gamma)}
\end{equation}
\begin{equation}
    \label{eq:barfest2}
    \left\| (1-\Delta_x)^{\sfrac{\alpha}{2}} \bar{f}^* \right\|_{L^2(\R^d;H^{-1}_\gamma)} \leq C \left\| \tilde{f}^* \right\|_{H^{\alpha}(\R^d;H^{-1}_\gamma)} + C \left\| \tilde{f} \right\|_{H^{\alpha}(\R^d;L^2_\gamma)},
\end{equation}
since $[\psi^\varepsilon \ast_x, \b \cdot]$ and $[\psi^\varepsilon \ast_x , c]$ are $H^\alpha(\R^d;H^{-1}_\gamma)$-bounded for all $\alpha \in [0,1]$ while $\b$ and $c$ are Lipschitz.
We apply $(1-\Delta_x)^{\sfrac{\alpha}{2}}$ to the PDE~\eqref{e.pde.f(x)} satisfied by $(\bar{f},\bar{f}^*)$ and denote $f_\alpha = (1-\Delta_x)^{\sfrac{\alpha}{2}} \bar{f}$. We have that $f_\alpha$ satisfies the equation
\begin{align*}
-\Delta_v f_\alpha + &v\cdot \nabla_v f_\alpha + v\cdot \nabla_x f_\alpha + \b \cdot \nabla_v f_\alpha + c f_\alpha\\
&= (1-\Delta)^{\sfrac{\alpha}{2}} \bar{f}^* - [(1-\Delta)^{\sfrac{\alpha}{2}}, \b \cdot] \nabla_v \bar{f} - [(1-\Delta)^{\sfrac{\alpha}{2}}, c]  \bar{f} \,
\end{align*}
in $\Rd \times \Rd$. The Cacciopoli inequality for $f_\alpha \in L^2(\R^d;H^1_\gamma)$, the H{\"o}rmander inequality, and~\eqref{eq:barfest1}-\eqref{eq:barfest2} give
\begin{equation}
    \label{eq:additionalsidebenefit}
    \| f_\alpha \|_{H^{\sfrac{1}{3}-\delta}(\R^d;L^2_\gamma)} + \| f_\alpha \|_{L^2(\R^d;H^1_\gamma)} \leq C \| \tilde{f} \|_{H^{\alpha}(\R^d;L^2_\gamma)} + C \| \tilde{f}^* \|_{H^{\alpha}(\R^d;H^{-1}_\gamma)}
\end{equation}
for all $\delta \in (0,\sfrac{1}{3})$, where $C$ depends on $\delta$. Sending the mollification parameter $\varepsilon \to 0^+$ completes the induction and the proof.
We emphasize that this induction demonstrates that $\p_{x_i} f \in L^2(B_{r'};H^1_\gamma)$ for all $r' < r$, where $f$ is a function satisfying the hypotheses of Lemma~\ref{l.differentiating}. Once this is known, one may plainly differentiate the equation in $\p_{x_i}$ and apply Caccioppoli's inequality to conclude. \end{proof}

\begin{lemma}[Differentiating in $v$]
\label{l.differentiatinginv}
Fix $r\in (0,\infty)$ and coefficients
$\b \in C^{0,1}(B_{r} \times\Rd;\Rd)$, $c\in C^{0,1}(B_r\times\Rd;\R)$. Suppose  that $f^* \in H^1(B_r;L^2_\gamma)$ and $f \in H^1_\hyp(B_r)$ satisfy
\begin{equation} 
\label{e.readyfordiffinv}
-\Delta_v f + v\cdot \nabla_v f 
+ v\cdot \nabla_x f 
+ \b \cdot \nabla_v f 
+ c f
= f^* \quad \mbox{in} \ B_r \times \Rd \, .
\end{equation}
Then, for each $i\in \{1,\ldots,d\}$, the function $h:= \partial_{v_i} f$ belongs to $H^1_{\hyp}(B_{r'})$ for all $r' \in (0,r)$ and satisfies
\begin{equation}
\label{eq:differentiatedinv}
-\Delta h + v\cdot \nabla_vh + v\cdot \nabla_x h +\b\cdot \nabla_vh + (c+1) h
=
h^* \quad \mbox{in} \ B_{r'} \times \Rd, 
\end{equation}
where 
\begin{equation}
\label{eq:hstardef}
h^*:= \partial_{v_i} f^* - \partial_{x_i} f - (\partial_{v_i} \b) \cdot \nabla_vf - (\partial_{v_i}c)f.
\end{equation}
Moreover, there exists $C\left(d,r,\left\| \b \right\|_{C^{0,1}(B_r \times \R^d)},\left\| c \right\|_{C^{0,1}(B_r \times \R^d)} \right)<\infty$ such that 
\begin{align}
\label{e.diff.cacc.v}
\left\| \partial_{v_i} f \right\|_{H^1_\hyp(B_{{\sfrac{r}{2}}})}
&
\leq
C \left\| f  \right\|_{L^2 (B_{r};L^2_\gamma)} 
+ C \left\| f^* \right\|_{H^1(B_r;L^2_\gamma)} \, .
\end{align}
\end{lemma}

\begin{proof}
The standard procedure is to differentiate the equation and apply Caccioppoli's inequality. This introduces a forcing term $h^*$, defined in~\eqref{eq:hstardef}, which contains $\p_{x_i} f$, and this is why we improve the spatial regularity beforehand in Lemma~\ref{l.differentiating}. That is, we already know
\begin{equation}\notag
    \| f \|_{H^1_\hyp(B_{r'})} + \| \p_{x_i}f  \|_{L^2(B_{r'};H^1_\gamma)} \leq C \left( \| f \|_{L^2(B_r;L^2_\gamma)} + \| f^* \|_{H^1(B_r;H^{-1}_\gamma)} \right),
\end{equation}
as in Lemma~\ref{l.differentiating}, where $r' = \sfrac{7r}{8}$.
In addition to this observation, we require
a cut-off and mollification procedure to compensate for the fact that we did not assume \emph{qualitatively} that $\p_{v_i} f \in L^2(B_r;H^1_\gamma)$, which would be enough to make the energy estimate rigorous.
\smallskip

For $\ell \geq 1$, consider a standard cut-off function $\varphi^\ell$ in $v$ at scale $\ell$. Define
\begin{equation}\notag
    \tilde{f} = \varphi^\ell f
\end{equation}
\begin{equation}\notag
    \tilde{f}^* = \varphi^\ell f^* - 2 \nabla_v f \cdot \nabla_v \varphi^\ell  - f \Delta_v \varphi^\ell  + f (v \cdot \nabla_v \varphi^\ell + \b \cdot \nabla_v \varphi^\ell) \, ,
\end{equation}
where we suppress the dependence on $\ell$ in the notation. Then $(\tilde{f},\tilde{f}^*)$ solves~\eqref{e.pde.f(x)} in $B_{r'} \times \R^d$, and it is not difficult to verify that
\begin{equation}\notag
    \| \tilde{f} \|_{L^2(B_{r'} ;H^1_\gamma)} \leq C \| f \|_{L^2(B_{r'} ;H^1_\gamma)} \, ,
\end{equation}
\begin{equation}\notag
    \| \p_{x_i} \tilde{f} \|_{L^2(B_{r'};L^2_\gamma)} \leq \|\p_{x_i} f \|_{L^2(B_{r'};L^2_\gamma)} \, ,
\end{equation}
and
\begin{equation}\notag
    \| \tilde{f}^* \|_{L^2(B_{r'};L^2_\gamma)} \leq C \left( \| f \|_{L^2(B_{r'} ;H^1_\gamma)} + \| f^* \|_{L^2(B_{r'} ;L^2_\gamma)} \right) \, .
\end{equation}
Next, we mollify. Let $\psi^\varepsilon$ be a standard mollification function in $v$ at scale $0 < \varepsilon \ll 1$.  Define
\begin{equation}\notag
    \bar{f} = \psi^{\varepsilon} \ast_v \tilde{f}
\end{equation}
\begin{equation}
    \label{eq:barfdef}
    \bar{f}^* = \psi^{\varepsilon} \ast_v \tilde{f}^* - [\psi^{\varepsilon}_v \ast_v, v \cdot ] (\nabla_v \tilde{f} + \nabla_x \tilde{f}) - ([\psi^{\varepsilon} \ast_v, \b \cdot] \nabla_v \tilde{f}) - [\psi^{\varepsilon} \ast_v, c] \tilde{f} \, ,
\end{equation}
where again we suppress the dependence on $\ell,\varepsilon$ in the notation.
Then $(\bar{f},\bar{f}^*)$ is well defined in $B_{r'} \times \R^d$ and solves~\eqref{e.pde.f(x)} there.

\smallskip

We highlight a few features of the cut-off and mollification procedure. Translations of $L^2_\gamma$ functions may not belong to $L^2_\gamma$, due to the superexponential nature of the weight (compare with exponential weights $e^{-c\la v \ra}$). Hence, mollification is not well behaved on $L^2_\gamma$. The velocity cut-off $\varphi^\ell$ tames this issue. This cut-off has the additional benefit of taming commutators with $v$ which occur naturally in the force term $\bar{f}^*$.

We claim
\begin{equation}
    \label{eq:limsup1}
    \limsup_{\varepsilon \to 0^+} \left\| \bar{f} \right\|_{L^2(B_{r'};H^1_\gamma)} \leq \left\| \tilde{f} \right\|_{L^2(B_{r'} ;H^1_\gamma)} \, ,
\end{equation}
\begin{equation}\notag
    \limsup_{\varepsilon \to 0^+} \left\| \p_{x_i} \bar{f} \right\|_{L^2(B_{r'};L^2_\gamma)} \leq \left\| \p_{x_i} \tilde{f} \right\|_{L^2(B_{r'} ;L^2_\gamma)} \, ,
\end{equation}
and, more subtly,
\begin{equation}
    \label{eq:limsup2}
    \limsup_{\varepsilon \to 0^+} \left\| \bar{f}^* \right\|_{L^2(B_{r'};L^2_\gamma)} \leq C \left( \left\| \tilde{f} \right\|_{L^2(B_{r'} ;H^1_\gamma)} +  \left\| \tilde{f}^* \right\|_{L^2(B_{r'} ;L^2_\gamma)} \right) \, ,
\end{equation}
where~\eqref{eq:limsup1} and~\eqref{eq:limsup2} are for fixed $\ell$. Both estimates in~\eqref{eq:limsup1} are evident due to the support properties of $\tilde{f}$, so we focus on~\eqref{eq:limsup2}. For each fixed $\ell$, we have
\begin{equation}\notag
    \left\| ([\psi^{\varepsilon} \ast_v, \b \cdot] \nabla_v \tilde{f}) + [\psi^{\varepsilon} \ast_v, c] \tilde{f} \right\|_{L^2(B_{r'};L^2_\gamma)} \to 0 \, . 
\end{equation}
as $\varepsilon \to 0^+$.\footnote{One may verify this by writing out the commutator explicitly and using the fundamental theorem of calculus for the difference terms that arise, such as $c(x,v-v')-c(x,v)$ if the mollification variable is $v'$.} Here, we use that the coefficients are Lipschitz and $\tilde{f}$ is compactly supported. It remains to analyze the second term in~\eqref{eq:barfdef}. From the compact support, we may replace $v$ by $\varphi^{2\ell} v$. Then
\begin{equation}
    \label{eq:necessitates1}
    \left\| [\psi^{\varepsilon}_v \ast_v, (\varphi^{2\ell} v) \cdot ]  (\nabla_v  \tilde{f}) \right\|_{L^2(B_{r'} ;L^2_\gamma)} \to 0
\end{equation}
\begin{equation}
    \label{eq:necessitates2}
    \left\| [\psi^{\varepsilon}_v \ast_v, (\varphi^{2\ell} v) \cdot ]  (\nabla_x  \tilde{f}) \right\|_{L^2(B_{r'} ;L^2_\gamma)} \to 0
\end{equation}
as $\varepsilon \to 0^+$ for fixed $\ell$.

\smallskip

Finally, we define $\bar{h} = \p_{v_i} \bar{f}$ and
\begin{equation}
\bar{h}^*:= \partial_{v_i} \bar{f}^* - \partial_{x_i} \bar{f} - (\partial_{v_i} \b) \cdot \nabla_v \bar{f} - (\partial_{v_i}c) \bar{f},
\end{equation}
which solve~\eqref{eq:differentiatedinv} in $B_{r'} \times \R^d$ and satisfy
\begin{equation}\notag
    \| \bar{h} \|_{L^2(B_{r'};L^2_\gamma)} \leq \| \bar{f} \|_{L^2(B_{r'};H^1_\gamma)}
\end{equation}
and
\begin{equation}\notag
    \| \bar{h}^* \|_{L^2(B_{r'};H^{-1}_\gamma)} \leq C \left( \| \bar{f}^* \|_{L^2(B_{r'};L^2_\gamma)} + \| \p_{x_i} \bar{f} \|_{L^2(B_{r'};L^2_\gamma)} + \| \bar{f} \|_{L^2(B_{r'};H^1_\gamma)} \right) \, .
\end{equation}
These, in turn, are estimated by the aforementioned inequalities for $\bar{f}$, $\tilde{f}$, and $f$.
Applying Caccioppoli's inequality and sending $\varepsilon \to 0^+$ and $\ell \to +\infty$ completes the proof. \end{proof}

Theorem~\ref{t.interior.regularity.both} concerning the interior regularity, jointly in the variables~$x$ and~$v$, is obtained by differentiating the equation and repeatedly applying Lemma~\ref{l.differentiating} and Lemma~\ref{l.differentiatinginv}, and we omit the details.

\section{The kinetic Fokker-Planck equation}
\label{s.kin}

In this last section, we study the time-dependent kinetic Fokker-Planck equation
\begin{equation}
\label{e.theotherpde}
\partial_t f - \varepsilon\left( \Delta_v f - v \cdot \nabla_v f \right) + v \cdot \nabla_x f + \b\cdot \nabla_v f = f^* \, .
\end{equation}
The parameter $\varepsilon$ is only relevant for the enhancement estimate, and one may imagine that $\varepsilon=1$ until the final subsection. As with the Kramers equation, we prove a Poincar\'e inequality for bounded domains $V\subset \R\times \R^d$ which are either $C^1$ or cylindrical products $I \times U$ where $I \subset \R$ is a bounded interval and $U$ is a bounded $C^1$ domain, but we consider the initial value problem only for $U=\T^d$. 

\subsection{Function spaces}
We define the function space
\begin{equation}
\label{e.H1kin.def}
H^1_{\kin}(V) := \Ll\{ f \in L^2(V;H^1_\ga) \ : \ \partial_t f + v \cdot \nabla_x f \in L^2(V;H^{-1}_\ga) \Rr\} ,
\end{equation}
equipped with the norm
\begin{equation}
\label{e.H1kin.norm.def}
\|f\|_{H^1_{\kin}(V)} := \|f\|_{L^2(V;H^1_\ga)} + \|\partial_t f + v\cdot \nabla_x f\|_{L^2(V;H^{-1}_\ga)}.
\end{equation}
We denote the unit exterior normal to $V$ by $\n_V \in L^\infty(\partial V;\R^{d+1})$. If $V$ is a $C^1$ domain, then $\n_V(t,x)$ is well defined for every $(t,x) \in \partial V$; if $V$ is of the form $I\times U$, then $\n_V(t,x)$ is well defined unless $(t,x) \in \partial I \times \partial U$, in which case we take the convention that $\n_V(t,x) = 0$. We define the hypoelliptic boundary of $V \subset \R \times \Rd$ as
\begin{equation*}  
\partial_\kin(V) := \Ll\{ ((t,x),v) \in \partial V \times \Rd \ : \ 
\omv \cdot \n_V(t,x) < 0
 \Rr\} ,
\end{equation*}
We denote by $H^1_{\kin,0}(V)$ the closure in $H^1_\kin(V)$ of the set of smooth functions which vanish on $\partial_\kin V$.

\begin{proposition}[Density of smooth functions]
\label{p.density.kin}
Let $V \subset \R \times \Rd$ be a bounded $C^1$ domain or cylindrical product $I \times U$, where $U$ is a  bounded $C^1$ domain. The set $C^\infty_c(\overline{V} \times \Rd)$ of smooth functions with compact support in $\overline{V} \times\Rd$ is dense in $H^1_{\kin}(V)$.
\end{proposition}
\begin{proof}
Mimicking the first step of the proof of Proposition~\ref{p.density}, which only uses that the domain is Lipschitz, we see that we can assume without loss of generality that for every $z \in V$ and $\ep \in (0,1]$, we have
\begin{equation*}  
B((1-\ep)z,\ep) \subset V.
\end{equation*}
Here we use $z$ to denote a generic variable in $\R\times \Rd$; in standard notation, $z=(t,x)$. Let $\zeta_\ep$ be a $(1+d)$-dimensional version of the mollifier defined in \eqref{e.def.zetaep}, and let $f \in H^1_\kin(V)$. We define, for every $\ep \in \Ll( 0,\frac 1 2 \Rr]$, $z \in V$ and $v \in \Rd$,
\begin{equation*}  
f_\ep (z,v) := \int_{\R^{1+d}}f((1-\ep) z + z',v) \zeta_\ep(z') \, dz'.
\end{equation*}
We then show as in Step 2 of the proof of Proposition~\ref{p.density} that $f$ belongs to the closed convex hull of the set $\Ll\{f_\ep \ : \ \ep \in \Ll( 0,\frac 1 2 \Rr] \Rr\}$, and then, as in Step 3 of this proof, that for each $\ep > 0$, we have that $f_\ep$ belongs to the closure of the set $C^\infty_c(\bar V\times \Rd)$.  \end{proof}

\subsection{Functional inequalities for \texorpdfstring{$H^1_\kin$}{H1kin}}
We next show a Poincar\'e inequality for $H^1_\kin(V)$. 
For the sake of generality, we allow for more flexible boundary conditions than in Theorem~\ref{t.hypoelliptic.poincare}, in the spirit of Remark~\ref{r.weaker.boundary.assumption}. 
\begin{proposition}[Poincar\'e inequality]
\label{p.poincare.kin} Let $V \subset \R\times \Rd$ be a bounded $C^1$ domain or a cylindrical product $I \times U$ where $U$ is a bounded $C^1$ domain.

(1) There exists a constant $C(V,d) < \infty$ such that for every $f \in H^1_\kin(V)$, we have
\begin{equation*}  
\Ll\| f - (f)_V \Rr\|_{L^2(V;L^2_\gamma)} \le C \Ll( \|\nabla_v f\|_{L^2(V;L^2_\gamma)} + \|v \cdot \nabla_x f + \partial_t f   \|_{L^2(V;H^{-1}_\ga)} \Rr) \, .
\end{equation*}

(2) Let $W$ be a relatively open subset of $\partial V \times \Rd$. There exists a constant $C(V,W,d) < \infty$ such that for every $f \in C^\infty_c(\bar V \times \Rd)$ that vanishes on $W$, we have
\begin{equation*}  
\Ll\| f  \Rr\|_{L^2(V;L^2_\gamma)} \le C \Ll( \|\nabla_v f\|_{L^2(V;L^2_\gamma)} + \|v \cdot \nabla_x f + \partial_t f   \|_{L^2(V;H^{-1}_\ga)} \Rr) \, .
\end{equation*}
\end{proposition}
\begin{proof}[Proof of Proposition~\ref{p.poincare.kin}]
The proof is similar to that of Theorem~\ref{t.hypoelliptic.poincare}. By Proposition~\ref{p.density.kin}, we can assume that $f \in C^\infty_c(\bar W \times \Rd)$. We start by using the Gaussian Poincar\'e inequality to assert that
\begin{equation*}  
\|f - \la f \ra_\ga \|_{L^2(V;L^2_\ga)} \le \|\nabla_v f\|_{L^2(V;L^2_\ga)} \, .
\end{equation*}
Paralleling the second step of the proof of Theorem~\ref{t.hypoelliptic.poincare}, we then aim to gain control on a negative Sobolev norm of the derivatives of $\la f \ra_\ga$. Here we treat the time and space variables on an equal footing, and thus are interested in controlling $\partial_t \la f \ra_\ga$ and $\nabla \la f \ra_\ga$ in the $H^{-1}(V)$ norm. The precise claim is that there exists $C(d,V) < \infty$ such that for every test function $\phi \in C^\infty_c(V)$ satisfying 
\begin{equation}  
\label{e.bounds.phi.kin}
\|\phi\|_{L^2(V)} + \|\nabla \phi\|_{L^2(V)} + \|\partial_t \phi\|_{L^2(V)} \le 1 \, ,
\end{equation}
we have
\begin{equation}  
\label{e.bracket.H-1.kin}
\Ll|\int_{V} \phi \, \partial_{t} \la f \ra_\ga \Rr| + \sum_{i = 1}^d \Ll|\int_{V} \phi \, \partial_{x_i} \la f \ra_\ga \Rr|  \le C \Ll( \|\nabla_v f\|_{L^2(V;L^2_\ga)} + \|v \cdot \nabla_x f + \partial_t f   \|_{L^2(V;H^{-1}_\ga)} \Rr) \, .
\end{equation}
We start by showing that the first term on the left side of \eqref{e.bracket.H-1.kin}, which refers to the time derivative of 
$\la f \ra_\ga$, is estimated by the right side of \eqref{e.bracket.H-1.kin}. We select a smooth function $\xi_0 \in C^\infty_c(\Rd)$ such that
\begin{equation}  
\label{e.prop.xi0}
\int_{\Rd} \xi_0(v) \, d\gamma(v) = 1 \quad \text{ and } \quad \int_\Rd v \xi_0(v) \, d\gamma(v) = 0 \, ,
\end{equation}
and observe that, using these properties of $\xi_0$, we can write
\begin{align*}  
& \int_{V} \partial_t \phi(t,x) \, \la f \ra_\ga(t,x) \, dt \, dx
\\
& \qquad  = \int_{V \times \Rd} \xi_0(v) \Ll(\partial_t \phi(t,x) + v \cdot \nabla_x \phi(t,x)\Rr) \la f \ra_\ga(t,x) \, dt \, dx \, d\gamma(v) 
\\
& \qquad = \int_{V \times \Rd} \xi_0(v) \Ll(\partial_t + v \cdot \nabla_x \Rr) \phi(t,x) \,  f(t,x,v) \, dt \, dx \, d\gamma(v)
\\
& \qquad \qquad  +  \int_{V \times \Rd} \xi_0(v) \Ll(\partial_t + v \cdot \nabla_x \Rr) \phi(t,x)  \Ll(\la f \ra_\ga(t,x) - f(t,x,v)\Rr) \, dt \, dx \, d\gamma(v) \, .
\end{align*}
Using \eqref{e.bounds.phi.kin} and the fact that $\xi_0$ has compact support, we can bound the second integral above by
\begin{equation*}  
C \|f - \la f \ra_\ga\|_{L^2(V;L^2_\ga)} \le C \|\nabla_v f\|_{L^2(V;L^2_\ga)} \, .
\end{equation*}
By integration by parts, the absolute value of the first integral is equal to
\begin{equation*}  
\Ll|\int_{V \times \Rd} \xi_0(v)\phi(t,x) \Ll( v \cdot \nabla_x + \partial_t \Rr)f(t,x,v) \, dt\, dx \, d\gamma(v)\Rr| \le C \|v\cdot \nabla_x f + \partial_tf\|_{L^2(V;H^{-1}_\ga)} \, .
\end{equation*}
This completes the proof of the estimate in \eqref{e.bracket.H-1.kin} involving the time derivative. To estimate the terms involving the space derivatives, we fix $i \in \{1,\ldots,d\}$ and use a smooth function $\xi_i \in C^\infty_c(\Rd)$ satisfying 
\begin{equation*}  
\int_\Rd \xi_i(v) \, d\ga(v) = 0 \quad \text{ and } \quad \int_\Rd v \xi_i(v) \, d\ga(v) = e_i
\end{equation*}
to get that
\begin{multline*}  
 \int_{V} \partial_{x_i} \phi(t,x) \, \la f \ra_\ga(t,x) \, dt \, dx
\\
  = \int_{V \times \Rd} \xi_i(v) \Ll(v \cdot \nabla_x \phi(t,x) + \partial_t \phi(t,x)\Rr) \la f \ra_\ga(t,x) \, dt \, dx \, d\gamma(v) \, .
\end{multline*}
The rest of the argument is then identical to the estimate involving the time derivative, and thus \eqref{e.bracket.H-1.kin} is proved. The remainder of the proof is then identical to that for Theorem~\ref{t.hypoelliptic.poincare}. Note that we need to invoke {Lemma~\ref{l.est.L2.H-1}}, which allows Lipschitz regularity, for the domain~$V$. \end{proof}

\subsection{The H\"ormander inequality for $H^1_{\kin}$}
For the H\"ormander inequality, we recall the parameter $\varepsilon$ from \eqref{e.theotherpde} and assume that the spatial/temporal domain is $V=[0,\varepsilon^{-\sfrac{1}{3}}]\times\T^d$, although a similar estimate would hold for $V=[0,\varepsilon^{-\sfrac{1}{3}}]\times\R^d$. We emphasize that we have included this particular factor of $\varepsilon$ due to the fact that the {a priori} estimates for \eqref{e.theotherpde} control only $\varepsilon^{\sfrac{1}{2}}\nabla_v f$, and also due to the scaling between the regularity exponent we shall be able to obtain for $\nabla_x f$ and the {a priori} estimate.  This inequality for~$H^1_\kin(V)$ is proved in an almost identical way to the one for~$H^1_\hyp(\T^d)$; the only difference is that the time variable is \emph{not} periodic as is the space variable.  So a bit of care must be taken with the finite differences corresponding to the vector field $\partial_t+v\cdot\nabla_x$. We track the parameter $\varepsilon$ throughout the proof for the purposes of the enhancement estimate later on. The version of~\eqref{e.firstordersplitting} we use here is
\begin{align}
\label{e.firstordersplitting.witht}
    f(t,&x+\eta^3\varepsilon^{\sfrac{1}{2}}x',v) - f(t,x,v) \\
    &= f(t,x+\eta^3 \varepsilon^{\sfrac{1}{2}} x',v) - f(t,x+\eta^3\varepsilon^{\sfrac{1}{2}}x',v-\eta\varepsilon^{\sfrac{1}{2}}x') \notag\\
    &\quad+ f(t,x+\eta^3\varepsilon^{\sfrac{1}{2}}x',v-\eta\varepsilon^{\sfrac{1}{2}}x') - f(t+\eta^2,x+\eta^3\varepsilon^{\sfrac{1}{2}}x'+\eta^2(v-\varepsilon^{\sfrac{1}{2}}\eta x'),v-\eta\varepsilon^{\sfrac{1}{2}} x') \notag\\
    &\quad + f(t+\eta^2,x+\eta^2v,v-\eta\varepsilon^{\sfrac{1}{2}} x') - f(t+\eta^2,x+\eta^2v,v) \notag\\
    &\quad + f(t+\eta^2,x+\eta^2 v,v) - f(t,x,v) \, .\notag
\end{align}

As before, we must define the following Besov spaces based on finite differences in the $\nabla_x$ and $D_t=\partial_t + v\cdot\nabla_x$ directions. The Besov space measuring fractional regularity in the $x$ variable now depends \emph{fundamentally} on $\varepsilon$ and $t$, and so we denote this space $Q_{\nabla_x}^{\sfrac{1}{3},\varepsilon}$.  To lighten the notation, in the context of proofs in which $\varepsilon$ is always fixed, we sometimes shall substitute the notation $\Qx$ instead of the more cumbersome $Q_{\nabla_x}^{\sfrac{1}{3},\varepsilon}$, and similarly for $\Qvxteps$.

\begin{definition}\label{d.Qvxt}
For measurable $u:(0,\varepsilon^{-\sfrac{1}{3}})\times\mathbb{T}^d\times\mathbb{R}^d\rightarrow\mathbb{R}$, we define
\begin{align}
    \left\| u \right\|_{\Qvxteps}^2 &:= \sup_{ 0< \eta \leq \sqrt{\frac{\epsthird}{{2}}}} \frac{1}{\eta^2} \bigg{(}\iiint_{\left(0,\frac{\epsthird}{2}\right)\times\mathbb{R}^d\times\mathbb{T}^d} \left( u(t+\eta^2,x+\eta^2 v, v) - u(t,x,v) \right)^2 \,dx\,d\gamma(v)\,dt\, \notag\\
     &\qquad +  \iiint_{\left(\frac{\epsthird}{2},\epsthird\right)\times\mathbb{R}^d\times\mathbb{T}^d} \left( u(t-\eta^2,x-\eta^2 v, v) - u(t,x,v) \right)^2 \,dx\,d\gamma(v)\,dt \bigg{)} \,. \label{e.Qvxt.norm}
\end{align}
We define
\begin{equation}\label{e.Qxt.norm}
    \left\| u \right\|_{Q_{\nabla_x}^{\sfrac{1}{3},\varepsilon}}^2 := \sup_{\substack{ 0< \eta \leq \sqrt{\frac{\varepsilon^{-\sfrac{1}{3}}}{2}} \\ x'\in \mathbb{S}^{d-1}}} \frac{1}{\eta^{2}} \iiint_{(0,\epsthird)\times\mathbb{R}^d\times\mathbb{T}^d} \left( u(t,x+\epshalf\eta^3 x', v) - u(t,x,v) \right)^2 \,dx\,d\gamma(v)\,dt\, .
\end{equation}
\end{definition}
Notice that the quantity $\varepsilon^{\sfrac{1}{2}}\eta^3$ is of order one if $\eta^2$ takes its maximum value of $\frac{\varepsilon^{-\sfrac{1}{3}}}{2}$.  Then by iterating the finite differences, the norm in \eqref{e.Qxt.norm} is equivalent to one in which the supremum is taken over values of $\eta$ at least as large as the diameter of $\T^d$, at which point the norm is equivalent to one including all positive values of $\eta$.

To streamline the proof of the enhancement estimate later, we assume in the following proposition that $\langle \partial_t u + v\cdot\nabla_x u \rangle_\gamma\equiv 0$ (a condition which will be satisfied in the enhancement context).  Then from Lemma~\ref{l.rep.H-1}, the $L^2_{t,x}H^{-1}_\gamma$ norm of $\partial_t u + v\cdot\nabla_x u$ may be obtained via duality against the gradients (in $v$) of $L^2_{t,x}H^1_\gamma$ functions which have vanishing means $\langle \cdot \rangle_\gamma$.  Thus the inequality \eqref{eq:interpolation:t} does not require the $L^2_{t,x}L^2_\gamma$ norm of $u$ on the right-hand side; one could easily adjust the statement in the case that $\langle \partial_t u + v\cdot\nabla_x u \rangle_\gamma \neq 0$ by including the necessary term.

\begin{lemma}[Interpolation]\label{lem:interpolation:t}
For every $\de > 0$, there exists a constant $C(\delta, d) < \infty$ (not depending on $\ep$) such that for any smooth function $u$ satisfying $\langle \partial_t u + v\cdot\nabla_x u \rangle_\gamma \equiv 0$, 
\begin{align}\label{eq:interpolation:t}
    \left\| u \right\|_{\Qvxteps}^2
    &\leq  {\delta} \left\| u \right\|^2_{Q_{\nabla_x}^{\sfrac{1}{3},\varepsilon}} \\
    &\quad +  C(\delta) \left( {\varepsilon} \left\| \nabla_v u \right\|^2_{L^2\left((0,\epsthird)\times\mathbb{T}^d;L^{2}_\gamma\right)} + {\varepsilon^{-1}} \left\| \partial_t u + v\cdot\nabla_x u \right\|^2_{L^2\left( (0,\epsthird)\times \T^d ; H^{-1}_\gamma \right)} \right) \, . \notag
\end{align}
\end{lemma}
\begin{remark}
The factors of $\varepsilon$ ensure that the right-hand side remains of order $1$ as $\varepsilon\rightarrow 0$ and arise naturally when deriving the a priori estimates for solutions to \eqref{e.theotherpde}; see section~\ref{ss.enhancement} for more details.
\end{remark}
\begin{proof}
The proof is similar for both halves of \eqref{e.Qvxt.norm}, i.e.\ the forward and backward differences, and so we focus on the case of the forward difference.
\smallskip

\textit{Step 1}. Let $\phi \in C^\infty_0((-1,1)^d)$ be a smooth, positive, radial function with unit $L^1$ norm.  For $\zeta>0$, we define $\phi_{\zeta} u (t,x,v)$ by 
\begin{equation}\label{eq:mollifier:t}
\phi_{\zeta} u (t,x,v) = \int_{\mathbb{R}^d} u (t,x+\zeta^3 \epshalf x',v) \phi(x') \, dx' \, .  \notag
\end{equation}
Analogously to Step 1 from the proof of Theorem~\ref{t.hormander}, we have that
\begin{equation}\label{eq:mollifying:estimate:t}
    \left\| \phi_{\zeta} u(t,x,v) - u(t,x,v) \right\|_{L^2 ((0,\epsthird)\times\mathbb{T}^d; L^2_\gamma)}^2 \leq \zeta^{2} \left\| u \right\|^2_{\Qx}\, .
\end{equation}
\smallskip

\textit{Step 2}.  Let 
\begin{equation}
f(\eta) = \left\| u(t+\eta^2,x+\eta^2 v, v) - u(t,x,v) \right\|_{L^2\left((0,\frac{\epsthird}{2})\times\mathbb{T}^d; L^2_\gamma\right)}^2\, . \notag
\end{equation}
We may write that 
\begin{align}
    f(\eta) &\lesssim \left\| \phi_{\delta \eta}u(t+\eta^2,x+\eta^2 v,v) - u(t+\eta^2,x+\eta^2 v,v) \right\|_{L^2\left((0,\frac{\epsthird}{2})\times\mathbb{T}^d; L^2_\gamma\right)}^2 \notag\\
    & \qquad + \left\| \phi_{\delta \eta}u(t+\eta^2,x+\eta^2 v,v) - \phi_{\delta \eta} u(t,x,v) \right\|_{L^2\left((0,\frac{\epsthird}{2})\times\mathbb{T}^d; L^2_\gamma\right)}^2\notag\\
    &\qquad + \left\| \phi_{\delta \eta}u(t,x,v) - u(t,x,v) \right\|_{L^2\left((0,\frac{\epsthird}{2})\times\mathbb{T}^d; L^2_\gamma\right)}^2 \, , \label{moll:split}
\end{align}
where the implicit constant is independent of $\eta$, $\delta$, and $u$. By Step 1 with $\zeta=\delta\eta$, the first and third terms are bounded by
$$ \delta^2 \eta^{2} \left\| u \right\|_{\Qx}^2 \, .$$
\smallskip

\textit{Step 3}. It remains to estimate the second term in~\eqref{moll:split}. For $\eta\in\left(0,\sqrt{\frac{\epsthird}{2}}\right)$ and $0\leq\tau\leq \eta^2$, consider
\begin{equation}\label{eq:Ftau:t}
    F(\tau) = \left\| \phi_{\delta \eta} u(t+\tau,x+\tau v,v) - \phi_{\delta \eta} u(t,x,v) \right\|^2_{L^2\left((0,\frac{\epsthird}{2})\times\mathbb{T}^d; L^2_\gamma\right)} \, .
\end{equation}
The term in question is $F(\eta^2)$. Since $F(0) = 0$, it suffices to estimate $ F'(\tau)$. We have that
\begin{align}
    F'(\tau) &= 2 {\iiint_{\left(0,\frac{\epsthird}{2}\right) \times\mathbb{R}^d\times\mathbb{T}^d}} \left( \phi_{\delta \eta}u(t+\tau,x+\tau v,v) - \phi_{\delta \eta} u(t,x,v) \right) \notag\\
    & \qquad \qquad \qquad \qquad \qquad \cdot D_t \left(\phi_{\delta \eta} u\right) (t+\tau,x+\tau v,v) \,dx\,d\gamma(v)\,dt \notag\\
    &= 2 \iiint_{\left(\tau,\frac{\epsthird}{2}+\tau\right)\times\mathbb{R}^d\times\mathbb{T}^d} \left(\phi_{\delta \eta}u(t,x,v) - \phi_{\delta \eta}u(t-\tau,x-\tau v,v) \right) \notag\\
    & \qquad \qquad \qquad \qquad \qquad \cdot D_t \left(\phi_{\delta \eta} u\right)(t,x,v)\, dx\,d\gamma(v)\,dt\, . \label{eq:Fprimetau}
\end{align}
From $[D_t,\phi_{\delta \eta}]u = \left[ \nabla_v, \phi_{\delta \eta} \right]u=0$, the assumption $\langle \partial_t u + v\cdot\nabla_x u \rangle_\gamma \equiv 0$, and our control of 
$$\left\|\partial_t u + v\cdot\nabla_x u \right\|_{L^2((0,\epsthird)\times\mathbb{T}^d; H^{-1}_\gamma)}\,, $$
we will achieve the desired estimate for $F'(\tau)$ if we can bound
$$ \nabla_v \left( \phi_{\delta \eta}u(t,x,v)-\phi_{\delta \eta}u(t-\tau,x-\tau v,v) \right) $$
in $L^2((\tau,\frac{\epsthird}{2}+\tau)\times\mathbb{T}^d;\ltwog)$. Notice that after obtaining these bounds, we apply the Cauchy-Schwarz inequality with a prefactor of $\varepsilon$ in front of one term and $\varepsilon^{-1}$ in front of the other in order to obtain \eqref{eq:interpolation:t}. The only non-trivial estimate comes when the $\nabla_v$ lands on the $x$ coordinate of the second term, which we may write out as
\begin{align}
    \int_{\mathbb{T}^d}& -\tau \nabla_x u \left(t-\tau, x+(\delta \eta)^3 \epshalf x'- \tau v,v\right) \phi(x')\,dx'\notag\\
    &= - \int_{\mathbb{T}^d} \frac{\tau}{(\delta \eta)^3 \epshalf } \nabla_{x'}u(t-\tau,x+(\delta \eta)^3 \epshalf x'-\tau v,v) \phi(x') \,dx'\notag\\
    &=\int_{\mathbb{T}^d}   \frac{\tau}{(\delta \eta)^3 \epshalf } u(t-\tau,x+(\delta \eta)^3 \epshalf x'-\tau v,v) \nabla_{x'} \phi(x') \,dx' \notag\\
    &=\int_{\mathbb{T}^d}   \frac{\tau}{(\delta \eta)^3 \epshalf } \left( u(t-\tau,x+(\delta \eta)^3 \epshalf x'-\tau v,v) - u(t-\tau,x-\tau v, v) \right) \nabla_{x'} \phi(x') \,dx' \, . \notag
\end{align}
But slight adjustments to the argument from Step 1 show that this is bounded in $L^2\left((\tau,\frac{\epshalf}{2}+\tau)\times\mathbb{T}^d;\ltwog\right)$ by a constant independent of $\delta$ times
$$ \frac{\tau}{(\delta \eta)^3 \epshalf } \delta\eta \left\| u \right\|_{\Qx} \leq \frac{1}{\delta^2 \epshalf } \left\| u \right\|_{\Qx}\, , $$
where here we have used the assumption that $\tau\leq \eta^2$. 
Using the Cauchy-Schwarz and Young inequalities to absorb the negative powers of $\varepsilon$ and $\delta$ with the $L^2_{t,x}H^{-1}_\gamma$ norm concludes the proof. \end{proof}

We may now state and prove the following proposition. As with the interpolation, in the case that $\langle \partial_t u + v\cdot\nabla_x u \rangle_\gamma \neq 0$, one could adjust the statement of the second inequality to include the necessary $L^2_{t,x}L^2_\gamma$ norm of $u$.
\begin{proposition}[H\"ormander inequality]\label{p.hormander.witht}
There exists $C(d) < \infty$ (not depending on $\ep$) such that for every smooth function $u$ satisfying $\langle \partial_t u + v\cdot\nabla_x u \rangle_\gamma \equiv 0$, we have
\begin{align}\label{e.hormander.t.eps}
    \left\| u \right\|_{Q_{\nabla_x}^{\sfrac{1}{3},\varepsilon}} &\leq C \left( \varepsilon^{\sfrac{1}{2}} \left\| \nabla_v u \right\|_{L^2\left((0,{\epsthird})\times\mathbb{T}^d;L^2_\gamma\right)} + \left\| u \right\|_{\Qvxteps} \right) \notag\\
    &\leq C \left( \varepsilon^{\sfrac{1}{2}}\left\| \nabla_v u \right\|_{L^2\left((0,{\epsthird})\times\mathbb{T}^d;L^2_\gamma\right)} + \varepsilon^{-\sfrac{1}{2}} \left\| \partial_t u + v\cdot\nabla_x u \right\|_{L^2\left((0,{\epsthird})\times\mathbb{T}^d;H^{-1}_\gamma\right)}  \right) \, .
\end{align}
\end{proposition}
\begin{proof}[Proof of Proposition~\ref{p.hormander.witht}]
Set $g(t,x,v)=f(t,x,v)\gamma^{\sfrac{1}{2}}(v)$, and choose $\eta^2\in(0,\varepsilon^{-\sfrac{1}{3}}]$ and $x'\in\mathbb{S}^{d-1}$. Then we may write that
\begin{align} \notag
    &\left\| f(t,x+\epshalf\eta^3x',v) - f(t,x,v) \right\|_{L^2\left((0,\frac{\epsthird}{2})\times\mathbb{T}^d;L^2_\gamma\right)} \\
    &\qquad = \left\| g(t,x+\epshalf\eta^3x',v) - g(t,x,v) \right\|_{L^2\left((0,\frac{\epsthird}{2})\times\mathbb{T}^d;L^2(\mathbb{R}^d)\right)} \, , \notag
\end{align}
and
\begin{align}
   g(t,x+\eta^3\varepsilon^{\sfrac{1}{2}}&x',v) - g(t,x,v)\label{e.hormander.string.t} \\
    &= g(t,x+\eta^3 \varepsilon^{\sfrac{1}{2}} x',v) - g(t,x+\eta^3\varepsilon^{\sfrac{1}{2}}x',v-\eta\varepsilon^{\sfrac{1}{2}}x') \notag\\
    &\qquad+ g(t,x+\eta^3\varepsilon^{\sfrac{1}{2}}x',v-\eta\varepsilon^{\sfrac{1}{2}}x') \notag\\
    &\qquad \qquad - g(t+\eta^2,x+\eta^3\varepsilon^{\sfrac{1}{2}}x'+\eta^2(v-\varepsilon^{\sfrac{1}{2}}\eta x'),v-\eta\varepsilon^{\sfrac{1}{2}} x') \notag\\
    &\qquad + g(t+\eta^2,x+\eta^2v,v-\eta\varepsilon^{\sfrac{1}{2}} x') - g(t+\eta^2,x+\eta^2v,v) \notag\\
    &\qquad + g(t+\eta^2,x+\eta^2 v,v) - g(t,x,v) \, . \notag
\end{align}
Dividing by $\eta$, integrating in $L^2\left((0,\frac{\epsthird}{2})\times\mathbb{T}^d; L^2(\mathbb{R}^d)\right)$, and appealing to \eqref{e.weighted.difference} as in the time-independent case yields that
\begin{align}
    &\frac{1}{\eta}\left\| f (t,x+\epshalf\eta^3x',v)-f(t,x,v) \right\|_{L^2\left((0,{\epsthird})\times\mathbb{T}^d;L^2_\gamma\right)} \notag\\
    &\qquad \lesssim \varepsilon^{\sfrac{1}{2}} \left\| \nabla_v f \right\|_{L^2\left((0,{\epsthird})\times\mathbb{T}^d;L^2_\gamma\right)} +  \left\| f \right\|_{\Qvxt}\, .  \notag
\end{align}
For the other half of the time interval, it is easy to rewrite \eqref{e.hormander.string.t} with a backwards difference in the $\partial_t+v\cdot\nabla_x$ direction by first adding $\eta\epshalf x'$ in the $v$ variable and then subtracting $\eta^2$ in the $t$ variable and $\eta^2(v+\epshalf \eta x')$ in the $x$ variable. Arguing as for the forward differences produces an identical estimate. Then using Lemma~\ref{lem:interpolation:t} {and absorbing the $\left\| f \right\|_{\Qx}^2$ factor required to bound $\left\| f \right\|_{\Qvxt}$ from the right-hand side onto the left-hand side gives the result.} 
\end{proof}

\begin{remark}\label{rem.poincare.txv}
From the embedding $\Qx \hookrightarrow {L^2\left((0,{\epsthird})\times\mathbb{T}^d;L^2_\gamma\right)}$ for functions with vanishing $x$-mean $\langle u \rangle(t,v) = \int_{\T^d} u(t,x,v) \,dx $ (see, for example, \cite{abn21}), we obtain the following $\varepsilon$-dependent Poincar\'e inequality:
\begin{equation}\label{eq:poincare:t:eps}
    \left\| u \right\|_{{L^2\left((0,{\epsthird})\times\mathbb{T}^d;L^2_\gamma\right)}} \leq C \varepsilon^{-\sfrac{1}{6}} \left\| u \right\|_{\Qx} \, .
\end{equation}
Note that to obtain this inequality, we have rescaled out the factors of $\varepsilon$ used in the finite differences of the $\Qx$ norm and then appealed to an $\varepsilon$-independent function space embedding.
\end{remark}

\begin{remark}[Regularity in time]
\label{e.sometea}
By an interpolation argument, the result of Proposition~\ref{p.hormander.witht} implies some time regularity for a function $f\in H^1_\kin(V)$ for $V=(0,\epsthird)\times\mathbb{T}^d$. Indeed, by the definition of the norm~$\left\| \cdot\right\|_{H^1_{\kin}}$, we have that 
\begin{equation*}
\left\| f \right\|_{L^2\left((0,{\epsthird})\times\mathbb{T}^d;H^1_\gamma\right)}
\leq 
\left\| f \right\|_{H^1_{\kin}\left((0,\epsthird)\times\T^d\right)} \, .
\end{equation*}
By interpolation and~\eqref{e.hormander.t.eps}, for every $\theta \in [0,1]$ and $\alpha\in \left[0,\tfrac13\right)$,
\begin{equation*}
\left\| f \right\|_{L^2\left((0,\epsthird); H^{\theta\alpha} (\T^d;H^{1-2\theta}_\gamma)\right)}
\leq 
C  \left\| f \right\|_{H^1_{\kin}\left((0,\epsthird)\times\T^d\right)} \, .
\end{equation*}
We also have, by~\eqref{e.hormander.t.eps}, for any $\alpha \in \left[0,\tfrac13\right)$,
\begin{align*}
\lefteqn{
\left\| f \right\|_{H^1\left((0,\epsthird);H^{\alpha-1}(\T^d;H^{-1}_\gamma)\right)}
} \quad & 
\\ &
\leq
\left\| f \right\|_{L^2\left((0,\epsthird);H^{\alpha-1}(\T^d;H^{-1}_\gamma)\right)}
+
\left\| \partial_t f \right\|_{L^2\left((0,\epsthird);H^{\alpha-1}(\T^d;H^{-1}_\gamma)\right)}
\\ & 
\leq 
\left\| f \right\|_{L^2\left((0,\epsthird);L^2(\T^d;H^{-1}_\gamma)\right)}
+
\left\| \partial_t f - v\cdot \nabla_x f \right\|_{L^2\left((0,\epsthird);L^2(\T^d;H^{-1}_\gamma)\right)} \notag\\
&\qquad \qquad +
\left\| v\cdot \nabla_x f \right\|_{L^2\left((0,\epsthird);H^{\alpha-1}(\T^d;H^{-1}_\gamma)\right)}
\\ &
\leq 
C \left\| f \right\|_{H^1_{\kin}\left((0,\epsthird)\times\T^d\right)} \, .
\end{align*}
By interpolation of the previous two displays, we obtain, for any $\theta,\sigma\in [0,1]$ and $\alpha \in \left[0,\tfrac 13\right)$, 
\begin{equation}
\label{e.regularity.sometea}
\left\| f \right\|_{H^\sigma\left((0,\epsthird); H^{\theta\alpha - \sigma(1-\al + \theta\alpha)} (\T^d;H^{1-2(\theta+\sigma-\theta\sigma)}_\gamma)\right)}
\leq
C \left\| f \right\|_{{H^1_{\kin}\left((0,\epsthird)\times\T^d\right)}} \, .
\end{equation}
Each of the constants~$C$ above depends only on~$(\alpha,d)$. Note that all three exponents can be made simultaneously positive, for example taking $\alpha = \theta = \frac14$ and $\sigma = \frac1{32}$ yields 
\begin{equation}
\label{e.allthreepositive}
\left\| f \right\|_{H^{\sfrac{1}{32}}\left((0,\epsthird); H^{\sfrac{1}{32}} (\T^d;H_\gamma^{\sfrac{7}{16}})\right)}
\leq
C \left\| f \right\|_{H^1_{\kin}\left((0,\epsthird)\times\T^d\right)} \, .
\end{equation}
\end{remark}

By~\eqref{e.allthreepositive} and an argument very similar to the proof of Proposition~\ref{p.compactembed}, which we omit, we obtain the following compact embedding statement. 

\begin{proposition}[{Compact embedding of $H^1_{\kin}$ into $L^2$}]
For any bounded $C^1$ domain $V\subseteq \R\times \Rd$ or cylindrical product $I \times U$ where $U$ is a bounded $C^1$ domain,
the inclusion map $H^1_{\kin}(V) \hookrightarrow L^2(V;L^2_\gamma)$ is compact. 
\label{p.compactembed.witht}
\end{proposition}

\subsection{Well-posedness of the Cauchy problem}

\begin{proposition}[Solvability of the kinetic Fokker-Planck equation]
\label{pro:solvabilitykinetic}
Let $T \in (0,+\infty]$,  $f_{\rm in} \in L^2_m$, and $g^*  \in L^2(\T^d \times (0,T);H^{-1}_\gamma)$. Under Assumption~\ref{a.conserv},
there exists a unique solution
\begin{equation}
    f \in C([0,T];L^2_m(\T^d \times \R^d)) \cap H^1_\kin((0,T)\times\T^d)
\end{equation}
to the kinetic Fokker-Planck equation \eqref{e.theotherpde} with initial data $f_{\rm in}$ and forcing term $g^*$.
\end{proposition}

\begin{proof}
 Let $T \in (0,+\infty]$. Let $f_{\rm in} \in L^2_m$ and $g^*  \in L^2\left((0,T);L^2_\sigma(\T^d;H^{-1}_\gamma)\right))$.
A function $g$ solves the kinetic Fokker-Planck equation if and only if $f(t,x,v) = g(t,x,v)e^t$ solves
\begin{equation}
    \label{eq:penalizedPDEkinetic}
    \p_t f  + (v \cdot \nabla_x + \b \cdot \nabla_v)f + f = f^* + \varepsilon (\Delta f - v \cdot \nabla_v f) \, ,
\end{equation}
where $f^* = e^t g^*$.
We solve~\eqref{eq:penalizedPDEkinetic} on $ (0,T)\times \T^d \times \R^d$ by applying Lemma~\ref{lem:lionslaxmilgram} with an appropriate functional setup:
\begin{enumerate}
\item the \emph{test function space}
\begin{equation}
    \label{eq:testfunctionspacekinetic}
    \Phi = C^\infty_0(\T^d \times \R^d \times [0,T))
\end{equation}
with inner product
\begin{equation}
    \label{eq:sameinnerproductkinetic}
    (\phi,\psi) = \int_0^T \int_{\T^d\times\R^d} \nabla_v \phi \cdot \nabla_v \psi \, dm \, dt + \int_0^T  \int_{\T^d\times\R^d} \phi \psi  \, dm \, dt \, ,
\end{equation}
    \item the \emph{solution space}
\begin{equation}\notag
    H = L^2(0,T;L^2_\sigma(\T^d;H^1_\gamma))
\end{equation}
with inner product~\eqref{eq:sameinnerproductkinetic},
\item the \emph{bilinear form}
\begin{align}\notag
    E(h,\phi) &= \varepsilon \int_0^T  \int_{\T^d\times\R^d } \nabla_v h \cdot \nabla_v \phi \, dm \, dt + \int_0^T  \int_{\T^d\times\R^d } h\phi \, dm \, dt \notag\\
    &\qquad - \int_0^T  \int_{\T^d\times\R^d } h (\p_t + v \cdot \nabla_x + \b \cdot \nabla_v) \phi \, dm \, dt \, , \notag
\end{align}
\item and the \emph{linear functional}
\begin{equation}\notag
    L\phi = \int_{\T^d \times \R^d} f_{\rm in} \phi(x,v,0) \, dm + g^*(\phi) \, .
\end{equation}
\end{enumerate}

As before, in the Kramers equation, one may verify that $E$ is continuous~\eqref{eq:bilinearformcont} on $H$ for each fixed $\phi \in \Phi$. We now verify coercivity~\eqref{eq:bilinearformcoer} and mention two essential new features: (i) the initial data $f_{\rm in}$ is built into the linear function $L$, and (ii) test functions $\phi \in \Phi$ vanish at $t = T$ but are not required to vanish at $t = 0$ (which is necessary for them to `detect' the initial data). After integrating by parts in all variables, we have
\begin{align}
    E(\phi,\phi) &= \varepsilon \int_0^T  \int_{\T^d\times\R^d } |\nabla_v \phi|^2 \, dm \, dt+ \int_0^T  \int_{\T^d\times\R^d } |\phi|^2 \, dm \, dt + \frac{1}{2} \int_{\T^d \times \R^d} |\phi(x,v,0) |^2 \, dm \,\notag\\
    &\geq \varepsilon (\psi,\psi)_H \, . \notag
\end{align}
Lemma~\ref{lem:lionslaxmilgram} generates a weak solution $f \in H$ to $E(f,\phi) = L\phi$ for all $\phi \in \Phi$. In particular, choosing $\phi \in \Phi$ that additionally vanish near $t=0$ guarantees that the PDE~\eqref{eq:penalizedPDEkinetic} is satisfied in the sense of distributions. From the PDE itself, we recover that $f \in H^1_\kin(\T^d \times (0,T))$ and, in particular, $f \in C([0,T];L^2_\sigma(\T^d;L^2_\gamma))$; see Lemma~\ref{l.continuityL2}. This is enough regularity to justify that the initial data is $f_{\rm in}$ and the basic energy estimate which guarantees uniqueness. \end{proof}

We do not include a proof of the following statement in this paper, since the argument is a close adaptation of the one of Theorem~\ref{t.interior.regularity.both}. We denote~$V_r := (-r,r)\times B_r$ and by $\nabla_{t,x}$ the full gradient in~$t$ and~$x$, that is, $\nabla_{t,x} = (\partial_t,\nabla_x)$.

\begin{proposition}[Interior regularity, kinetic Fokker-Planck]
\label{p.interior.regularity.both.witht}
Let~$k \in\N$, $r \in (0,\infty)$, and $\b \in C^{k-1,1}(V_{r} \times \Rd;\Rd)$. There exists a constant~$C<\infty$ depending on
\begin{equation*}
\left(d,k,r, \left\| \b \right\|_{C^{k-1,1}(V_{r} \times \Rd;\Rd)} \right) 
\end{equation*}
such that, for every~$f \in H^1_\kin(V_r)$ and~$f^* \in L^2(V_r;H^{-1}_\gamma)$ satisfying
\begin{equation} 
\label{e.intreg.both.witht}
\partial_t f -\Delta_v f + v\cdot \nabla_v f + v\cdot \nabla_x f + \b \cdot \nabla_v f = f^* \quad \mbox{in} \ V_r \times \Rd \, ,
\end{equation}
the following holds: If $\p^\alpha f^* \in L^2(B_r;H^{-1}_\gamma)$ for all multi-indices $\alpha \in \N \times \N^d \times \N^d$ satisfying $|\alpha| \leq k$, then we have $\p^\alpha f\in H^1_{\rm kin}\left(V_{r/2}\right)$ and the estimate
\begin{equation*}
\left\| \p^\alpha f \right\|_{H^1_{\rm kin}\left(V_{r/2}\right)}
\leq 
C\left( 
\left\| f - \left(f\right)_{V_{r}} \right\|_{L^2 (V_{r};L^2_\gamma)}
+
\sum_{|\beta| \leq k}
\left\| \p^\beta \tilde{f}^* \right\|_{L^2(V_r;H^{-1}_\gamma)}
\right).
\end{equation*}
for all multi-indices $\alpha \in \N \times \N^d \times \N^d$ satisfying $|\alpha| \leq k$.
\end{proposition}

\subsection{Exponential decay in time}
\label{ss.decay}
For each bounded interval $I = (I_-,I_+)\subset \R$ and bounded $C^1$ domain $U$, we denote by $H^1_{\kin,||}(I\times U)$ the closure in $H^1_\kin(I\times U)$ of the set of smooth functions which vanish on $I\times \partial_\hyp U$. Note that in particular, we allow the trace of $f \in H^1_{\kin,||}(I\times U)$ on the initial time slice $\{I_-\}\times U$ to be non-zero. In this section, we show that a solution to the kinetic Fokker-Planck equation with zero right-hand side and belonging to $H^1_{\kin,||}(I\times U)$ decays to zero exponentially fast in time. We start with a preliminary classical lemma.
\begin{lemma}[continuity in $L^2$]
Every function in $H^1_{\kin,||}(I\times U)$ can be identified (up to a set of null measure) with an element of $C(\bar I;L^2(U;L^2_\ga))$.
\label{l.continuityL2}
\end{lemma}
\begin{proof}
If $f$ is a smooth function which vanishes on $I\times \partial_\hyp U$, then for every $t \in I$, we have
\begin{multline*}  
\partial_t \|f(t,\cdot)\|_{L^2(U;L^2_\ga)}^2 + \int_{\partial U \times \Rd} f^2(t,x,v) (v \cdot \n_U(x))_+ \, dx \, d\gamma(v) 
\\
= 2 \int_{U\times \Rd} \Ll(f (\partial_t f + v\cdot \nabla_x f)\Rr)(t,x,v) \, dx \, d\ga(v) \, ,
\end{multline*}
where we recall that $(r)_+ := \max(0,r)$.
Since the second integral on the left side is nonnegative, we deduce that  for every $s , t \in I$,
\begin{equation*}  
\Ll|\|f(t,\cdot)\|_{L^2(U;L^2_\ga)}^2 - \|f(s,\cdot)\|_{L^2(U;L^2_\ga)}^2\Rr| \le 2\|f\|_{L^2((s,t)\times U;H^1_\ga)} \, \|\partial_t f + v\cdot \nabla_x f\|_{L^2((s,t)\times U;H^{-1}_\ga)} \, ,
\end{equation*}
and thus, for a constant $C(I) < \infty$,
\begin{equation*}  
\sup_{t \in \bar I} \|f(t,\cdot)\|_{L^2(U;L^2_\ga)}  \le C \|f\|_{H^1_{\kin}(I\times U)} \, .
\end{equation*}
For a general $f \in H^1_{\kin,||}(I\times U)$, there exists a sequence $(f_n)$ of smooth functions which vanish on $I\times \partial_\hyp U$ and such that $f_n$ converges to $f$ in $H^1_\kin(I\times U)$. It follows from the inequality above that $f_n$ converges to $f$ with respect to the $L^\infty(I;L^2(U;L^2_\ga))$ norm; in particular, $f \in C(\bar I;L^2(U;L^2_\ga))$. \end{proof}

We finally turn to the proof of Theorem~\ref{t.hypo.equilibrium}, which is restated in the following proposition. Notice that, by linearity, it suffices to prove the theorem in the case~$f^*=0$ and~$f_\infty=0$.

\begin{proposition}[Exponential decay to equilibrium]
\label{p.decay}
Let $U \subset \Rd$ be a bounded $C^1$ domain and $\b \in L^\infty(U\times \Rd)^d$. There exists $\lambda(\|\b\|_{L^\infty(U\times \Rd)},U,d) > 0$ such that, for every $T \in (0,\infty)$ and $f \in H^1_{\kin,||}((0,T)\times U)$ satisfying 
\begin{equation*}  
\partial_t f -\Delta_v f + v \cdot \nabla_v f + v \cdot \nabla_x f + \b\cdot \nabla_v f = 0 \qquad \text{in } (0,T)\times U\times \Rd \, ,
\end{equation*}
we have, for every $t \in (0,T)$,
\begin{equation*}  
\|f(t,\cdot)\|_{L^2(U;L^2_\ga)} \le 2\exp \Ll( -\lambda t \Rr)  \|f(0,\cdot)\|_{L^2(U;L^2_\ga)}  \, .
\end{equation*}
\end{proposition}
\begin{proof}
For every $0 \le s < t$, we compute
\begin{equation*}  
\frac 1 2 \Ll(\|f(t,\cdot)\|_{L^2(U;L^2_\ga)}^2 - \|f(s,\cdot)\|_{L^2(U;L^2_\ga)}^2 \Rr) \le - \|\nabla_v f\|_{L^2((s,t)\times U;L^2_\ga)}^2 \, .
\end{equation*}
In particular, 
\begin{equation}
\label{e.nonincr}
\mbox{the mapping $t \mapsto \|f(t,\cdot)\|_{L^2(U;L^2_\ga)}$ is nonincreasing.}
\end{equation}
Since 
\begin{equation*}  
-\nabla_v^* \nabla_v f = \partial_t f + v\cdot \nabla_x f + \b\cdot \nabla_v f \, ,
\end{equation*}
we have
\begin{align*}
\lefteqn{
\|\partial_t f + v\cdot \nabla_x f \|_{L^2((s,t)\times U;H^{-1}_\ga)} 
} \qquad & 
\\ &
\leq \|\partial_t f + v\cdot \nabla_x f  + \b\cdot \nabla_v f\|_{L^2((s,t)\times U;H^{-1}_\ga)} + \|  \b\cdot \nabla_v f \|_{L^2((s,t)\times U;H^{-1}_\ga)} 
\\ &
\leq C \|\nabla_v f\|_{L^2((s,t)\times U;L^2_\ga)} \, ,
\end{align*}
and thus
\begin{align}  
\label{e.no.cross.products}
& -\Ll(\|f(t,\cdot)\|_{L^2(U;L^2_\ga)}^2 - \|f(s,\cdot)\|_{L^2(U;L^2_\ga)}^2\Rr) 
\\ & \qquad \notag
\geq 
\frac 1 C \Ll(\|\nabla_v f\|_{L^2((s,t)\times U;L^2_\ga)}^2 + \|\partial_t f + v\cdot \nabla_x f \|_{L^2((s,t)\times U;H^{-1}_\ga)}^2\Rr) \, .
\end{align}
We aim to appeal to Proposition~\ref{p.poincare.kin} to conclude. 
We define
\begin{equation}  
\label{e.incl.V}
V := [0,1]\times U.
\end{equation}
For every $t \geq 0$, we write
\begin{equation*}  
V_t := (t,0) + V = \{(t+s,x) \in \R\times \Rd \ : \ (s,x) \in V\}.
\end{equation*}
Inequality \eqref{e.no.cross.products} implies that, for every $t \ge 0$, 
\begin{align*}  
& -\Ll(\|f(t+1,\cdot)\|_{L^2(U;L^2_\ga)}^2 - \|f(t,\cdot)\|_{L^2(U;L^2_\ga)}^2\Rr)
\\ & \qquad 
\geq 
\frac 1 C \Ll(\|\nabla_v f\|_{L^2(V_t;L^2_\ga)}^2 + \|\partial_t f - v\cdot \nabla_xf \|_{L^2(V_t; H^{-1}_\ga)}^2\Rr).
\end{align*}
Proposition~\ref{p.poincare.kin} yields that
\begin{equation*}  
-\Ll(\|f(t+1,\cdot)\|_{L^2(U;L^2_\ga)}^2 - \|f(t,\cdot)\|_{L^2(U;L^2_\ga)}^2\Rr) \ge \frac 1C \|f\|_{L^2(V_t;L^2_\ga)}^2.
\end{equation*}
Using \eqref{e.nonincr} and \eqref{e.incl.V}, we deduce that 
\begin{equation*}  
-\Ll(\|f(t+1,\cdot)\|_{L^2(U;L^2_\ga)}^2 - \|f(t,\cdot)\|_{L^2(U;L^2_\ga)}^2\Rr) \ge \frac 1C \|f(t+1,\cdot)\|_{L^2(U;L^2_\ga)}^2.
\end{equation*}
This implies exponential decay of the mapping $t \mapsto \|f(t,\cdot)\|_{L^2(U;L^2_\ga)}$ along integer values of~$t$, and we then obtain the conclusion of the proposition by using~\eqref{e.nonincr} once more. \end{proof}

\subsection{Enhancement}\label{ss.enhancement}

Finally, we prove Theorem~\ref{t.enhancement}. Recall that $f$ is assumed to be a solution to 
\begin{equation}\label{e.enhancement.redux}
    \partial_t f + v\cdot\nabla_x f = \varepsilon\left( \Delta_v f - v\cdot\nabla_v f \right)  \textnormal{  in }(0,\infty)\times\T^d\times\R^d \, .
    \end{equation}
\begin{proof}[Proof of Theorem~\ref{t.enhancement}]
After multiplying \eqref{e.enhancement.redux} by $f$ and integrating over $(0,\epsthird)\times\T^d\times\R^d$, we obtain the {a priori} estimates
\begin{align}
    \varepsilon \left\| \nabla_v f \right\|_{L^2(\left(0,\epsthird)\times\T^d\times\R^d\right)}^2 &\leq \left\| f_{\rm in} \right\|^2_{L^2(\T^d;L^2_\gamma)} - \left\| f(\varepsilon^{-\sfrac{1}{3}},\cdot,\cdot) \right\|_{L^2(\T^d;L^2_\gamma)}^2 \notag\\
    \varepsilon^{-1} \left\| \partial_t f + v\cdot\nabla_x f \right\|_{L^2(\left(0,\epsthird)\times\T^d(H^{-1}_\gamma)\right)}^2 &\lesssim \left\| f_{\rm in} \right\|^2_{L^2(\T^d;L^2_\gamma)} - \left\| f(\varepsilon^{-\sfrac{1}{3}},\cdot,\cdot) \right\|_{L^2(\T^d;L^2_\gamma)}^2 \, . \notag
\end{align}
Applying the inequality in \eqref{e.hormander.t.eps} from Proposition~\ref{p.hormander.witht}, which is justified since $\langle \partial_t f + v\cdot\nabla_x f \rangle_\gamma = \varepsilon\langle \Delta_v f - v\cdot\nabla_v f \rangle_\gamma \equiv 0$, we obtain that
\begin{equation}\notag
    \left\| f \right\|_{\Qx}^2 \lesssim \varepsilon \left\| \nabla_v f \right\|^2_{L^2\left((0,\epsthird)\times\T^d;L^2_\gamma\right)} \lesssim \left\| f_{\rm in} \right\|^2_{L^2(\T^d;L^2_\gamma)} - \left\| f(\varepsilon^{-\sfrac{1}{3}},\cdot,\cdot) \right\|_{L^2(\T^d;L^2_\gamma)}^2 \, .
\end{equation}
From \eqref{eq:poincare:t:eps} and the observation that the mean-zero in $x$ condition from \eqref{e.meanzero} is propagated forward in time, we then obtain that
\begin{align*}
    \left\| f \right\|^2_{L^2\left((0,\epsthird)\times\T^d;L^2_\gamma\right)} &\lesssim \varepsilon^{-\sfrac{1}{3}} \left\| f \right\|_{\Qx}^2 \notag\\
    &\lesssim \varepsilon^{\sfrac{2}{3}} \left\| \nabla_v f \right\|_{L^2\left((0,\epsthird)\times\T^d;L^2_\gamma\right)}^2 \notag\\
    &\lesssim \varepsilon^{-\sfrac{1}{3}} \left(\left\| f_{\rm in} \right\|^2_{L^2(\T^d;L^2_\gamma)} - \left\| f(\varepsilon^{-\sfrac{1}{3}},\cdot,\cdot) \right\|_{L^2(\T^d;L^2_\gamma)}^2 \right) \, .
\end{align*}
Translating in time and iterating this procedure yields exponential decay with rate $\exp(-c\varepsilon^{-\sfrac{1}{3}}t)$ along integer multiples of $\epsthird$, similarly to the proof of Proposition~\ref{p.decay}.  Applying \eqref{e.nonincr}, which holds as well for solutions to \eqref{e.enhancement.redux}, we obtain \eqref{e.enhancement}. \end{proof}

\begin{remark}
    \label{rmk:inprincipleonecan}
In principle, one can also incorporate a conservative $\b$ satisfying Assumption~\ref{a.conserv} into the enhancement estimate, since $[\b(x) \cdot \nabla_v , \p_{v_i}] = 0$ for all $i = 1, \hdots, d$.
\end{remark}

\subsection*{Acknowledgments} 
SA and JCM kindly thank Julia Brunken for pointing out their mistake in the first version of this paper.
DA was supported by NSF Postdoctoral Fellowship  Grant No. 2002023 and Simons Foundation Grant No. 816048.
SA was partially supported by NSF Grants DMS-1700329 and DMS-2000200. 
JCM was partially supported by the ANR grants LSD
(ANR-15-CE40-0020-03) and Malin (ANR-16-CE93-0003) and by the NSF grant DMS-1954357. SA and JCM were partially supported by a grant from the NYU-PSL Global Alliance. MN was partially supported by the NSF under Grant No. DMS-1928930 while participating in a program hosted by the Mathematical Sciences Research Institute during the spring 2021 semester, and by the NSF under Grant No. DMS-1926686 while a member at the Institute for Advanced Study.
\small
\bibliographystyle{abbrv}
\bibliography{hypoelliptic}

\end{document}